\documentclass[11pt]{article}
\usepackage[a4paper]{geometry}
\usepackage{amsthm}
\usepackage{amsmath}
\usepackage{amssymb}
\usepackage{amsfonts}
\usepackage{graphicx}
\usepackage{color}
\usepackage[bf,SL,BF]{subfigure}
\usepackage{url}
\usepackage{epstopdf}
\usepackage{pdfsync}
\usepackage[colorlinks]{hyperref}
\hypersetup{
    linkcolor=blue,
}

\usepackage{epsfig,amsbsy,graphicx,multirow}

\usepackage[all]{xy}

 \numberwithin{equation}{section}

\def\V{\mathfrak{V}}
\def\ve{\varepsilon}
\def\N{\mathbb{N}}

\def\ve{\varepsilon}

\def\R{\mathbb{R}}
\def\cT{\mathcal{T}}
\def\cN{\mathcal{N}}

\def\Z{\mathcal{Z}}

\def\Z{\mathcal{Z}}

\def\L{L}

\def\H{\mathcal{H}}
\def\T{\mathcal{T}}

\def\<{\big\langle}
\def\>{\big\rangle}

\def\diiv{\operatorname{div}}
\def\dist{\operatorname{dist}}

\def\uin{u^{\rm in}}
\def\dx{\,{\rm d}x}

\def\pp{\partial}
\def\loc{{\rm loc}}

\definecolor{red}{rgb}{0.9, 0, 0}

\newtheorem{Theorem}{Theorem}[section]
\newtheorem{Proposition}[Theorem]{Proposition}
\newtheorem{Lemma}[Theorem]{Lemma}
\newtheorem{Corollary}[Theorem]{Corollary}
\newtheorem{Remark}[Theorem]{Remark}

\begin{document}
\title{Polyharmonic homogenization, rough polyharmonic splines and sparse super-localization.}

\date{\today}

\author{Houman Owhadi\footnote{Corresponding author. California Institute of Technology, Computing \& Mathematical Sciences , MC 9-94 Pasadena, CA 91125, owhadi@caltech.edu}, \quad  Lei Zhang\footnote{Shanghai Jiaotong University,
Institute of Natural Sciences and Department of Mathematics, Key Laboratory of Scientific and Engineering Computing (Shanghai Jiao Tong University ), Ministry of Education, 800 Dongchuan Road, 200240, Shanghai, China, lzhang2012@sjtu.edu.cn}, \quad Leonid Berlyand \footnote{Pennsylvania State University, Department of Mathematics, University Park, PA, 16802, berlyand@math.psu.edu}}

\maketitle

\begin{abstract}
We introduce a new variational method for the numerical homogenization of divergence form elliptic, parabolic and hyperbolic
  equations with arbitrary rough ($L^\infty$) coefficients.  Our method does not rely on concepts of ergodicity or scale-separation but on compactness properties of the solution space and a new variational approach to homogenization.
The approximation space is generated by an interpolation basis (over scattered points forming a mesh of resolution $H$) minimizing the $L^2$ norm of the source terms; its (pre-)computation involves minimizing $\mathcal{O}(H^{-d})$ quadratic (cell) problems on (super-)localized sub-domains of size $\mathcal{O}(H \ln (1/ H))$. The resulting localized linear systems remain sparse and banded. The resulting interpolation basis functions are biharmonic for $d\leq 3$, and polyharmonic for $d\geq 4$, for the operator $-\diiv(a\nabla \cdot)$ and can be seen as a generalization of polyharmonic splines to differential operators with arbitrary rough coefficients. The accuracy of the method  ($\mathcal{O}(H)$ in energy norm and independent from aspect ratios of the mesh formed by the scattered points) is established via the introduction of a new class of higher-order Poincar\'{e} inequalities. The method bypasses  (pre-)computations on the full domain and naturally generalizes to time dependent problems, it also provides a natural solution to the inverse problem of recovering the solution of a divergence form elliptic equation from a finite number of point measurements.
\end{abstract}

\tableofcontents

\section{Introduction}\label{sec:intro}
Consider the partial differential equation
\begin{equation}\label{eqn:scalar}
\begin{cases}
    -\diiv \Big(a(x)  \nabla u(x)\Big)=g(x) \quad  x \in \Omega;\, g \in L^2(\Omega), \;   a(x)=\{a_{ij} \in L^{\infty}(\Omega)\}\\
    u=0 \quad \text{on}\quad \partial \Omega,
    \end{cases}
\end{equation}
where $\Omega$ is a bounded subset of $\R^d$ with piecewise Lipschitz boundary  which satisfies an interior cone condition
and $a$ is symmetric, uniformly elliptic on $\Omega$ and with entries in $L^\infty(\Omega)$. It follows
that the eigenvalues of $a$ are uniformly bounded from below and
above by two strictly positive constants, denoted by $\lambda_{\min}(a)$ and
$\lambda_{\max}(a)$.  More precisely, for all $\xi \in \R^d$ and almost all $x\in \Omega$,
\begin{equation}\label{skjdhskhdkh3e}
\lambda_{\min}(a)|\xi|^2 \leq \xi^T a(x) \xi \leq \lambda_{\max}(a)
|\xi|^2.
\end{equation}
In this paper, as in  \cite{OwZh:2007a, BeOw:2010, OwZh:2011}, we are interested in the homogenization of \eqref{eqn:scalar} (and its parabolic and hyperbolic analogues in Section \ref{sec:timedependent}), but not in the classical sense, i.e., that  of asymptotic analysis \cite{BeLiPa78} or that of $G$ or $H$-convergence \cite{Mur78,Spagnolo:1968,Gio75} in which one considers a sequence of operators $-\diiv(a_\epsilon \nabla \cdot)$ and seeks to characterize the limit of solutions. We are interested in the homogenization of \eqref{eqn:scalar} in the sense of ``numerical homogenization,'' i.e., that of the approximation of the solution space of \eqref{eqn:scalar}  by a finite-dimensional space.
This approximation is not based on concepts of scale separation and/or of ergodicity but on strong compactness properties, i.e., the fact that the unit ball of the solution space is strongly compactly embedded into $\H^1_0(\Omega)$ if source terms ($g$) are in $L^2$ (see \cite[Appendix B]{ OwZh:2011}).

In other words if \eqref{eqn:scalar} is ``excited'' by  source terms $g\in L^2$ then the solution space can be approximated by a finite dimensional sub-space (of $\H^1_0(\Omega)$) and the question becomes ``how to approximate'' this solution space.
The approximate solution space introduced in this paper is constructed as follows (see Section \ref{sec:varforminterbas}):\\
 (1) Select a finite subset of points  $(x_i)_{i\in \cN}$ in $\Omega$ with $\cN:=\{1,\ldots,N\}$.\\
  (2) Identify the nodal interpolation basis $(\phi_i)_{i \in \cN}$ as minimizers of  $\int_{\Omega}\big|\diiv(a\nabla \phi)\big|^2$ subject to $\phi_i(x_j)=\delta_{i,j}$ (where $\delta_{i,j}$ is the Kronecker delta
  defined by $\delta_{i,i}=1$ and $\delta_{i,j}=0$ for $i\not=j$).\\ (3) Identify the approximate solution space as the interpolation space generated by the elements $\phi_i$.

  Note that we construct here  an  interpolation basis with $N$ elements,  which interpolates a given function  using its values at  the nodes $(x_i)_{i\in \cN}$ (analogously to e.g.,  the classical interpolation polynomials of degree $N-1$).

Our motivation in identifying the interpolation elements through the minimization step (2) lies in observation that the origin of the compactness of the solution space lies in the higher integrability of $g$. In Section \ref{sec:varforminterbas} we will show that one of the properties of the interpolation basis $\phi_i$ is that $\sum_{i} w_i \phi_i$ minimizes $\int_{\Omega}\big|\diiv(a\nabla w)\big|^2$ over all functions $w$ interpolating the points $x_i$
(i.e. such that $w(x_i)=w_i$). The error of the interpolation and the accuracy of the resulting finite element method is established in Section \ref{sec:Poincare} through the introduction of a new class of higher  order Poincar\'{e} inequalities.

The accuracy of our method (as a function of the locations of the interpolation points) will depend only on the mesh norm of the discrete set $(x_i)_{i\in \cN}$, i.e., the constant $H$ defined by
\begin{equation}\label{def:H}
H:=\sup_{x\in \Omega} \min_{i\in \cN}|x-x_i|
\end{equation}
 Our method is meshless in its formulation and error analysis and the density of the interpolation points $(x_i)_{i\in \cN}$ can be increased near points of the domain  where high accuracy is required.
In the fully discrete formulation of our method (Section \ref{sec:numexp}), $a$ will be chosen to be piecewise constant on a tessellation $\cT_h$ (of $\Omega$) of resolution $h\ll H$ and the points $x_i$ will be a subset of interior nodes of $\cT_h$. More precisely, in numerical applications where $a$ is represented via constant values on a fine mesh $\cT_h$ of resolution $h$, $H$ will be the resolution of a coarse mesh $\cT_H$  having interior nodes $x_i$ (the nodes of $\cT_H$ will be nodes of $\cT_h$ but triangles of $\cT_H$ may not be unions of triangles of $\cT_h$) and the accuracy of our method will depend only on $H$ (i.e., the accuracy will be independent from the regularity of $\cT_H$ and the aspect ratios of its elements (tetrahedra for $d=3$, triangles for $d=2$)).

In Section \ref{subsec:inverseproblem} we show that our method also provides a natural solution to the (inverse) problem of recovering $u$ from the (partial) measurements of its (nodal) values at the sites $x_i$.
Although we restrict, in this paper, our attention to $d\leq 3$, we show, in Section \ref{subsec:gnerealizationdgeq4}, how our method can be generalized to
$d\geq 4$ by requiring that $g \in L^{2m}(\Omega)$ (for  $(d-1)/2 \leq  m \leq d/2$) and obtaining the elements $\phi_i$ as minimizers of $\int_{\Omega}\big|\diiv(a\nabla \phi)\big|^{2m}$. The resulting elements are $2m$-harmonic in the sense that they satisfy $\big(\diiv(a\nabla\cdot)\big)^{2m}\phi_i=0$ in $\Omega\setminus \{x_j\,:\, j\in \cN\}$.
In that sense our method can be seen as a polyharmonic formulation of homogenization and a generalization of polyharmonic splines \cite{Harder:1972, Duchon:1976} to PDEs with rough coefficients (see Section \ref{sec:rps}, the basis elements $\phi_i$ can be seen as ``rough polyharmonic splines'').
In Section \ref{sec:Localization} we show how the computation of the basis elements $\phi_i$ can be localized to sub-domains $\Omega_i \subset \Omega$ and obtain a posteriori error estimates. Those a posteriori error estimates show that if the localized elements $\phi_i^{\loc}$ decay at an exponential rate away from $x_i$ (which we observe in our numerical experiments in Section \ref{sec:numexp}) then the accuracy of the method remains $\mathcal{O}(H)$  in $\H^1$-norm
for sub-domains  $\Omega_i$ of size $\mathcal{O}\big(H\ln (1/H)\big)$, which we refer to as the {\it super-localization property}.  We defer the proof and statement of a priori error estimates to a sequel work.
We give a short review of numerical homogenization in Section \ref{sec:numhom} and we discuss  connections with classical homogenization in Subsection \ref{subsec:connectionsclasshom}.

  We discuss and compare the accuracy and computational cost of our method in Subsection \ref{subsec:compcost}.
  The linear systems to be solved for the identification of the (super-)localized basis $\phi_i^{\loc}$ are sparse and banded (by virtue of this property, the computational cost of our method is minimal).


\section{Variational formulation and properties of the interpolation basis.}\label{sec:varforminterbas}

\subsection{Identification of the interpolation basis.}\label{sec:formulation}

Let $V$ be the set of functions  $u\in \H^1_0(\Omega)$ such that $\diiv(a \nabla u) \in L^2(\Omega)$.
Observe that $V$ is the solution space of \eqref{eqn:scalar} for $g\in L^2(\Omega)$.
By \cite{St:1965, GiTr:1983}, solutions of \eqref{eqn:scalar} are H\"{o}lder continuous functions over $\Omega$ provided that their source terms $g$ belong to  $L^{d/2+\ve}(\Omega)$ for some $\ve>0$. It follows that if the dimension of the physical space is less than or equal to three ($d\leq 3$)
then elements of $V$ are H\"{o}lder continuous functions over $\Omega$. For the sake of simplicity, we will restrict our presentation to $d\leq 3$
and show in Section \ref{subsec:gnerealizationdgeq4} how the method and results of this paper can be generalized to $d\geq 4$.

Let $\|.\|_V$ be the norm on $V$ defined by
\begin{equation}
\|u\|_V:=\|\diiv(a \nabla u)\|_{L^{2}(\Omega)}.
\end{equation}
It is easy to check that $(V,\|.\|_V)$ is a complete linear space; in particular, it is closed under the norm $\|.\|_V$.

Let ($x_i$, $i=1,\dots,N$) be a finite subset of points in $\Omega$. Write
$\cN:=\{1,\ldots,N\}$
and define $H$ as in \eqref{def:H} (in practical applications, $H>0$ will be a parameter determined by available computing power and desired accuracy).
For each $i\in \{1,\dots,N\}$ define
\begin{equation}
V_i:=\{\phi\in V \mid \phi(x_i)=1 \text{ and } \phi(x_j)=0,  \text{ for } j\in \{1,\ldots, N\} \text{ such that } j\neq i\}
\end{equation}
and
consider
the following optimization problem over $V_i$:
\begin{equation}\label{eqn:convexopt}
\begin{cases}
\text{Minimize }  \int_{\Omega}\big|\diiv (a\nabla \phi)\big|^2\\
\text{Subject to }  \phi\in V_i.
\end{cases}
\end{equation}
We will now show that \eqref{eqn:convexopt} is a well posed (strictly convex) quadratic optimization problem and we will identify its unique minimizer $\phi_i$.
\begin{Remark}
Observe that $\phi_i$ depends only on $a$, $\Omega$ and the locations of the points $(x_i)_{i\in \cN}$.
Note that we do not require the distribution of those points be uniform or regular  (in particular, the density of points $x_i$ can be adapted to the structure of $a$  or increased in locations where higher accuracy is desired).
In particular, in  the construction of $\phi_i$, we do not  use  any tessellation, we just use the set of points $(x_i)_{i\in \cN}$.
\end{Remark}
Let $G(x,y)$ be the  Green's function of \eqref{eqn:scalar}. Recall that $G$ is symmetric and that it satisfies (for $y\in \Omega$)
\begin{equation}\label{eqn:scalarGreens}
\begin{cases}
    -\diiv \Big(a(x)  \nabla G(x,y)\Big)=\delta(x-y) \quad  &x \in \Omega;\\
    G(x,y)=0 \quad \text{for}\quad &x\in \partial \Omega,
    \end{cases}
\end{equation}
where $\delta(x-y)$ is the Dirac distribution centered at $y$. For  $i, j\in \{1,\ldots, N\}$ define
\begin{equation}\label{eq:theta:edf}
\Theta_{i,j}:=\int_{\Omega} G(x,x_i) G(x,x_j)\,dx.
\end{equation}
\begin{Lemma}\label{lem:Theta}
The integrals \eqref{eq:theta:edf} are finite, well defined and for all $i, j\in \{1,\ldots, N\}$, $\Theta_{i,j}$ is bounded by a constant
depending only on $\lambda_{\min}(a)$, $\lambda_{\max}(a)$ and $\Omega$.
The $N\times N$ matrix $\Theta$ is symmetric positive  definite. Furthermore for all $l\in \R^{N}$,
\begin{equation}
l^T \Theta l=\|v\|_{L^2(\Omega)}^2
\end{equation}
where $v$ is the solution of
\begin{equation}\label{eqn:scalarGreensTheta}
\begin{cases}
    -\diiv \Big(a(x)  \nabla v(x)\Big)=\sum_{j=1}^{N} l_j \delta(x-x_j)& \text{ for }  x \in \Omega;\\
   v(x)=0 & \text{ for }x\in \partial \Omega.
    \end{cases}
\end{equation}
\end{Lemma}
\begin{proof}
By Theorem 1.1. of \cite{GrWi:1982}, for $d\geq 3$, the Green's function is bounded by
\begin{equation}
G(x, y) \leq K_{d}|x-y|^{2-d}
\end{equation}
where $K_d$ depends only on $d,\lambda_{\min}(a)$ and $\lambda_{\max}(a)$. For $d=2$, we have (see for instance \cite{Taylor:2012} and references therein)
\begin{equation}
G(x, y) \leq K_2 \big(1+\ln \frac{\text{diam}(\Omega)}{|x-y|} \big)
\end{equation}
where $K_2$ depends only on $\lambda_{\min}(a)$, $\lambda_{\max}(a)$ and $\Omega$. For $d=1$, the Green function is trivially bounded by a constant
depending only on $\lambda_{\min}(a)$, $\lambda_{\max}(a)$ and $\Omega$. It follows from these observations that for $d=1,2,3$
\begin{equation}\label{eq:boundonGtheta}
\int_{\Omega} \big(G(x_i, y)\big)^2 \,dy\leq K
\end{equation}
where the constant $K$ is finite and depends only on $\lambda_{\min}(a)$, $\lambda_{\max}(a)$ and $\Omega$. We deduce that the integrals in \eqref{eq:theta:edf} (and hence $\Theta$) are well defined.

Let $l\in \R^{N}$, write
\begin{equation}
v(x):=\sum_{i=1}^N G(x,x_i) l_i
\end{equation}
Observe that $v$ is the solution of \eqref{eqn:scalarGreensTheta} and that $\|v\|_{L^2(\Omega)}^2=l^T \Theta l$. Then, it follows that $\Theta$ is symmetric positive definite. Indeed if $\Theta$ is not positive definite, then there would exist a non zero vector $l\in \R^N$ such that $\Theta l=0$. This would imply $\|v\|_{L^2(\Omega)}=0$ which is a contradiction since the equation $-\diiv \Big(a(x)  \nabla v(x)\Big)=\sum_{j=1}^{N} l_j \delta(x-x_j)$ has a non zero solution.
\end{proof}
Let $P$ be the $N\times N$ symmetric definite positive matrix defined as the inverse of $\Theta$, i.e.
\begin{equation}\label{eq:defP}
P:=\Theta^{-1}
\end{equation}
Note that Lemma \ref{lem:Theta} (i.e. the well-posedness and invertibility of $\Theta$) guarantees the existence of $P$.

Define
\begin{equation}\label{eq:detftau}
\tau(x,y):=\int_{\Omega}G(x,z)G(z,y)\,dz.
\end{equation}
Note that $\tau$ is symmetric and (see proof of Lemma \ref{lem:Theta}) that for each $y\in \Omega$, $x\rightarrow \tau(x,y)$ is a well defined element of $V$ (in particular, it is H\"{o}lder continuous on $\Omega$). Note also that for $i,j\in \cN$ we have $\tau(x_i,x_j)=\Theta_{i,j}$ and that $\tau$ is the fundamental solution of $\big(\diiv(a\nabla \cdot)\big)^2$ in the sense that for $y\in \Omega$
\begin{equation}\label{eq:odhuippokoh}
\begin{cases}
\diiv\Big(a\nabla \big(\diiv (a\nabla \tau(x,y))\big)\Big)=\delta(x-y) &\text{ for } x\in\Omega\\
\tau(x,y)=\diiv\big(a\nabla \tau(x,y)\big)=0 &\text{ for } x\in \partial \Omega.
\end{cases}
\end{equation}

\begin{Proposition}\label{prop:nonempty}
$V_i$ is a non-empty closed affine subspace of $V$.
Problem \eqref{eqn:convexopt} is a strictly convex optimization problem over $V_i$. The unique minimizer of \eqref{eqn:convexopt} is
\begin{equation}\label{eq:phiidef}
\phi_i(x):=\sum_{j=1}^{N} P_{i,j}\, \tau(x,x_j).
\end{equation}
\end{Proposition}
\begin{Remark}
 It is important to note that in practical (numerical) applications each element
$\phi_i$ would be obtained by solving the quadratic optimization problem \eqref{eqn:convexopt} rather than through the representation formula \eqref{eq:phiidef}.
\end{Remark}
\begin{proof}
Let us first prove that $\phi_i \in V_i$.
First observe that for all $j \in \{1,\ldots,N\}$,
\begin{equation}\label{eq:ugwygwugye}
-\diiv \Big(a(x)  \nabla \tau(x,x_j)\Big)= G(x,x_j)
\end{equation}
and $\tau(x,x_j)=0$ on $\partial \Omega$. Noting that $\big\|\diiv (a(x)  \nabla \tau(x,x_j))\big\|_{L^2(\Omega)}^2=\Theta_{j,j}$ we deduce from
Lemma \ref{lem:Theta} that $\tau(x,x_j) \in V$. We conclude from \eqref{eq:phiidef} that $\phi_i \in V$.
Now observing  that
\begin{equation}
\Theta_{i,j}=\tau(x_i,x_j)
\end{equation}
we deduce (using the fact that $P \cdot \Theta$ is the $N\times N$ identity matrix) that
\begin{equation}
\phi_i(x_j)= (P \cdot \Theta)_{i,j}= \delta_{i,j}.
\end{equation}
We conclude that $\phi_i \in V_i$ which implies that   $V_i$ is non empty (it is  easy to check that it is a closed affine sub-space of $V$).

Now let us prove that problem \eqref{eqn:convexopt} is a strictly convex optimization problem over $V_i$. Let $v, w \in V_i$ such that $v\not=w$. Write for $\lambda \in [0,1]$,
\begin{equation}\label{eqn:convseept}
f(\lambda):=  \int_{\Omega}\Big|\diiv \big(a\nabla (v+\lambda (w-v))\big)\Big|^2.
\end{equation}
and we need to show that $f(\lambda)$ is a strictly convex function.
Observing that
\begin{equation}\label{eqn:convseepshiuhiwt}
f(\lambda)=  \int_{\Omega}\big|\diiv (a\nabla v)\big|^2+ 2\lambda \int_{\Omega}(\diiv (a\nabla v))(\diiv (a\nabla (w-v))) +\lambda^2\int_{\Omega}\big|\diiv (a\nabla (v-w))\big|^2
\end{equation}
and noting that $\int_{\Omega}\big|\diiv (a\nabla (v-w))\big|^2>0$ (otherwise one would have $v=w$ given that $v$ and $w$ are both equal to zero on $\partial \Omega$) we deduce that $f$ is strictly convex in $\lambda$. We conclude that (see, for example, \cite[pp. 35, Proposition 1.2]{EkTe:1987})
that Problem \eqref{eqn:convexopt} is a strictly convex optimization problem over $V_i$ and that it admits a unique minimizer in $V_i$.

Let us now prove that $\phi_i$ is the minimizer of \eqref{eqn:convexopt}. Let $v\in V_i$ with $v\not=\phi_i$. Write $I(v):=\int_{\Omega}|\diiv (a\nabla v)|^2$. We have
\begin{equation}\label{eqn:convseeedeept}
I(v)=  I(\phi_i)+I(v-\phi_i)+2 J(\phi_i,v-\phi_i)
\end{equation}
with
\begin{equation}\label{eqn:conv26eedeept}
J(\phi_i,v-\phi_i) =  \int_{\Omega}(\diiv (a\nabla \phi_i))(\diiv (a\nabla (v-\phi_i))).
\end{equation}
Note that (as before) $I(v-\phi_i)>0$ if $v\not=\phi_i$ and using \eqref{eq:ugwygwugye} and \eqref{eq:phiidef} we obtain that
\begin{equation}\label{eq:phijkkijoijiidef}
-\diiv (a(x)\nabla \phi_i(x)):=\sum_{j=1}^{N} P_{i,j}\, G(x,x_j)
\end{equation}
which leads to
\begin{equation}\label{eqn:conv26eedieieept}
J(\phi_i,v-\phi_i)=  \sum_{j=1}^{N} P_{i,j}\, \big(v(x_j)-\phi_i(x_j)\big)=0
\end{equation}
where in the last equality we have used $v(x_j)-\phi_i(x_j)=0$ for all $j\in \{1,\ldots,N\}$. It follows that $I(v)>I(\phi_i)$ which implies that
$\phi_i$ is the minimizer of \eqref{eqn:convexopt}.
\end{proof}

\subsection{Variational properties of the interpolation basis}
Write  $V_0$ the subset of $V$ defined by
\begin{equation}\label{eq:hiuhiu33}
 V_0:=\big\{v\in V: v(x_i)=0, \forall i\in \{1,\ldots,N\}\big\}
\end{equation}
For $u, v\in V$, define the (scalar) product $\<\cdot,\cdot\>$ by
\begin{equation}
 \<u, v\> := \int_\Omega (\diiv (a \nabla u))(\diiv (a \nabla v))
\end{equation}

The following theorem shows that the functions $\phi_i$ generate a linear space interpolating elements of $V$ with minimal $\<.,.\>$-norm (see \eqref{eqn:minimizingproperty}).

\begin{Theorem}\label{lem:minimizingproperty}
It holds true that
\begin{itemize}
\item The basis $\phi_i$ is orthorgonal to $V_0$ with respect to the product $\<\cdot,\cdot\>$, i.e.
\begin{equation}\label{eqn:orthogonality}
 \<\phi_i,v\> =0, \quad \forall i\in \{1,\ldots,N\} \text{ and } \forall v\in V_0
\end{equation}
\item $\sum_i w_i \phi_i$ is the unique minimizer of
\begin{equation}\label{eqn:minimizingproperty}
 \<w,w\>=\int_\Omega \big(\diiv (a \nabla w)\big)^2
\end{equation}
over all $w\in V$ such that $w(x_i)=w_i$.
\item For all $i\in \{1,\ldots,N\}$ and for all $v\in V$,
\begin{equation}\label{eqn:orthogonalidoddoty}
 \<\phi_i,v\> =\sum_{j=1}^{N} P_{i,j}\,  v(x_j)
\end{equation}
\item For all $i,j\in \{1,\ldots,N\}$,
\begin{equation}\label{eqn:orthogjuoddoty}
 \<\phi_i,\phi_j\> =P_{i,j}
\end{equation}
\end{itemize}
\end{Theorem}

\begin{proof}
Equation \eqref{eqn:orthogonality}  trivially follows from \eqref{eqn:orthogonalidoddoty}.
\eqref{eqn:orthogonalidoddoty} is a direct consequence of \eqref{eq:phijkkijoijiidef}.
\eqref{eqn:orthogonalidoddoty} leads to $ \<\phi_i,v\>=0$ for $v\in V_0$.
To prove that $w=\sum_i w_i \phi_i$ is the unique minimizer of \eqref{eqn:minimizingproperty}
over  $W:=\big\{w\in V\mid w(x_i)=w_i \text{ for all }i\in \{1,\ldots,N\}\big\}$ we simply observe that if $w'$ is another element of $W$ then $w'-w$ belongs to $V_0$ and for $w'\not=w$ we have (from \eqref{eqn:orthogonality}, using $<w,w'-w>=0$) that
\begin{equation}\label{eqn:minimizingpropertybis}
 \<w',w'\>=\<w,w\>+\<w'-w,w'-w\>> \<w,w\>\,.
\end{equation}
For \eqref{eqn:orthogjuoddoty}, using \eqref{eqn:orthogonalidoddoty} and writing $I_N$ the $N\times N$ identity matrix we have
\begin{equation}\label{eqn:orthogonalidoddojty}
 \<\phi_i,\phi_j\> =\sum_{k=1}^{N} P_{ik}\,  \phi_j(x_k)=\big(P\cdot I_N\big)_{i,j}=P_{i,j}.
\end{equation}
\end{proof}

\begin{Remark}
It is easy to check that the orthogonality \eqref{eqn:orthogonality} implies that $\phi_i$ solves
\begin{equation}\label{eq:kuhishiu3e}
\begin{cases}
\diiv\Big(a\nabla \big(\diiv (a\nabla \phi_i)\big)\Big)=0 &\text{ on }\Omega\setminus \{x_j\mid j\in \cN\}\\
\phi_i=\diiv (a\nabla \phi_i)=0 &\text{ on }\partial\Omega \\
\phi_i(x_j)=\delta_{i,j}.
\end{cases}
\end{equation}
However, \eqref{eq:kuhishiu3e} alone cannot be used to uniquely identify $\phi_i$ as \eqref{eq:phiidef}. This is due to the fact that the solution of \eqref{eq:kuhishiu3e} may not be unique. As a counter-example consider $\Omega=(-1,1)$, $x_1=0$, $\cN=\{1\}$ and $a=I_d$. For $\alpha \in \R$ define
\begin{equation}
\begin{cases}
\psi_\alpha(x):=\frac{\alpha-1}{2}x^3 +\frac{3}{2}(\alpha-1)x^2+\alpha x+1 &\text{ on }[-1,0]\\
\psi_\alpha(x):=\frac{\alpha+1}{2}x^3 -\frac{3}{2}(\alpha+1)x^2+\alpha x+1 &\text{ on }[0,1]
\end{cases}
\end{equation}
then for any $\alpha \in \R$,  $\psi_\alpha$ is a solution \eqref{eq:kuhishiu3e}. Using the variational formulation of $\phi_i$ it is possible to show that to enforce the uniqueness of the solution of \eqref{eq:kuhishiu3e} we also need to require that (see Proposition \ref{prop:idiuuhueh3}), for some $c\in \R^N$,
\begin{equation}\label{eq:duu3iuhee}
\diiv\Big(a\nabla \big(\diiv (a\nabla \phi_i)\big)\Big)=\sum_{j=1}^N c_j \delta(x-x_j) \text{ on }\Omega
\end{equation}
which in dimension one is equivalent to the continuity of $\diiv(a\nabla \phi_i)$ across the coarse nodes $(x_j)_{j\in \cN}$. Note that \eqref{eq:duu3iuhee}
 implies that for all $v\in V_0$
\begin{equation}
\int_{\Omega} \diiv(a\nabla \phi_i) \diiv(a\nabla v)=\int_{\Omega}v\,\diiv\Big(a\nabla \big(\diiv (a\nabla \phi_i)\big)\Big)=0.
\end{equation}
\end{Remark}

\section{From a Higher Order Poincar\'{e} Inequality to the accuracy of the interpolation basis}\label{sec:Poincare}
In this section we introduce a new higher order  Poincar\'{e} inequality and derive the accuracy of the interpolation space computed in \eqref{eqn:convexopt}.
This new class can be thought of as a generalization of Sobolev inequalities for functions with scattered zeros (see \cite{Matveev:1992a, Matveev:1994, Matveev:1996, Matveev:1997, Narcowich:2005}) to operators with rough coefficients.

 \subsection{A Higher Order Poincar\'{e} Inequality}
 We will first present the new Poincar\'{e} inequality (Theorem\ref{lem:Poincare}). Analogous inequalities when $a=I_d$ can be found in Theorem 1.1 of \cite{Narcowich:2005} and in \cite{Duchon:1978} but the presence of the $L^\infty$ matrix $a$ renders the proof much more difficult and requires a new strategy based on the following lemma.

\begin{Lemma}\label{lem:murat}
Let $d\leq 3$ and $B_1$ be the open ball of center $0$ and radius $1$. There exists a finite strictly positive constant $C_{d,\lambda_{\min}(a),\lambda_{\max}(a)}$ such that for all $v\in \H^1(B_1)$
 such that $\diiv(a\nabla v)\in L^2(B_1)$ it holds true that
\begin{equation}\label{eqn:PoincareNarcowichmurat}
\|v-v(0)\|_{L^2(B_1)}^2 \leq C_{d,\lambda_{\min}(a),\lambda_{\max}(a)}  \Big(\|\nabla v\|_{L^2(B_1)}^2+\big\|\diiv (a\nabla v)\big\|_{L^2(B_1)}^2\Big)
\end{equation}
\end{Lemma}
\begin{proof}
The proof is per absurdum. Note that since $d\leq 3$ the assumptions $v\in \H^1(B_1)$ and $\diiv(a\nabla v)\in L^2(B_1)$ imply the H\"{o}lder continuity of $v$ in $B_1$.
Assume that \eqref{eqn:PoincareNarcowichmurat} does not hold. Then there exists a sequence $v_n$ and a sequence $a_n'$ whose maximum and minimum eigenvalues are uniformly bounded by $\lambda_{\min}(a)$ and $\lambda_{\max}(a)$ (we need to introduce that sequence because we want the constant in \eqref{eqn:PoincareNarcowichmurat} to depend only $d,\lambda_{\min}(a),\lambda_{\max}(a)$)
such that
\begin{equation}\label{eqn:PoincareNarcowichmurat1}
\|v_n-v_n(0)\|_{L^2(B_1)}^2 > n  \Big(\|\nabla v_n\|_{L^2(B_1)}^2+\big\|\diiv (a_n'\nabla v_n)\big\|_{L^2(B_1)}^2\Big)
\end{equation}
Letting $w_n=\frac{v_n-v_n(0)}{\|v_n-v_n(0)\|_{L^2(B_1)}}$ we obtain that $w_n(0)=0$, $\|w_n\|_{L^2(B_1)}=1$ and
\begin{equation}\label{eqn:PoincareNarcowichmurat2}
\|\nabla w_n\|_{L^2(B_1)}^2+\big\|\diiv (a_n'\nabla w_n)\big\|_{L^2(B_1)}^2 < \frac{1}{n}
\end{equation}
Since
\begin{equation}\label{eqn:PoincareNarcowichmurat3}
\| w_n\|_{\H^1(B_1)}< 1+\frac{1}{n}\leq 2
\end{equation}
it follows that there exists a subsequence $w_{n_j}$ and a $w\in \H^1(B_1)$ such that $w_{n_j}\rightharpoonup w$ weakly in $\H^1(B_1)$ and
$\nabla w_{n_j}\rightharpoonup \nabla w$ weakly in $L^2(B_1)$.  Using $\|\nabla w_n\|_{L^2(B_1)}\leq 1/n$ we deduce that $\nabla w=0$ which implies that $w$ is a constant in $B_1$. Since by the Rellich–-Kondrachov theorem the embedding $\H^1(B_1)\subset L^2(B_1)$ is compact it follows from \eqref{eqn:PoincareNarcowichmurat3} that $w_{n_j}\rightarrow w$ strongly in $L^2(B_1)$ which (using $\|w_n\|_{L^2(B_1)}=1$) implies that $\|w\|_{L^2(B_1)}=1$.
Now \eqref{eqn:PoincareNarcowichmurat3}  together with the fact that $\big\|\diiv (a_n'\nabla w_n)\big\|_{L^2(B_1)}^2$ is uniformly bounded and that $d\leq 3$ implies that $w_n$ is uniformly H\"{o}lder continuous on $B(0,\frac{1}{2})$ (see for instance \cite{Stampaccia:1964}). This implies that $w$ is continuous in $B(0,\frac{1}{2})$ and that $w(0)=0$. This contradicts the fact that $w$ is a constant in $B_1$ with $\|w\|_{L^2(B_1)}=1$.
\end{proof}

We also need the following lemma.
\begin{Lemma}\label{lem:lemdesuipoinc}
Let $f\in \H^1_0(\Omega)$ and define $\Omega_{2H}$ as the set of points in $\Omega$ at distance at most $2H$ from $\partial \Omega$. There exists a constant $C_{d,\Omega}$ such that
\begin{equation}\label{eq:gj2gjhee}
\|f\|_{L^2(\Omega_{2H})}\leq C_{d,\Omega} H \|\nabla f\|_{L^2(\Omega_{2H})}
\end{equation}
\end{Lemma}
\begin{proof}
The proof is analogous to Poincar\'{e} inequality. We  observe that $u=0$ on $\partial \Omega$. Recalling  that  $\Omega$ is piecewise Lipschitz continuous we
cover $\partial \Omega \cup \Omega_{2H}$ with a collection of charts and express $u(x)$ as an integral from the boundary (where $u=0$) to $x$ of $\nabla u(y)$ along a path of length bounded by a multiple of $H$.
\end{proof}

\begin{Theorem}\label{lem:Poincare}
Let $f\in V_0$ (and $d\leq 3$). It holds true that
\begin{equation}\label{eqn:Poincare}
\|\nabla f\|_{L^2(\Omega)}\leq C H \big\|\diiv(a\nabla f)\big\|_{L^2(\Omega)}
\end{equation}
where the constant $C$ depends only $d,\lambda_{\min}(a)$, $\lambda_{\max}(a)$ and $\Omega$.
\end{Theorem}

\begin{proof}
We have by integration by parts and the Cauchy-Schwartz inequality,
\begin{equation}
\begin{split}
\int_{\Omega}(\nabla f)^T a\nabla f&= - \int_\Omega f\big(\diiv (a \nabla f)\big)\\
&\leq  \| f\|_{L^2(\Omega)} \big\|\diiv(a\nabla f)\big\|_{L^2(\Omega)}.
\end{split}
\end{equation}
Therefore, using Young's inequality we obtain that for any $\alpha>0$,
\begin{equation}\label{eq:kdgi3igee}
\begin{split}
\int_{\Omega}(\nabla f)^T a\nabla f
&\leq \frac{1}{2\alpha H^2} \| f\|_{L^2(\Omega)}^2 +\frac{\alpha H^2}{2}\big\|\diiv(a\nabla f)\big\|_{L^2(\Omega)}^2.
\end{split}
\end{equation}

Now it follows from the definition \eqref{def:H} that there exists and index set $\cN' \subset \cN$ such that:
(i) for all $i\in \cN'$, $B(x_i,H) \subset \Omega$ (where $B(x_i,H)$ is the ball of center $x_i$ and radius $R$) (ii)
 $\cup_{ i\in \cN'} B(x_i,H) \cup \Omega_{2H}=\Omega $  (iii) and that any $x\in \Omega$ is contained in at most $K$ balls of $(B(x_i,H))_{i\in \cN'}$ where $K$ is a finite number depending only on the dimension $d$. It follows from (ii) that
\begin{equation}\label{eq:kdkhduheerd}
 \| f\|_{L^2(\Omega)}^2\leq \sum_{i\in \cN' } \|f\|_{L^2(B(x_i,H))}^2+\|f\|_{\Omega_{2H}}^2
\end{equation}
Observing that $f(x_i)=0$ we obtain from  Lemma \ref{lem:murat} and scaling that for each $i\in \cN'$
\begin{equation}\label{eq:kwwehduheerd}
\|f\|_{L^2(B(x_i,H))}^2\leq C_{d,\lambda_{\min}(a),\lambda_{\max}(a)}  \Big(H^2 \|\nabla f\|_{L^2(B(x_i,H))}^2+H^4 \big\|\diiv (a\nabla f)\big\|_{L^2(B(x_i,H))}^2\Big)
\end{equation}
Therefore \eqref{eq:kwwehduheerd} and \eqref{eq:gj2gjhee} imply that
\begin{equation}\label{eq:sgjhsejwh2gjhee}
\begin{split}
\|f\|_{L^2(\Omega)}^2 \leq &C_{d,\Omega} H^2 \|\nabla f\|_{L^2(\Omega_{2H})}^2 \\+C_{d,\lambda_{\min}(a),\lambda_{\max}(a)}&  \sum_{i\in \cN'}
 \Big(H^2 \|\nabla f\|_{L^2(B(x_i,H))}^2+H^4 \big\|\diiv (a\nabla f)\big\|_{L^2(B(x_i,H))}^2\Big).
 \end{split}
\end{equation}
Using (iii) we deduce that
\begin{equation}\label{eq:sgjhssedejsehee}
\begin{split}
\|f\|_{L^2(\Omega)}^2 \leq & C_{d,\Omega, \lambda_{\min}(a),\lambda_{\max}(a)} K \\&\Big( H^2 \|\nabla f\|_{L^2(\Omega)}^2 + H^4 \big\|\diiv (a\nabla f)\big\|_{L^2(\Omega)}^2\Big)
\end{split}
\end{equation}
Combining \eqref{eq:sgjhssedejsehee} with \eqref{eq:kdgi3igee} we obtain that
\begin{equation}\label{eq:kdgiiiiyiuy3igee}
\begin{split}
\int_{\Omega}(\nabla f)^T a\nabla f
&\leq  \frac{\alpha H^2}{2}\big\|\diiv(a\nabla f)\big\|_{L^2(\Omega)}^2\\&+
\frac{1}{\alpha} C_{d,\Omega, \lambda_{\min}(a),\lambda_{\max}(a)} \Big(\|\nabla f\|_{L^2(\Omega)}^2 + H^2 \big\|\diiv (a\nabla f)\big\|_{L^2(\Omega)}^2\Big).
\end{split}
\end{equation}
which concludes the proof by taking $\alpha=  \frac{1}{2\lambda_{\min}(a)}C_{d,\Omega, \lambda_{\min}(a),\lambda_{\max}(a)}$.
\end{proof}

 \subsection{Accuracy of the interpolation basis}
Now let us use \eqref{eqn:Poincare} to estimate the interpolation error.

\begin{Corollary}\label{lem:interpolationerror}
Let $u\in V$ be the solution of \eqref{eqn:scalar} and  $\uin(x):=\sum_{i=1}^N u(x_i)\phi_i(x)$. It holds true that
\begin{equation}
 \|u-\uin\|_{\H^1_0(\Omega)}\leq C H \|g\|_{L^2(\Omega)}
\end{equation}
where the constant $C$ depends only on $\lambda_{\min}(a)$, $\lambda_{\max}(a)$ and $\Omega$.
\end{Corollary}
\begin{proof}
Observing that $u-\uin\in V_0(\Omega)$ we deduce from Theorem \ref{lem:Poincare} that
\begin{align}
 \int_\Omega |\nabla (u-\uin)|^2  & \leq C H^2 \int_\Omega \big|g-\sum_i u(x_i)\diiv (a \nabla \phi_i)\big|^2\,dx \\
				     \label{eqn:311}& \leq C H^2 \Big(\int_\Omega g^2 \,dx +\int_\Omega \big(\sum_i u(x_i)\diiv (a \nabla \phi_i)\big)^2 \,dx\Big)\\
				     \label{eqn:312}& \leq C H^2 \Big(\int_\Omega g^2 \,dx + \int_\Omega \big(\diiv (a \nabla u)\big)^2 \,dx\Big)\\
				     & \leq C H^2 \|g\|_{L^2(\Omega)}^2
\end{align}
Note that in the third inequality, we have used the Lemma \ref{lem:minimizingproperty}, which states that the linear combination of $\phi_i$ minimizes $\|\cdot\|_V$ among all the functions in $V(\Omega)$ sharing the same  values at the nodes $\{x_i\}$. Note also that through this paper, for the sake of clarity, we only keep track of the dependence of $C$ and not its precise numerical value (we will, for instance, write $2 C/\lambda_{\min}(a)$ as $C$ to avoid cumbersome expressions).
\end{proof}

\subsection{Accuracy of the FEM with elements $\phi_i$}

The following theorem shows that the finite element method  with elements $\phi_i$
 achieves the optimal (see  \cite{BeOw:2010}) convergence rate  $\mathcal{O}(H)$ in $\H^1$ norm in its approximation of \eqref{eqn:scalar}.

\begin{Theorem}\label{thm:convergenceglobalbasis}
Let $u$ be the solution of equation \eqref{eqn:scalar}, and $u^H$ be the finite element solution of \eqref{eqn:scalar} over the linear space spanned by the elements $\{\phi_i,\, i=1,\ldots,N\}$ identified in \eqref{eqn:convexopt}. It holds true that
 \begin{equation}
	\label{eq:convergenceglobalbasis}
  \|u-u^H\|_{\H^1_0(\Omega)} \leq C H \|g\|_{L^2(\Omega)}
 \end{equation}
 where the constant $C$ depends only on $\lambda_{\min}(a)$ and $\lambda_{\max}(a)$.
\end{Theorem}
\begin{proof}
The theorem is  a straightforward consequence  of Corollary \ref{lem:interpolationerror} and of
 the fact that $u^H$ minimizes the (squared) distance $\int_\Omega {}(\nabla u -\nabla u^H)^T a(\nabla u -\nabla u^H)$ over the linear space spanned by the elements $\phi_i$.
\end{proof}

\begin{Remark}
Recall that the convergence rate of a FEM applied to \eqref{eqn:scalar} with piecewise linear elements can be arbitrarily bad \cite{BaOs:2000}.
Recall also that the convergence rate of Theorem \ref{thm:convergenceglobalbasis} is optimal  \cite{BeOw:2010}
(this is related to  the Kolmogorov n-widths \cite{Pi:1985,  Melenk:2000, Vy:2008}, which measure how accurately a given set of functions can be approximated by linear spaces of dimension $n$ in a given norm).
\end{Remark}

\section{Recovering $u$ from partial measurements}\label{subsec:inverseproblem}
In this section we will show that our method provides a natural solution to the inverse problem of recovering $u$ from partial measurements.
Consider the problem \eqref{eqn:scalar}. In this inverse problem one is given the following (incomplete) information:
 $a(x)$ is known, $g$ is unknown but we know that $\|g\|_{L^2(\Omega)}\leq M$, $u$ is unknown but its values are known on a finite set of points $\{x_j\,:\, j\in \cN\}$ (through site measurements). The inverse problem is then to recover $u$ (accurately in the $\H^1$ norm).
A solution of this inverse problem can be obtained by first computing the minimizers of \eqref{eqn:convexopt} (or their localized version \eqref{eqn:convexoptlocal}) and approximate $u$ with $\uin(x):=\sum_{i=1}^N u(x_i)\phi_i(x)$. Corollary  \ref{lem:interpolationerror} can then be used to bound the accuracy of the recovery by
\begin{equation}
 \|u-\uin\|_{\H^1_0(\Omega)}\leq C H M
\end{equation}
where the constant $C$ depends only on $\lambda_{\min}(a)$, $\lambda_{\max}(a)$ and $\Omega$. Note that the method is meshless.

\section{Generalization to $d\geq 4$}\label{subsec:gnerealizationdgeq4}

Although for the sake of simplicity we have restricted our presentation to $d\leq 3$, the method and results of this paper can be generalized to $d\geq 4$ by introducing the space $V^{m}$ defined as the set of functions $u\in \H^1_0(\Omega)$ such that $\big(\diiv(a \nabla \cdot)\big)^m u \in L^2(\Omega)$ where $m$ is an integer $m\geq 1$ and $\big(\diiv(a \nabla \cdot)\big)^m u$ is the $m$-iterate of the operator $\big(\diiv(a \nabla \cdot)\big)$, i.e.
\begin{equation}
\big(\diiv(a \nabla \cdot)\big)^1 u:=\big(\diiv(a \nabla u)\big)\text{ and } \big(\diiv(a \nabla \cdot)\big)^m u:= \big(\diiv(a \nabla \cdot)\big)^{m-1}  \big(\diiv(a \nabla u)\big).
\end{equation}
Introducing  $\|.\|_{V^m}$ as the norm on $V^m$ defined by
\begin{equation}
\|u\|_{V^m}:=\Big\|\big(\diiv(a \nabla \cdot)\big)^m u\Big\|_{L^{2}(\Omega)}.
\end{equation}
The results of this paper can be generalized to $d\geq 4$ by choosing $m$ such that $(d-1)/2 \leq  m \leq d/2$, assuming that $g\in L^{2m}(\Omega)$
 and using the interpolation basis defined as the minimizer of

\begin{equation}\label{eqn:convexoptgenm}
\begin{cases}
\text{Minimize } \|\phi\|_{V^m}\\
\text{Subject to }  \phi\in V^m_i.
\end{cases}
\end{equation}
where
\begin{equation}\label{eq:uygey33}
V^m_i:=\{\phi\in V^m \mid \phi(x_i)=1 \text{ and } \phi(x_j)=0,  \text{ for } j\in \{1,\ldots, N\} \text{ such that } j\neq i\}.
\end{equation}
Note that the solutions of \eqref{eqn:convexoptgenm} are $2m$-harmonic in the sense that they satisfy
\begin{equation}
\big(\diiv(a \nabla \cdot)\big)^{2m} \phi(x)=0 \text{ for }x\not=x_j.
\end{equation}
Note that the solutions of \eqref{eqn:convexopt} are bi-harmonic on $\Omega\setminus \{x_j\,:\, j\in \cN\}$.

For $d=1$ one can also obtain an accurate interpolation basis by minimizing $\int_{\Omega}\nabla \phi a \nabla \phi$ subject to the pointwise interpolation constraints used in \eqref{eq:uygey33}. By doing so one would obtain basis elements $\phi_i$ that are harmonic in $\Omega\setminus \{x_j\,:\, j\in \cN\}$ and recover a numerical homogenization method based on ``harmonic coordinates''   \cite{BaCaOs:1994, OwZh:2007a}.

Note also that the variation formulation \eqref{eqn:convexoptgenm} does not require the operator $L:=\diiv(a\nabla \cdot)$ to be self-adjoint (but solely depends on the minimization of the $L^2$ norm of $L^{2m}\phi$).

\section{Rough Polyharmonic Splines}\label{sec:rps}

\subsection{The elements $\phi_i$ as rough polyharmonic splines}
The interpolation basis $\phi_i$ constructed in this work can be seen as a generalization of polyharmonic splines to differential operators with rough coefficients. More precisely, as polyharmonic splines lead to an accurate interpolation of smooth functions (solutions of the laplace operator), the ``rough polyharmonic splines'' introduced in this paper lead to an accurate interpolation of solutions of differential operators with rough coefficients. Note also that as polyharmonic splines are $m$-harmonic in the sense that they satisfy $\Delta^m \phi=0$ away from the coarse nodes,  the solutions of
 \eqref{eqn:convexoptgenm} are $2m$-harmonic in the sense that they satisfy $\big(\diiv(a \nabla \cdot)\big)^{2m} \phi(x)=0$ away from those coarse nodes.
This generalization is challenging in several major ways due to the lack of regularity of coefficients. In particular, the Fourier analysis toolkit is no longer available and the simple task of enforcing clamped boundary conditions (trivial with smooth coefficients) becomes a significant difficulty with rough coefficients: consider, for instance, the problem of finding a function $\phi\in \H^1\big(B(0,1)\big)$ that satisfies zero Neumann and Dirichlet boundary conditions on $\partial B(0,1)$ and such that $\diiv(a\nabla \phi)\in L^2\big(B(0,1)\big)$. We will now give a short review of polyharmonic splines.

\subsection{Polyharmonic splines: a short review.}
Let $X$ be a discrete (possibly finite) set of points of $\R^d$. Let $u$ be a real valued function defined on $X$.
Polyharmonic splines have emerged through the problem of interpolating $u$ (from its known values on $X$) to a real function $\phi$ defined on $\R^d$ (with $\phi(x_j)=u(x_j)$ for $x_j\in X$). Let $m$ be an integer strictly greater than $d/2$.  The idea behind polyharmonic splines is to seek the most ``regular'' interpolation by selecting $\phi$ to be a minimizer of the semi-norm
 \begin{equation}\label{eq:poharsp}
 \Big(\int_{\R^d} \sum_{\alpha \in \N^d,\, |\alpha|=m} c_\alpha  \big(\frac{\partial^{\alpha} u}{\partial x^\alpha}(x)\big)^2\,dx\Big)^\frac{1}{2}
 \end{equation}
  over all functions $\phi$ in $\H^m(\R^d)$ interpolating $u$ on $X$,  where the parameters $c_{\alpha}$ are positive coefficients usually \cite{Duchon:1977, Rossini:2009} chosen to be equal to
  $\frac{m!}{\alpha!}$ to ensure the rotational invariance of the semi-norm (those coefficients are also sometimes chosen to be equal to one  \cite{Rabut1:1992, Rabut2:1992}).
 Writing $\Delta$ the Laplace operator on $\R^d$ ($\Delta \phi=\sum_{i=1}^d \frac{\partial^2 \phi}{\partial x_i^2}$) and $\Delta^m$ the m-iterate of $\phi$ ($\Delta^1=\Delta$ and $\Delta^k \phi=\Delta(\Delta^{k-1}\phi)$) the solutions of \eqref{eq:poharsp} satisfy $\Delta^m \phi=0$ on $\R^d\setminus X$ and are therefore called
 m-harmonic splines.  A Polyharmonic  spline of order $m$  is therefore also commonly defined \cite{Kounchev:2005}
 as a tempered distribution $\phi$ on $\R^d$ that is $2m-d-1$ continuously differentiable and such that $ \Delta^m \phi=0$ on $\R^d\setminus X$.

Polyharmonic splines can be represented via weighted sums of shifted fundamental solutions of $\Delta$, (i.e. they can be written  \cite{Duchon:1976}
  $\phi=\sum_{x_j\in X} c_j \tau(x-x_j) +p_{m-1}(x)$ where $\tau$ is the fundamental solution of $\Delta^m$ ($\Delta^m \tau(x)=\delta(x)$) and $p_{m-1}$ is a polynomial of degree at most $m-1$).

 When $X$ forms a regular lattice of $\R^d$ the resulting splines are referred to as polyharmonic cardinal splines. In this case (which has been extensively  studied \cite{Schoenberg:1973, Madych1:1990,Madych2:1990,Madych3:1990}) it can be shown \cite{Rabut:1990, Rabut1:1992, Rabut2:1992} that the basis elements allowing for the interpolation of $u$ can be expressed as  $\Delta_d^m \tau$  (where $\tau$ is the fundamental solutions of $\Delta^m$  and $\Delta_d$ is the finite difference discretization of $\Delta$ on $\Z^d$).

Polyharmonic splines can be traced back to the seminal work of Harder and Desmarais \cite{Harder:1972} on the interpolation of functions of two variables based on the  minimization of a quadratic functional corresponding to the bending energy of a thin plate (we refer to \cite{Brownlee:2004} for a review) and the extension of this idea to higher dimensions in the seminal work of Duchon  \cite{Duchon:1976,Duchon:1977,Duchon:1978} (built on earlier work by Atteia \cite{Atteia:1970}). We also refer to the related work of Schoenberg ($d=1$) on cardinal spline interpolation \cite{Schoenberg:1973}, to
\cite{Madych:1988} for $L^\infty$  radial basis functions interpolation error estimates and to \cite{Matveev:1992b} for the rapid decay (locality) of polyharmonic splines.

\section{Localization of the interpolation basis}\label{sec:Localization}

\subsection{Introduction of the localized basis}\label{subsec:introlocbasis}
For each $i\in \{1,\dots,N\}$ let $\Omega_i$ be an open subset of $\Omega$ containing $x_i$. Let $\cN(\Omega_i)$ be the set of indices corresponding to the
interior nodes of the coarse
mesh contained in $\Omega_i$, i.e.
\begin{equation}\label{eq:defcnomi}
\cN(\Omega_i):=\{j \in \{1,\ldots, N\} \mid  x_j \in \Omega_i \}.
\end{equation}
Define
\begin{equation}
\begin{split}
V(\Omega_i):=\big\{\phi\in \H^1_0(\Omega)\mid \phi=0 \text{ on }  \partial\Omega_i \cup (\Omega\setminus \Omega_i) \text{ and } \diiv(a\nabla \phi)\in L^2(\Omega_i)\big\}
\end{split}
\end{equation}
and
\begin{equation}
\begin{split}
V_i(\Omega_i):=&\big\{\phi\in V(\Omega_i) \mid
\phi(x_i)=1 \text{ and } \phi(x_j)=0,  \text{ for } j\in \cN(\Omega_i) \text{ such that } j\neq i \big\}
\end{split}
\end{equation}
Consider
the following optimization problem over $V_i(\Omega_i)$:
\begin{equation}\label{eqn:convexoptlocal}
\begin{cases}
\text{Minimize }  \int_{\Omega_i}\big|\diiv (a\nabla \phi)\big|^2\\
\text{Subject to }  \phi\in V_i(\Omega_i).
\end{cases}
\end{equation}

The proof of the following proposition is identical to that of Proposition \ref{prop:nonempty} ($V_i$ is simply replaced by $V_i(\Omega_i)$).
\begin{Proposition}\label{prop:dhi3uhe}
$V_i(\Omega_i)$ is a non-empty closed convex affine subspace of $V(\Omega_i)$ under the norm $\|v\|_{V(\Omega_i)}^2:=\int_{\Omega_i}\big(\diiv(a\nabla v)\big)^2$.
Problem \eqref{eqn:convexoptlocal} is a strictly convex (quadratic) optimization problem over $V_i(\Omega_i)$ with a unique minimizer denoted by $\phi_i^{\loc}$.
\end{Proposition}
\begin{Remark}
As in Proposition \ref{prop:nonempty}, the unique minimizer of \eqref{eqn:convexoptlocal} can be represented using the Green's function of the
operator $-\diiv(a\nabla \cdot)$ with Dirichlet boundary condition on $\partial \Omega_i$, i.e., we have for $x\in \Omega_i$
\begin{equation}\label{eq:phiidefloc}
\phi_i^{\loc}(x):=\sum_{j\in \cN(\Omega_i)} P_{i,j}^{i,\loc}\, \tau^{i,\loc}(x,x_j)
\end{equation}
with
\begin{equation}\label{eq:thetaijloc}
\tau^{i,\loc}(x,y):=\int_{\Omega_i} G^{i,\loc}(x,z) G^{i,\loc}(z,y)\,dy
\end{equation}
where $G^{i,\loc}(x,y)$ is the Green's function of $-\diiv(a\nabla)$ in $\Omega_i$ with zero Dirichlet boundary condition, i.e.
the solution of (for $y\in \Omega_i$)
\begin{equation}\label{eqn:scalarGreensloc}
\begin{cases}
    -\diiv \Big(a(x)  \nabla G^{i,\loc}(x,y)\Big)=\delta(x-y) \quad \text{ for }\quad x \in \Omega_i;\\
    G^{i,\loc}(x,y)=0 \quad \text{for}\quad x\in \partial \Omega_i,
    \end{cases}
\end{equation}
and,  writing $N_i:=|\cN(\Omega_i)|$ the number of coarse nodes $x_j$ that are contained in $\Omega_i$, $P^{i,\loc}$ is the $N_i\times N_i$ positive definite symmetric matrix, defined over $\cN(\Omega_i)$, by $P^{i,\loc}=(\Theta^{i,\loc})^{-1}$ where $\Theta^{i,\loc}$ is the $N_i\times N_i$ positive definite symmetric matrix, defined over $\cN(\Omega_i)$, by
\begin{equation}\label{eq:theta:edfloc}
\Theta_{k,j}^{i,\loc}:=\int_{\Omega_i} G^{i,\loc}(x,x_k) G^{i,\loc}(x,x_j)\,dx.
\end{equation}
\end{Remark}
 \begin{Remark}
 It is important to note that in practical (numerical) applications each (localized) element $\phi_i^{\loc}$ is obtained by solving one (local, i.e. over $\Omega_i$) quadratic optimization problem \eqref{eqn:convexoptlocal} rather than through the representation formula \eqref{eq:phiidefloc} (which would require solving $N_i$ elliptic problems in $\Omega_i$ corresponding to the values of the local Green's function).
\end{Remark}

\subsection{Accuracy of the localized basis as a function of the norm of the difference between the global and the localized basis}
We introduce in this subsection lemmas that will be used to control the accuracy of the local basis $(\phi_i^{\loc})_{i\in \cN}$ in approximating solutions of
\eqref{eqn:scalar} (i.e.,  the error associated with replacing the global basis $(\phi_i)_{i \in \cN}$ with a localized version $(\varphi_i)_{i \in \cN}$). Although $(\varphi_i)_{i\in \cN}$  is arbitrary in the following lemma, it is intended to be a localized version of the global basis.
\begin{Lemma}\label{lem:errorloc1}
Let $(\varphi_i)_{i \in \cN}$ be $N$ elements of $\H^1_0(\Omega)$.
For $c\in \R^N$, $c\not=0$, define
\begin{equation}
v(x):=\sum_{i=1}^N c_i \phi_i(x)\quad \text{ and }\quad
w(x):=\sum_{i=1}^N c_i \varphi_i(x)
\end{equation}
It holds true that
\begin{equation}\label{eq:intermcomp2}
\frac{\|v-w\|_{\H^1_0(\Omega)}}{\big\|\diiv(a\nabla v)\big\|_{L^2(\Omega)}} \leq  C N \max_{i\in \cN}\|\phi_i-\varphi_i\|_{\H^1_0(\Omega)}
\end{equation}
where the constant $C$ is finite and depends only on $\lambda_{\min}(a)$, $\lambda_{\max}(a)$ and $\Omega$.
\end{Lemma}
\begin{proof}
We have, using \eqref{eqn:orthogjuoddoty},
\begin{equation}
\|v-w\|_{\H^1_0(\Omega)}= \big(\frac{c^T M c}{c^T Pc}\big)^\frac{1}{2} \big\|\diiv(a\nabla v)\big\|_{L^2(\Omega)}
\end{equation}
where $M$ is the $N\times N$ matrix defined by
\begin{equation}
M_{i,j}:=\int_{\Omega}(\nabla \phi_i-\nabla \varphi_i)^T (\nabla \phi_j-\nabla \varphi_j)
\end{equation}
and $P$ is defined in \eqref{eq:defP}.
Note that since $\Theta$ is symmetric positive definite, there exists a symmetric positive definite $N\times N$ matrix $T$ such that $\Theta=T^2$. Recalling that $P=\Theta^{-1}$, we obtain
\begin{equation}
\sup_{c\in \R^N,\, c\not=0}\frac{c^T M c}{c^T Pc}=\lambda_{\max}(T M T)\leq \lambda_{\max}(M)\lambda_{\max}(\Theta)
\end{equation}
Using \eqref{eq:theta:edf} we have for $c\in \R^N$
\begin{equation}
c^T \Theta c=\int_{\Omega}\big(\sum_{i=1}^N G(x,x_i) c_i\big)^2\,dx
\end{equation}
It follows from the Cauchy-Schwartz inequality that
\begin{equation}
c^T \Theta c\leq \big(\sum_{i=1}^N  c_i^2 \big)\Big(\sum_{i=1}^N \int_{\Omega}\big(G(x,x_i)\big)^2\,dx \Big).
\end{equation}
Using \eqref{eq:boundonGtheta}, we conclude that
\begin{equation}
\lambda_{\max}(\Theta)\leq N K
\end{equation}
where the constant $K$ is finite and depends only on $\lambda_{\min}(a)$, $\lambda_{\max}(a)$ and $\Omega$. To summarize we have obtained that
\begin{equation}\label{eq:intermcomp}
\frac{\|v-w\|_{\H^1_0(\Omega)}}{\big\|\diiv(a\nabla v)\big\|_{L^2(\Omega)}} \leq  C N^{1/2}\big(\lambda_{\max}(M)\big)^\frac{1}{2}
\end{equation}
 We now need to control $\lambda_{\max}(M)$. Since  $M$ is symmetric and positive its maximum eigenvalue is less or equal than its trace.
It follows that,
\begin{equation}\label{eq:gdtetd}
\lambda_{\max}(M)\leq \sum_{i=1}^N M_{i,i} \leq N \max_{i} M_{i,i}.
\end{equation}
Combining \eqref{eq:gdtetd} with \eqref{eq:intermcomp} we deduce \eqref{eq:intermcomp2}.
\end{proof}

The following lemma bounds the accuracy of the localized basis $(\phi_i^{\loc})_{i\in \cN}$ in approximating solutions of $\eqref{eqn:scalar}$. In particular, it shows that if $\max_{i\in \cN}\|\phi_i-\phi_i^{\loc}\|_{\H^1_0(\Omega)}$ is sufficiently small then the accuracy of the localized basis is similar to that of the global basis.

\begin{Lemma}\label{lem:errorbasis}
Let $u$ be the solution of equation \eqref{eqn:scalar}, and $u^{H,\loc}$ be the finite element solution of \eqref{eqn:scalar} over the linear space spanned by the elements $\{\phi_i^{\loc},\, i=1,\ldots,N\}$ identified in \eqref{eqn:convexoptlocal}. It holds true that
 \begin{equation}
	\label{eq:errorlocalbasis}
 \|u-u^{H,\loc}\|_{\H^1_0(\Omega)} \leq C \|g\|_{L^2(\Omega)} \Big( H  +  N \max_{i\in \cN}\|\phi_i-\phi_i^{\loc}\|_{\H^1_0(\Omega)}\Big)
 \end{equation}
 where the constant $C$ depends only on $\lambda_{\min}(a)$, $\lambda_{\max}(a)$ and $\Omega$.
\end{Lemma}
\begin{proof}
Let $\uin$ be the interpolation of the solution of \eqref{eqn:scalar} over the elements $\{\phi_i,\, i=1,\ldots,N\}$ identified in \eqref{eqn:convexopt}.
 Let $u^{\rm in,\loc}$ be the interpolation of the solution of \eqref{eqn:scalar} over the elements $\{\phi_i^{\loc},\, i=1,\ldots,N\}$ identified in \eqref{eqn:convexoptlocal}.
 Using the triangle inequality we have
 \begin{equation}
	\label{eq:convtriangleis}
  \|u-u^{\rm in,\loc}\|_{\H^1_0(\Omega)} \leq  \|u-\uin\|_{\H^1_0}+ \|\uin-u^{\rm in,\loc}\|_{\H^1_0(\Omega)}.
 \end{equation}
Using Corollary  \ref{lem:interpolationerror} and Lemma \ref{lem:errorloc1} we deduce that
  \begin{equation}
	\label{eq:conhggngleis}
  \|u-u^{\rm in,\loc}\|_{\H^1_0(\Omega)} \leq  C H \|g\|_{L^2(\Omega)}   + C N \max_{i\in \cN}\|\phi_i-\phi_i^{\loc}\|_{\H^1_0(\Omega)}\big\|\diiv(a\nabla \uin)\big\|_{L^2(\Omega)}.
 \end{equation}
 Using the minimization property given in Theorem \ref{lem:minimizingproperty} we have
   \begin{equation}
\big\|\diiv(a\nabla \uin)\big\|_{L^2(\Omega)}\leq \big\|\diiv(a\nabla u)\big\|_{L^2(\Omega)}=\|g\|_{L^2(\Omega)}.
 \end{equation}
 We deduce that
   \begin{equation}
	\label{eq:ceddegngleis}
  \|u-u^{\rm in,\loc}\|_{\H^1_0(\Omega)} \leq  C \|g\|_{L^2(\Omega)} \Big( H  +  N \max_{i\in \cN}\|\phi_i-\phi_i^{\loc}\|_{\H^1_0(\Omega)}\Big).
 \end{equation}
 We conclude the proof using the fact that $u^{H,\loc}$ minimizes $\int_{\Omega}(\nabla u-\nabla v) a (\nabla u-\nabla v)$ over all $v$ in the linear space spanned by the elements $\{\phi_i^{\loc},\, i=1,\ldots,N\}$.
\end{proof}

\subsection{Identification of the difference between the global and the localized basis}
In this subsection we will represent  $\phi_i-\phi_i^{\loc}$ as the solution of a fourth order PDE.
We will assume that $\partial \Omega_i \cap \Omega$ is at some strictly positive distance from the set of coarse nodes, i.e. writing
\begin{equation}\label{eq:defdelta_i}
\delta_i:=\inf_{x\in \partial \Omega_i\cap \Omega,\, j\in \cN} \|x-x_j\|
\end{equation}
we assume that $\delta_i>0$. Although the construction of the localized basis $\varphi_i^{\loc}$ does not require $\delta_i$ to be strictly positive our proof of the accuracy of the localized basis will require that $\delta_i$ is uniformly bounded from below by some arbitrary power of $H$ (i.e., $C\,H^k$ for some $C>0$ and $k\geq 1$).  More precisely, our (final and simplified) a posteriori error estimates (Theorem \ref{thm:errorbasisloc}) will be given under the assumption that
$\min_{i\in \cN}\delta_i\geq H_{\min}/10$ where
\begin{equation}\label{eq:defhmin}
H_{\min}:=\min_{i,j \in \cN}\|x_i-x_j\|.
\end{equation}
 and that the assumption that $\max_{i\in \cN}\delta_i\leq H \leq 1$ (we also make these assumptions to simplify our expressions, note that there is
no loss of generality here since $\Omega$ can be re-scaled to have a diameter bounded by $1$).

Let $f\in H^{-1}(\Omega)$ and consider the system of equations with unknowns  $c\in \R^N$ and $\chi \in V_0$  (where $V_0$ is defined in \eqref{eq:hiuhiu33})
\begin{equation}\label{eq:odhuih}
\begin{cases}
\diiv\Big(a\nabla \big(\diiv (a\nabla \chi)\big)\Big)=f+\sum_{i=1}^N c_i \delta(x-x_i) &\text{ on } \Omega\\
\chi=\diiv(a\nabla \chi)=0 &\text{ on } \partial \Omega\\
\chi=0 &\text{ on }\{x_j\,:\,j\in\cN\}.
\end{cases}
\end{equation}

\begin{Proposition}\label{prop:idiuuhueh3}
Equation \eqref{eq:odhuih} admits a unique solution $\chi \in V_0$ given by
\begin{equation}\label{eq:guyg22}
\chi(x)=\chi^1(x)+\chi^2(x)
\end{equation}
where
\begin{equation}
\chi^1(x)=\int_{\Omega} \tau(x,y) f(y)\,dy
\end{equation}
and
\begin{equation}
\chi^2(x)=-\sum_{k=1}^N\sum_{j=1}^N \tau(x,x_j) P_{j,k} \chi^1(x_k).
\end{equation}
\end{Proposition}
\begin{proof}
Noting that $\tau$ is the fundamental solution of \eqref{eq:odhuippokoh} we obtain that (for $c\in \R^N$) $\chi$ is the unique solution of
\begin{equation}\label{eq:odhuilklkh}
\begin{cases}
\diiv\Big(a\nabla \big(\diiv (a\nabla \chi)\big)\Big)=f+\sum_{i=1}^N c_i \delta(x-x_i) &\text{ on } \Omega\\
\chi=\diiv(a\nabla \chi)=0 &\text{ on } \partial \Omega
\end{cases}
\end{equation}
if and only if
\begin{equation}
\chi(x)=\int_{\Omega} \tau(x,y) f(y)\,dy+\sum_{j=1}^N \tau(x,x_j) c_j.
\end{equation}
The additional constraint
\begin{equation}\label{eq:odhuilklkhbis}
\chi=0 \text{ on }\{x_j\}_{j\in\cN}
\end{equation}
 is then equivalent to
\begin{equation}
\chi^1(x_i)+\sum_{j=1}^N \tau(x_i,x_j) c_j=0
\end{equation}
which (recalling that $\tau(x_i,x_j)=\Theta_{i,j}$ and that $P=\Theta^{-1}$) admits the unique solution $c=-P (\chi^1(x_i))_{i\in \cN}$.
Note that $\chi^1$ has well defined values at the nodes $(x_j)_{j\in \cN}$ since, by \cite{St:1965, GiTr:1983}, it is  H\"{o}lder continuous over $\Omega$.
\end{proof}

Consider the equation
\begin{equation}\label{eq:geuyg3ee}
\begin{cases}
\diiv\Big(a\nabla \big(\diiv (a\nabla \chi)\big)\Big)=-\diiv\Big(a\nabla \big(\diiv (a\nabla \phi_i^{\loc})\big)\Big)+\sum_{i=1}^N c_i \delta(x-x_i) \text{ on } \Omega \\
\chi=\diiv(a\nabla \chi)=0 \text{ on } \partial \Omega\\
\chi=0 \text{ on }\{x_j\,:\, j\in \cN\}.
\end{cases}
\end{equation}
We will show in the proof of Lemma \ref{lem:lemidediff} that \eqref{eq:geuyg3ee} admits a unique solution. We will denote that solution by $\chi_i$ and obtain (in Lemma \ref{lem:lemidediff}) a representation formula for it.

\begin{Lemma}\label{lem:lemidediff}
It holds true that
\begin{equation}\label{eq:yguydgu3ygee}
\phi_i=\phi_i^{\loc}+\chi_i
\end{equation}
where $\chi_i=\chi_i^1+\chi_i^2$ with ($n$ is the unit surface vector oriented towards the exterior of $\Omega_i$)
\begin{equation}\label{eq:chi1}
\chi_i^1(x):=\int_{\partial \Omega_i \cap \Omega}\tau(x,y) n\cdot a \nabla \big(\diiv (a\nabla \phi_i^{\loc})\big)\,dy
-\int_{\partial \Omega_i \cap \Omega} G(x,y) n\cdot a\nabla \phi_i^{\loc}\,dy
\end{equation}
and
\begin{equation}\label{eq:chi1bis}
\chi_i^2(x):=-\sum_{k=1}^N\sum_{j=1}^N \tau(x,x_j) P_{j,k} \chi_i^1(x_k).
\end{equation}
\end{Lemma}
\begin{proof}
{\bf Initial proof.}
Let $\bar{f}$ be an arbitrary element of $\H^{-1}(\Omega)$ and $\bar{\chi}$ the solution of \eqref{eq:odhuilklkh} and \eqref{eq:odhuilklkhbis}.
We have (using $\phi_i -\phi_i^{\loc}=0$ on $\{x_j\,:\, j\in \cN\}$)
\begin{equation}
\int_{\Omega} (\phi_i -\phi_i^{\loc}) \bar{f}=\int_{\Omega} (\phi_i -\phi_i^{\loc}) \diiv\Big(a\nabla \big(\diiv (a\nabla \bar{\chi})\big)\Big).
\end{equation}
Integrating by parts we deduce that
\begin{equation}
\int_{\Omega} (\phi_i -\phi_i^{\loc}) \bar{f}=\int_{\Omega \setminus \partial \Omega_i}\diiv\big(a\nabla  (\phi_i -\phi_i^{\loc})\big) \diiv(a\nabla \bar{\chi})
+\int_{\partial \Omega_i\cap \Omega} \diiv(a\nabla  \bar{\chi})\, n\cdot a\nabla \phi_i^{\loc}.
\end{equation}
Integrating by parts again (using the continuity of $\diiv(a\nabla \phi_i^{\loc})$  across $\partial \Omega_i \cap \Omega$ and the fact that $\bar{\chi}$ is null on the  nodes $x_j$) we have
\begin{equation}
\int_{\Omega \setminus \partial \Omega_i}\diiv\big(a\nabla  (\phi_i -\phi_i^{\loc})\big) \diiv(a\nabla \bar{\chi})
=\int_{\partial \Omega_i \cap \Omega}\bar{\chi} \,n\cdot a\nabla \diiv\big(a\nabla  \phi_i^{\loc}\big).
\end{equation}
Replacing  $\bar{\chi}$ by the solution expressed in \eqref{eq:guyg22}, i.e.
\begin{equation}\label{eq:gnklknuyg22}
\bar{\chi}(x)=\int_{\Omega} \tau(x,y) \bar{f}(y)\,dy-\sum_{k=1}^N\sum_{j=1}^N \tau(x,x_j) P_{j,k} \int_{\Omega} \tau(x_k,y) \bar{f}(y)\,dy
\end{equation}
 we obtain that
\begin{equation}
\begin{split}
\int_{\Omega} (\phi_i -\phi_i^{\loc}) \bar{f}=&-\int_{y\in \Omega}\int_{x\in \partial \Omega_i\cap \Omega} G(x,y)\bar{f}(y)\, n\cdot a\nabla \phi_i^{\loc}(x)\,dy\,dx\\
&+\int_{y\in \Omega}\int_{x\in \partial \Omega_i\cap \Omega} \tau(x,y)\bar{f}(y) \,n\cdot a\nabla \diiv\big(a\nabla  \phi_i^{\loc}(x)\big) \,dy\,dx\\
-\sum_{k=1}^N\sum_{j=1}^N   P_{j,k}\int_{y\in \Omega}&\int_{x\in \partial \Omega_i\cap \Omega} \tau(x,x_j) \tau(x_k,y)\bar{f}(y) \,n\cdot a\nabla \diiv\big(a\nabla  \phi_i^{\loc}(x)\big) \,dy\,dx\\
+\sum_{k=1}^N\sum_{j=1}^N   P_{j,k}\int_{y\in \Omega}&\int_{x\in \partial \Omega_i\cap \Omega} G(x,x_j) \tau(x_k,y)\bar{f}(y) \, n\cdot a\nabla \phi_i^{\loc}(x) \,dy\,dx
\end{split}
\end{equation}
which leads us to \eqref{eq:yguydgu3ygee}.

{\bf Alternate proof.}
For the sake of clarity we will now also give an intuitive proof of \eqref{eq:yguydgu3ygee} based on a direct application
of  Proposition \ref{prop:idiuuhueh3}.
Note that since $\phi_i^{\loc}$ and $\diiv(a\nabla \phi_i^{\loc})$ are continuous across $\partial \Omega_i \cap \Omega$ we have (using the representation formulae \eqref{eq:phiidef} and \eqref{eq:phiidefloc}) that
\begin{equation}
\diiv\Bigg(a\nabla \Big(\diiv \big(a\nabla (\phi_i-\phi_i^{\loc})\big)\Big)\Bigg)=f+\sum_{i=1}^N c_i \delta(x-x_i)
\end{equation}
where
\begin{equation}
f=n\cdot a \nabla \big(\diiv (a\nabla \phi_i^{\loc})\big) \delta_{\partial \Omega_i \cap \Omega}+ \diiv\Big(a\nabla \big(n\cdot a\nabla \phi_i^{\loc}    \delta_{\partial \Omega_i \cap \Omega} \big)\Big)
\end{equation}
and $\delta_{\partial \Omega_i \cap \Omega}$ is the Dirac (surface) distribution  on $\partial \Omega_i \cap \Omega$.
We conclude by using Proposition \ref{prop:idiuuhueh3} and integrating by parts. Note that although Proposition \ref{prop:idiuuhueh3} requires $f\in \H^{-1}(\Omega)$, our initial proof shows that it can  still be applied here (with $f\in \H^{-3}(\Omega)$) to show \eqref{eq:yguydgu3ygee}. Note that our assumption  that  $\partial \Omega_i\cap \Omega$ (the support of $f$) is at some strictly positive distance from the coarse nodes $(x_j)_{j\in \cN}$ (i.e., $\delta_i>0$) implies $\chi_i^1$ has well defined values at the coarse nodes $(x_j)_{j\in \cN}$ (since $\diiv(a\nabla \chi_i^1)$ is in $\H^1$ in a neighborhood of those nodes).
\end{proof}

\subsection{Reverse Poincar\'{e} inequality}
The following lemma is a generalization of the classical Caccioppoli's inequality (usually referred to as the reversed Poincar\'{e} inequality, see \cite{Giaquinta:1983} and Proposition 1.4.1 of \cite{Moser:2012}). We will remind its statement and proof for the sake of completeness.

\begin{Lemma}\label{lem:caccop}
Let $D_1, D_2$ be open subsets of $D_0$ such that $D_2 \subset D_1$ (and such that $D_0\setminus D_1$ is non-empty). If $v\in \H^1(D_0)$ satisfies $\diiv(a\nabla v)=0$ in $D_1$, and $v=0$ on $ \partial D_0$, then
\begin{equation}
\|\nabla v \|_{L^2(D_2)} \leq \frac{C}{\dist(D_2,D_0\setminus D_1)} \|v \|_{L^2(D_1)}+\Big(\big\|\diiv( a \nabla v)\big\|_{L^2(D_1)} \|v\|_{L^2(D_1)}\Big)^\frac{1}{2}.
\end{equation}
\end{Lemma}
\begin{proof}
The idea of the proof is as follows, choose $\eta$ to be a differentiable function on $D_0$ such that $\eta=1$ on $D_2$,  $\eta=0$ on $D_0\setminus D_1$, $0\leq \eta \leq 1$ and $\|\nabla \eta\|_{L^\infty(D_0)}\leq \frac{C}{\dist(D_2,D_0\setminus D_1)}$. Using the fact that $\eta^2 v=0$ on $\partial D_1$
 we obtain that
\begin{equation}
\int_{D_1}\nabla (\eta^2 v) a \nabla v=-\int_{D_1} \eta^2 v \diiv( a \nabla v)
\end{equation}
which leads to
\begin{equation}
\int_{D_1}\eta^2 \nabla v a \nabla v= -2\int_{D_1}\eta v \nabla \eta a \nabla v -\int_{D_1} \eta^2 v \diiv( a \nabla v).
\end{equation}
Using the Cauchy-Schwartz inequality, the uniform bound on $ \nabla \eta$, and the uniform bound on $a$ we deduce that
\begin{equation}
\begin{split}
\int_{D_1}\eta^2 \nabla v a \nabla v \leq & \frac{C}{\dist(D_2,D_0\setminus D_1)} \big(\int_{D_1}\eta^2 \nabla v a \nabla v \big)^{\frac{1}{2}} \|v\|_{L^2(D_1)}
\\&+ \big\|\diiv( a \nabla v)\big\|_{L^2(D_1)} \|v\|_{L^2(D_1)}.
\end{split}
\end{equation}
Writing $\big(\int_{D_1}\eta^2 \nabla v a \nabla v \big)^{\frac{1}{2}}$, we conclude by using the fact that, for positive $x,y,z$, $x\leq y+z/x$ implies that $x\leq y +\sqrt{z}$
and also using the fact that  $\int_{D_2}\nabla v a \nabla v\leq \int_{D_1}\eta^2 \nabla v a \nabla v$.
\end{proof}

\subsection{Bound on the maximum eigenvalue of $P$}
In this subsection we will provide an upper bound on the maximum eigenvalue of $P$ as defined in \eqref{eq:defP}.

\begin{Proposition}\label{prop:contlammaxp}
Let $\Theta$ be defined as in \eqref{eq:theta:edf}, $P$  as in \eqref{eq:defP}, and $H_{\min}$ as in \eqref{eq:defhmin}. It holds true that
\begin{equation}
\lambda_{\max}(P)\leq C \frac{N}{H_{\min}^4}.
\end{equation}
\end{Proposition}
\begin{proof}
 The definition \eqref{eq:defP} implies that the maximum eigenvalue of $P$ is the inverse of the minimum eigenvalue of $\Theta$ as defined in \eqref{eq:theta:edf}. Let $l\in \R^{N}$ such that $\|l\|=1$. We wish to bound from below $l^T \Theta l$. Write
\begin{equation}\label{eq:defgreenv}
v(x):=\sum_{j=1}^N G(x,x_i) l_j
\end{equation}
and observe that
\begin{equation}\label{eq:dh3iuhih}
l^T \Theta l= \|v\|_{L^2(\Omega)}^2.
\end{equation}
Since $\|l\|=1$, there exists $i \in \cN$ such that $|l_i|\geq 1/\sqrt{N}$.
Let $B_{\rho}$ be the intersection of $\Omega$ with the ball of center $x_i$ and radius $\rho$.  Let $\rho=H_{\min}/8$.
Observing that $v$ is harmonic in $B_{7\rho}\setminus B_{\rho}$ we have using Caccioppoli's inequality (see lemma \ref{lem:caccop}) that
\begin{equation}\label{eq:hduyg3ug}
\|\nabla v\|_{L^2(B_{5\rho}\setminus B_{4\rho})}\leq \frac{C}{\rho} \| v\|_{L^2(B_{6\rho}\setminus B_{3\rho})}.
\end{equation}
Using Green's identity and \eqref{eq:defgreenv} we have
\begin{equation}\label{eq:greenidli}
\int_{\partial B_{5 \rho}} n\cdot a \nabla v=\int_{B_{5\rho}}\diiv(a\nabla v)=-l_i.
\end{equation}
Let $\eta$ be a smooth function equal to one on $\partial B_{5 \rho}$, zero on $ B_{4 \rho}$ and such that $0\leq \eta \leq 1$ and
$\|\nabla \eta\|_{L^\infty(B_{5 \rho})}\leq C/\rho$. Integrating by parts, we have
\begin{equation}\label{eq:intbrho4andrho5}
\begin{split}
\int_{\partial B_{5 \rho}} n\cdot a \nabla v&=\int_{\partial B_{5 \rho}}\eta n\cdot a \nabla v\\
&= \int_{B_{5 \rho}\setminus B_{4\rho}}\nabla \eta a \nabla v.
\end{split}
\end{equation}
Observing that
\begin{equation}
\begin{split}
\big|\int_{B_{5 \rho}\setminus B_{4\rho}}\nabla \eta a \nabla v\big| \leq \frac{C}{\rho} \|\nabla v\|_{L^2(B_{5 \rho}\setminus B_{4\rho})}
\end{split}
\end{equation}
we obtain from $|l_i|\geq 1/\sqrt{N}$, \eqref{eq:greenidli} and \eqref{eq:intbrho4andrho5} that
\begin{equation}\label{eq:iidg3d}
\frac{1}{\sqrt{N}}\leq  \frac{C}{\rho} \|\nabla v\|_{L^2(B_{5 \rho}\setminus B_{4\rho})}.
\end{equation}
Combining \eqref{eq:iidg3d} with \eqref{eq:hduyg3ug}  we deduce that
\begin{equation}\label{eq:hjjjnduygs3ug}
 C \frac{\rho^2}{\sqrt{N}}\leq  \| v\|_{L^2(B_{6\rho}\setminus B_{3\rho})}.
\end{equation}
Combining \eqref{eq:hjjjnduygs3ug} with  \eqref{eq:dh3iuhih} we deduce that
\begin{equation}
l^T\Theta l\geq C \frac{\rho^4}{N}
\end{equation}
which concludes the proof.
\end{proof}

\subsection{Pointwise estimates on solutions of elliptic equations with discontinuous coefficients}
Write $B(x,\rho)$ the Euclidean ball of center $x$ and radius $\rho$ and $\Omega(x,\rho)$ the intersection of $B(x,\rho)$ with $\Omega$.

\begin{Lemma}\label{lem:lem3controlnew}
Let $v\in \H^1(\Omega)$ such that $\diiv(a\nabla v)\in L^2(\Omega)$, $d\leq 3$ and $x_0\in \Omega$. Then there exists a constant $C$ depending only on $\lambda_{\min}(a)$ and $\lambda_{\max}(a)$ such that
\begin{equation}
\|v\|_{L^\infty(\Omega(x_0,\rho))}\leq \big( \frac{C}{\rho^d}\int_{\Omega(x_0,2\rho)} v^2\big)^\frac{1}{2} + C \rho^{2-\frac{d}{2}}
\big\|\diiv(a\nabla v)\big\|_{L^{2}(\Omega(x_0,\rho))}.
\end{equation}
\end{Lemma}
\begin{proof}
Lemma \ref{lem:lem3controlnew} is classical and we refer to \cite{Stampaccia:1964} for its proof. The idea is to decompose
$v$ as $v=v_1+v_2$ where $v_1$ and $v_2$ satisfy  $\diiv(a\nabla v_1)=\diiv(a\nabla v)$ and $\diiv(a\nabla v_2)=0$ on $ \Omega(x_0,\rho)$   with boundary conditions $v_1=0$ and $v_2=v$ on $\partial \Omega(x_0,\rho)$.
By  the estimates on the $L^2$-norm of the Green's functions given in Lemma \ref{lem:Theta} (see also \cite{Stampaccia:1964, St:1965, GiTr:1983})
\begin{equation}
\|v_1\|_{L^\infty(\Omega(x_0,\rho))} \leq  C \rho^{2-\frac{d}{2}}
\big\|\diiv(a\nabla v)\big\|_{L^{2}(\Omega(x_0,\rho))}.
\end{equation}
Since $v_2$ is harmonic we have (using for instance Theorem 5.1 of \cite{Stampaccia:1964}),
\begin{equation}\label{eq:dguge3dddeee}
\|v_2\|_{L^\infty(\Omega(x_0,\rho))} \leq C \|v_2\|_{L^2(\Omega(x_0,2\rho))}.
\end{equation}
\end{proof}

\begin{figure}[tp]
	\begin{center}
		\subfigure[$\Omega_i$]{
			\includegraphics[width=0.4\textwidth]{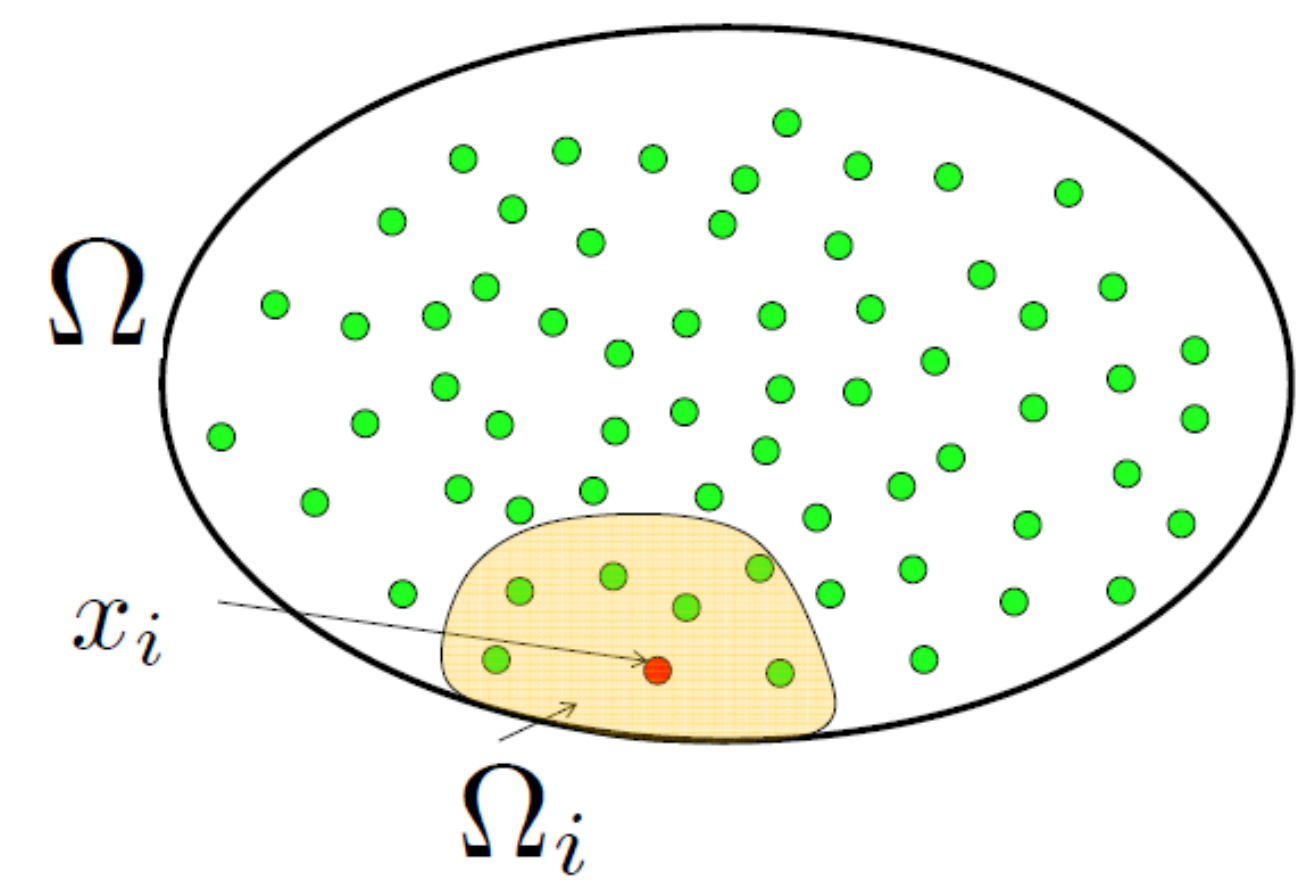}\label{fig:subdomain1}
		}
		\subfigure[$\Omega_{i}^\delta$]{
			\includegraphics[width=0.4\textwidth]{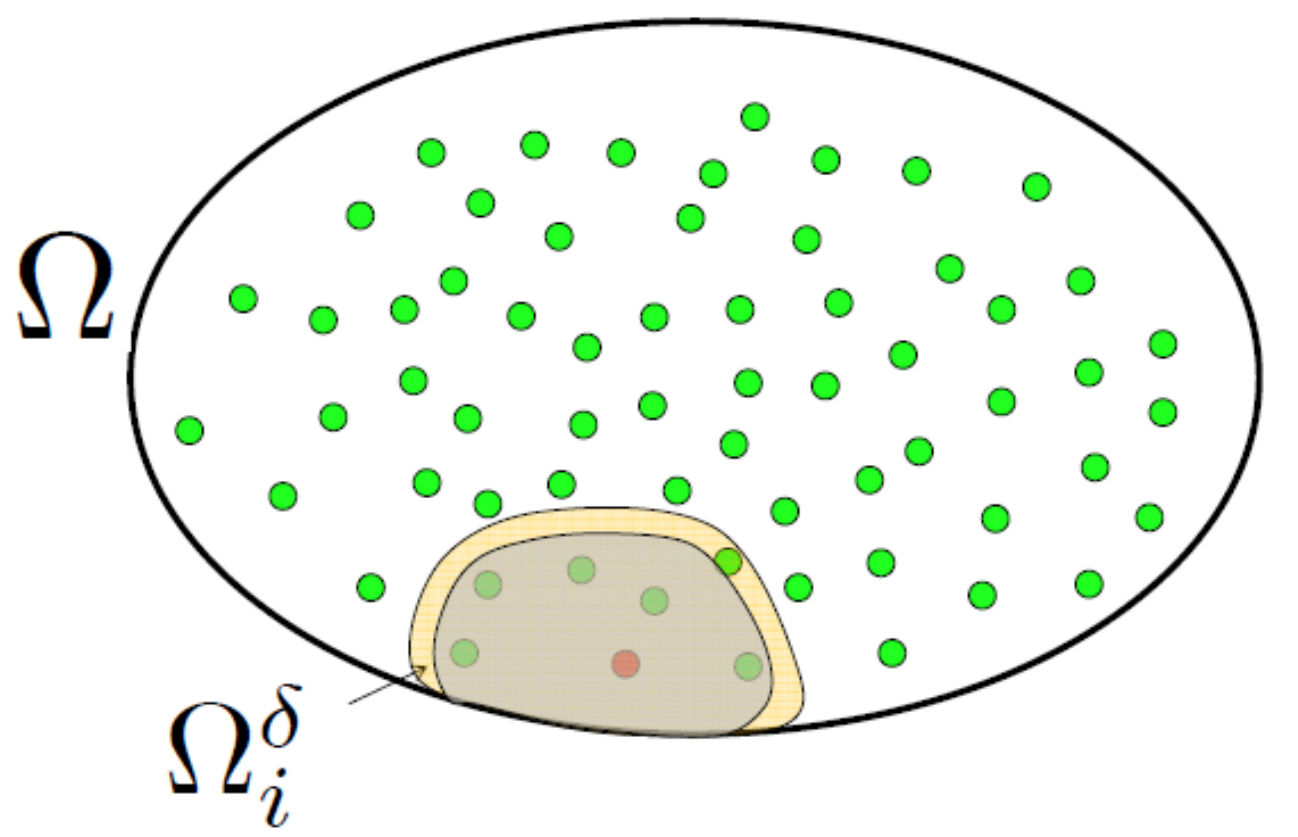}\label{fig:subdomain2}
		}
		\caption{$\Omega_i$ and $\Omega_i^\delta$ as defined in \eqref{eq:vgdvy3ve3e}.}
		\label{fig:subdomain}
	\end{center}
\end{figure}

\subsection{Control of the norm of the difference between the global and the localized basis}

Write $\operatorname{dist}(x,A)$ the distance from a point $x$ to a set $A$ in Euclidean norm.  Define for $\delta>0$
\begin{equation}\label{eq:vgdvy3ve3e}
\Omega_{i}^\delta:=\big\{x\in \Omega_i \mid \operatorname{dist}(x,\partial \Omega_i \cap \Omega)<\delta\big\}.
\end{equation}
See Figure \ref{fig:subdomain} for an illustration of $\Omega_i$ and $\omega_i^\delta$ when $\partial \Omega_i \cap \partial \Omega \not=\emptyset$.

\begin{Lemma}\label{lem:contchi1}
We have for $\delta_i \leq 1$
\begin{equation}\label{eq:unifchi}
\begin{split}
\big\|\chi_i^1\big\|_{L^\infty(\Omega)}\leq& C \delta_i^{-2}\big\|\diiv (a\nabla \phi_i^{\loc})\big\|_{L^2(\Omega_i^\frac{3\delta_i}{4})}
+C \delta_{i}^{-\frac{5}{2}}\big\| \phi_i^{\loc}\big\|_{L^2(\Omega_i^\frac{3\delta_i}{4})}.
\end{split}
\end{equation}
\end{Lemma}
\begin{proof}
Let $\eta$ be a smooth function  equal to $1$ on $\bar{\Omega}_{i}^{\frac{\delta_i}{4}}$, $0$ on $\Omega_i\setminus \Omega_{i}^{\frac{\delta_i}{2}}$ and such that $0\leq \eta \leq 1$ and such that
\begin{equation}
\sup_{x\in \Omega_i} |\nabla \eta(x)|\leq \frac{C}{\delta_i}
\end{equation}
for some constant $C$ independent from $H$ and $\delta_i$.
Let $\chi_i^1$ be the function defined in \eqref{eq:chi1}.
Noting that $\eta=1$ on  $\partial \Omega_i \cap \Omega$,  $\tau(x,y)=0$ and $G(x,y)=0$ for $y\in \partial \Omega$ we have
\begin{equation}
\chi_i^1(x)=\int_{\partial \Omega_i }\tau(x,y) n\cdot \eta a \nabla \big(\diiv (a\nabla \phi_i^{\loc})\big)\,dy
-\int_{\partial \Omega_i } G(x,y) n\cdot \eta a\nabla \phi_i^{\loc}\,dy.
\end{equation}
We deduce, after integration by parts and using the fact that $\diiv\Big(a \nabla \big(\diiv (a\nabla \phi_i^{\loc})\big)\Big)=0$ on $\Omega_i^{\frac{\delta_i}{2}}$, that
\begin{equation}
\begin{split}
\chi_i^1(x)=&\int_{ \Omega_i }\nabla \tau(x,y) \eta a \nabla \big(\diiv (a\nabla \phi_i^{\loc})\big)\,dy
+\int_{ \Omega_i } \tau(x,y) \nabla \eta a \nabla \big(\diiv (a\nabla \phi_i^{\loc})\big)\,dy\\
&-\int_{\Omega_i } \nabla G(x,y) \eta a\nabla \phi_i^{\loc}\,dy-\int_{\Omega_i }  G(x,y) \nabla \eta a\nabla \phi_i^{\loc}\,dy\\&
-\int_{\Omega_i }  G(x,y)  \eta \diiv(a\nabla \phi_i^{\loc})\,dy.
\end{split}
\end{equation}
Noting (by integration by parts) that
\begin{equation}
\begin{split}
-\int_{\Omega_i } \nabla G(x,y) \eta a\nabla \phi_i^{\loc}\,dy&=-\phi_i^{\loc}(x)\eta(x)+\int_{\Omega_i } \nabla G(x,y) a \nabla \eta \phi_i^{\loc}\,dy\\
\leq &\big|\phi_i^{\loc}(x)\eta(x)\big|+ C \big\|\nabla \eta \phi_i^{\loc}\big\|_{L^\infty(\Omega_i^\frac{\delta_i}{2})} \big\|\nabla G(x,y)\big\|_{L^1(\Omega_i^{\frac{\delta_i}{2}})}
\end{split}
\end{equation}
and observing that the $L^2$ norms of $\nabla \tau(x,y)$, $\tau(x,y)$, $ G(x,y)$ and the $L^1$ norm of $\nabla G(x,y)$ are bounded by a constant $C$ depending only on $\lambda_{\min}(a)$, $\lambda_{\max}(a)$ and $\Omega$ (see proof of Lemma \ref{lem:Theta}, equation \eqref{eq:detftau}, \cite{GrWi:1982,Taylor:2012}   and references therein) we deduce that
\begin{equation}
\begin{split}
\big\|\chi_i^1\big\|_{L^\infty(\Omega)}\leq& C \delta_i^{-1}\Big\|\nabla \big(\diiv (a\nabla \phi_i^{\loc})\big)\Big\|_{L^2(\Omega_i^\frac{\delta_i}{2})}
+C \delta_i^{-1}\big\|\nabla \phi_i^{\loc}\|_{L^2(\Omega_i^\frac{\delta_i}{2})}\\&+ C\|\diiv(a\nabla \phi_i^{\loc})\|_{L^2(\Omega_i^\frac{\delta_i}{2})}
+ C \delta_i^{-1} \big\| \phi_i^{\loc}\|_{L^\infty(\Omega_i^\frac{\delta_i}{2})}.
\end{split}
\end{equation}
Since $\big(\diiv (a\nabla \phi_i^{\loc})$ is harmonic on $\Omega_i^{\frac{3\delta_i}{4}}$ we obtain from Caccioppoli's inequality (i.e. Lemma \ref{lem:caccop} with $D_0=\Omega_i$, $D_2=\Omega_i^{\frac{\delta_i}{2}}$ and $D_1=\Omega_i^{\frac{3\delta_i}{4}}$) that
\begin{equation}\label{eq:divachi}
\begin{split}
\Big\|\nabla \big(\diiv (a\nabla \phi_i^{\loc})\big)\Big\|_{L^2(\Omega_i^\frac{\delta_i}{2})}\leq C \delta_i^{-1}
\big\|\diiv (a\nabla \phi_i^{\loc})\big\|_{L^2(\Omega_i^\frac{3\delta_i}{4})}.
\end{split}
\end{equation}
Therefore
\begin{equation}
\begin{split}
\big\|\chi_i^1\big\|_{L^\infty(\Omega)}\leq& C \delta_i^{-2} \big\|\diiv (a\nabla \phi_i^{\loc})\big\|_{L^2(\Omega_i^\frac{3\delta_i}{4})}
\\&+C \delta_i^{-1}\big\|\nabla \phi_i^{\loc}\|_{L^2(\Omega_i^\frac{\delta_i}{2})}+ C \delta_i^{-1} \big\| \phi_i^{\loc}\|_{L^\infty(\Omega_i^\frac{\delta_i}{2})}.
\end{split}
\end{equation}
Similarly, using the reverse Poincar\'{e} inequality (Lemma \ref{lem:caccop} with $D_0=\Omega_i$, $D_2=\Omega_i^{\frac{\delta_i}{2}}$ and $D_1=\Omega_i^{\frac{3\delta_i}{4}}$) and Young's inequality ($|ab|\leq \frac{\epsilon}{2}a^2+\frac{2}{\epsilon}b^2$) we obtain that
\begin{equation}\label{eq:nabchi}
\begin{split}
\big\|\nabla \phi_i^{\loc}\|_{L^2(\Omega_i^\frac{\delta_i}{2})}\leq C \delta_{i}^{-1}\big\| \phi_i^{\loc}\big\|_{L^2(\Omega_i^\frac{3\delta_i}{4})}+ \delta_i \big\|\diiv (a\nabla \phi_i^{\loc})\big\|_{L^2(\Omega_i^\frac{3\delta_i}{4})}.
\end{split}
\end{equation}
Therefore
\begin{equation}\label{eq:uguyg3ee}
\begin{split}
\big\|\chi_i^1\big\|_{L^\infty(\Omega)}\leq& C \delta_i^{-2} \big\|\diiv (a\nabla \phi_i^{\loc})\big\|_{L^2(\Omega_i^\frac{3\delta_i}{4})}
\\&+C \delta_{i}^{-2}\big\| \phi_i^{\loc}\big\|_{L^2(\Omega_i^\frac{3\delta_i}{4})}+ C \delta_i^{-1} \big\| \phi_i^{\loc}\|_{L^\infty(\Omega_i^\frac{\delta_i}{2})}.
\end{split}
\end{equation}

Now to control $\| \phi_i^{\loc}\|_{L^\infty(\Omega_i^\frac{\delta_i}{2})}$ note that on $\Omega_i^\frac{\delta_i}{2}$,
$\phi_i^{\loc}$ can be decomposed as $\phi_i^{\loc}=v_1+v_2$ where $v_1$ and $v_2$ satisfy  $\diiv(a\nabla v_1)=\diiv(a\nabla\phi_i^{\loc})$ and $\diiv(a\nabla v_2)=0$ on $ \Omega_i^\frac{\delta_i}{2}$   with boundary conditions $v_1=0$ and $v_2=\phi_i^{\loc}$ on $\partial \Omega_i^\frac{\delta_i}{2}$.
We therefore obtain that
\begin{equation}\label{eq:ihi2iww}
\big\| \phi_i^{\loc}\|_{L^\infty(\Omega_i^\frac{\delta_i}{2})} \leq \|v_1\|_{L^\infty(\Omega_i^\frac{\delta_i}{2})}+ \|v_2\|_{L^\infty(\Omega_i^\frac{\delta_i}{2})}.
\end{equation}
By \cite{St:1965, GiTr:1983} (one can also use the bounds on the Green's function given in the proof Lemma \ref{lem:Theta}) we have
\begin{equation}
\|v_1\|_{L^\infty(\Omega_i^\frac{\delta_i}{2})} \leq C \big\| \diiv(a\nabla\phi_i^{\loc}) \big\|_{L^2(\Omega_i^\frac{\delta_i}{2})}.
\end{equation}
Since $v_2$ is harmonic we have,
\begin{equation}\label{eq:dguge3eee}
\|v_2\|_{L^\infty(\Omega_i^\frac{\delta_i}{2})} \leq C \|\phi_i^{\loc}\|_{L^\infty(\partial \Omega_i^\frac{\delta_i}{2})}.
\end{equation}
Using lemma \ref{lem:lem3controlnew} we obtain that for $x_0 \in \partial \Omega_i^\frac{\delta_i}{2} \cap \Omega_i$ and $\rho= \delta_i/8$,
\begin{equation}\label{eq:dgygugy3e}
\|\phi_i^{\loc}\|_{L^\infty(B(x_0,\rho)\cap \Omega_i)}\leq \big( \frac{C}{\rho^d}\int_{B(x_0,2\rho)\cap \Omega_i} (\phi_i^{\loc})^2\big)^\frac{1}{2} + C \rho^{2-\frac{d}{2}}
\big\|\diiv(a\nabla \phi_i^{\loc})\big\|_{L^{2}(B(x_0,\rho)\cap \Omega_i)}.
\end{equation}
Combining \eqref{eq:dgygugy3e} with \eqref{eq:dguge3eee} we deduce that
\begin{equation}\label{eq:deddguge3eee}
\|v_2\|_{L^\infty(\Omega_i^\frac{\delta_i}{2})} \leq C \delta_i^{-\frac{d}{2}} \|\phi_i^{\loc}\|_{L^2(\Omega_i^\frac{3\delta_i}{4})}+
\delta_i^{2-\frac{d}{2}} \big\|\diiv(a\nabla \phi_i^{\loc})\big\|_{L^{2}(\Omega_i^\frac{3\delta_i}{4})}.
\end{equation}
Using \eqref{eq:ihi2iww} we deduce that (for $\delta_i\leq 1$ and $d\leq 3$)
\begin{equation}\label{eq:ihi2dddiww}
\big\| \phi_i^{\loc}\|_{L^\infty(\Omega_i^\frac{\delta_i}{2})} \leq C \delta_i^{-\frac{d}{2}} \|\phi_i^{\loc}\|_{L^2(\Omega_i^\frac{3\delta_i}{4})}+
C \big\|\diiv(a\nabla \phi_i^{\loc})\big\|_{L^{2}(\Omega_i^\frac{3\delta_i}{4})}.
\end{equation}
Combining \eqref{eq:ihi2dddiww} and \eqref{eq:uguyg3ee}  we deduce \eqref{eq:unifchi} for $\delta_i \leq 1$.
\end{proof}

\begin{Lemma}\label{lem:contchi2}
Let $\chi_i^2$ be defined as in \eqref{eq:chi1bis}.
We have
\begin{equation}
\|\nabla \chi_i^2\|_{L^2(\Omega)}\leq N \big\|\chi^1_i\big\|_{L^\infty(\Omega)} \max_{j\in \cN} \sqrt{P_{jj}}.
\end{equation}
\end{Lemma}
\begin{proof}
We have
\begin{equation}
\chi_i^2(x):=-\sum_{k=1}^N \phi_k(x) \chi_i^1(x_k).
\end{equation}
Using
\begin{equation}
\|\nabla \phi_k\|_{L^2(\Omega)}\leq C \big\|\diiv(a\nabla \phi_k)\big\|_{L^2(\Omega)}
\end{equation}
we deduce that
\begin{equation}
\|\nabla \chi_i^2\|_{L^2(\Omega)}\leq C \sum_{k=1}^N \sqrt{P_{kk}} \big|\chi_i^1(x_k)\big|_{L^\infty(\Omega)}.
\end{equation}
\end{proof}

\begin{Lemma}\label{lem:contchi1nab}
We have for $0<\delta_i\leq 1$
\begin{equation}\label{eq:4uhuhkjkiu}
\begin{split}
\big\|\nabla \chi_i^1\|_{L^2(\Omega)} \leq  C \delta_i^{-3}\Big( \big\|\diiv (a\nabla \phi_i^{\loc})\big\|_{L^2(\Omega_i^\frac{3\delta_i}{4})}
+\big\| \phi_i^{\loc}\big\|_{L^2(\Omega_i^\frac{3\delta_i}{4})}\Big).
\end{split}
\end{equation}
\end{Lemma}
\begin{proof}
Using
\begin{equation}
\int_{\Omega}\nabla v a \nabla v=-\int_{\Omega} v \diiv(a\nabla v)
\end{equation}
we obtain from \eqref{eq:chi1} that
\begin{equation}\label{eq:1}
\int_{\Omega}\nabla \chi_i^1 a \nabla \chi_i^1 \leq  I_1+ I_2
\end{equation}
with
\begin{equation}\label{eq:ckjkjkjuujhi1}
I_1=\int_{x\in \Omega}\int_{y\in \partial \Omega_i \cap \Omega}\chi_i^1(x) G(x,y) n\cdot a \nabla \big(\diiv (a\nabla \phi_i^{\loc})\big)\,dy
\end{equation}
and
\begin{equation}\label{eq:ckjkuujhi1}
I_2=
-\int_{\partial \Omega_i \cap \Omega} \chi_i^1(y) n\cdot a\nabla \phi_i^{\loc}\,dy.
\end{equation}
Defining $\eta$ as in the proof of Lemma \ref{lem:contchi1} we have (since $\eta=1$ on $\partial \Omega_i \cap \Omega$ and $\chi_i^1=0$ on $\partial \Omega$)

\begin{equation}\label{eq:ckjkuujddhi1}
I_2=
-\int_{\partial \Omega_i } \chi_i^1(y) \eta \, n\cdot a\nabla \phi_i^{\loc}\,dy.
\end{equation}
Integrating by parts, we obtain
\begin{equation}
I_2=-\int_{\Omega_i} \eta  \chi_i^1 \diiv( a\nabla \phi_i^{\loc})-\int_{\Omega_i} \eta \nabla  \chi_i^1 a\nabla \phi_i^{\loc}
- \int_{\Omega_i}   \chi_i^1 \nabla \eta a\nabla \phi_i^{\loc}.
\end{equation}
Using Young's inequality we obtain
\begin{equation}
\big|\int_{\Omega_i} \eta \nabla  \chi_i^1 a\nabla \phi_i^{\loc}\big| \leq \frac{1}{4}\int_{\Omega} \nabla \chi_i^1 a \nabla\chi_i^1+ C
\|\nabla \phi_i^{\loc}\|_{L^2(\Omega_i^\frac{\delta_i}{2})}^2.
\end{equation}
Therefore
\begin{equation}\label{eq:2}
\begin{split}
|I_2|\leq & \|\chi_i^1\|_{L^\infty(\Omega)}\big \|\diiv( a\nabla \phi_i^{\loc})\big\|_{L^2(\Omega_i^\frac{\delta_i}{2})}+C\delta_{i}^{-1} \|\chi_i^1\|_{L^\infty(\Omega)} \|\nabla \phi_i^{\loc}\|_{L^2(\Omega_i^\frac{\delta_i}{2})}\\&+ C \|\nabla \phi_i^{\loc}\|_{L^2(\Omega_i^\frac{\delta_i}{2})}^2+\frac{1}{4}\int_{\Omega} \nabla \chi_i^1 a \nabla\chi_i^1.
\end{split}
\end{equation}
Similarly we have
\begin{equation}
I_1=\int_{x\in \Omega}\int_{y\in \partial \Omega_i }\chi_i^1(x) G(x,y)\eta\, n\cdot a \nabla \big(\diiv (a\nabla \phi_i^{\loc})\big)\,dy
\end{equation}
which, after integration by parts, leads us to
\begin{equation}
\begin{split}
I_1=&\int_{x\in \Omega}\int_{y\in  \Omega_i }\chi_i^1(x) G(x,y)\eta \diiv\Big( a \nabla \big(\diiv (a\nabla \phi_i^{\loc})\big) \Big)\,dy\\
&+ \int_{x\in \Omega}\int_{y\in  \Omega_i } \chi_i^1(x) G(x,y) \nabla \eta  a \nabla \big(\diiv (a\nabla \phi_i^{\loc})\big) \,dy\\
&+ \int_{x\in \Omega}\int_{y\in  \Omega_i } \chi_i^1(x) \nabla_y G(x,y) \eta  a \nabla \big(\diiv (a\nabla \phi_i^{\loc})\big) \,dy.
\end{split}
\end{equation}
Using the fact that $\eta \diiv\Big( a \nabla \big(\diiv (a\nabla \phi_i^{\loc})\big) \Big)=0$ (since the support of $\eta$ does not contain any coarse nodes), and using the bounds the $L^2$-norm of $G$ (see proof of Lemma \ref{lem:Theta}) as well as
\begin{equation}
\begin{split}
\int_{y\in  \Omega_i }  \big(\int_{x\in \Omega}\chi_i^1(x) \nabla_y G(x,y) \big)^2 \leq C \|\chi_i^1\|_{L^\infty(\Omega)}^2
\end{split}
\end{equation}
we obtain that
\begin{equation}\label{eq:3}
\begin{split}
I_1\leq  C(1+\delta_i^{-1}) \|\chi_i^1\|_{L^\infty(\Omega)}\Big\| a \nabla \big(\diiv (a\nabla \phi_i^{\loc})\big) \Big\|_{L^2(\Omega_i^\frac{\delta_i}{2})}.
\end{split}
\end{equation}
Combining \eqref{eq:1}, \eqref{eq:2} and \eqref{eq:3}, we have obtained that
\begin{equation}\label{eq:4}
\begin{split}
\int_{\Omega}\nabla \chi_i^1 a \nabla \chi_i^1 \leq & C \|\chi_i^1\|_{L^\infty(\Omega)}\big \|\diiv( a\nabla \phi_i^{\loc})\big\|_{L^2(\Omega_i^\frac{\delta_i}{2})}+C \|\nabla \phi_i^{\loc}\|_{L^2(\Omega_i^\frac{\delta_i}{2})}^2\\
&+C(1+\delta_i^{-1}) \|\chi_i^1\|_{L^\infty(\Omega)}\Big\| \nabla \big(\diiv (a\nabla \phi_i^{\loc})\big) \Big\|_{L^2(\Omega_i^\frac{\delta_i}{2})}
\\&+ C\delta_{i}^{-1} \|\chi_i^1\|_{L^\infty(\Omega)} \|\nabla \phi_i^{\loc}\|_{L^2(\Omega_i^\frac{\delta_i}{2})}.
\end{split}
\end{equation}
Therefore, using Young's inequality, we have obtained that
\begin{equation}\label{eq:4uiuuhuhiu}
\begin{split}
\big\|\nabla \chi_i^1\|_{L^2(\Omega)} \leq & C (1+\delta_i^{-1}) \|\chi_i^1\|_{L^\infty(\Omega)}+\big \|\diiv( a\nabla \phi_i^{\loc})\big\|_{L^2(\Omega_i^\frac{\delta_i}{2})}+ C \|\nabla \phi_i^{\loc}\|_{L^2(\Omega_i^\frac{\delta_i}{2})}\\
&+\Big\|  \nabla \big(\diiv (a\nabla \phi_i^{\loc})\big) \Big\|_{L^2(\Omega_i^\frac{\delta_i}{2})}.
\end{split}
\end{equation}
We deduce \eqref{eq:4uhuhkjkiu}  by using  \eqref{eq:unifchi}, \eqref{eq:divachi} and \eqref{eq:nabchi} to bound
$\|\chi_i^1\|_{L^\infty(\Omega)}$, $\|\nabla \phi_i^{\loc}\|_{L^2(\Omega_i^\frac{\delta_i}{2})}$ and $\Big\|  \nabla \big(\diiv (a\nabla \phi_i^{\loc})\big) \Big\|_{L^2(\Omega_i^\frac{\delta_i}{2})}$ in \eqref{eq:4uiuuhuhiu}.
\end{proof}

Combining lemmas \ref{lem:lemidediff}, \ref{lem:contchi2} and \ref{lem:contchi1nab} with Proposition \ref{prop:contlammaxp} we deduce the following lemma.

\begin{Lemma}\label{lem:lemidediffcontest}
For $0<\delta_i\leq 1$, it holds true that
\begin{equation}\label{eq:yguydhtwgu3ygee}
\begin{split}
\big\|\nabla (\phi_i-\phi_i^{\loc})\big\|_{L^2(\Omega)}\leq  C N^2  H_{\min}^{-4} \delta_i^{-3}
 \Big( & \big\|\diiv (a\nabla \phi_i^{\loc})\big\|_{L^2(\Omega_i^\frac{3\delta_i}{4})}
+ \big\| \phi_i^{\loc}\big\|_{L^2(\Omega_i^\frac{3\delta_i}{4})}\Big).
\end{split}
\end{equation}
\end{Lemma}

\subsection{A posteriori error estimates.}

\begin{Theorem}\label{thm:errorbasisloc}
Let $u$ be the solution of equation \eqref{eqn:scalar}, and $u^{H,\loc}$ be the finite element solution of \eqref{eqn:scalar} over the linear space spanned by the elements $\{\phi_i^{\loc},\, i=1,\ldots,N\}$ identified in \eqref{eqn:convexoptlocal}. Assume that $\min_{i\in \cN}\delta_i\geq H_{\min}/10$ (where $\delta_i$ is defined in \eqref{eq:defdelta_i} and $H_{\min}$ in \eqref{eq:defhmin}) and that $\max_{i\in \cN}\delta_i\leq H \leq 1$.
It holds true that
 \begin{equation}
	\label{eq:errorlocalbasisthm}
 \|u-u^{H,\loc}\|_{\H^1_0(\Omega)} \leq C \|g\|_{L^2(\Omega)} ( H  +  E)
 \end{equation}
 where the constant $C$ depends only on $\lambda_{\min}(a)$, $\lambda_{\max}(a)$ and $\Omega$, and
 \begin{equation}
E=   H_{\min}^{-7-3d} \max_{i \in \cN}
 \Big(  \big\|\diiv (a\nabla \phi_i^{\loc})\big\|_{L^2(\Omega_i^{H})}
+ \big\| \phi_i^{\loc}\big\|_{L^2(\Omega_i^H)}\Big)
 \end{equation}
 where $\Omega_i^H$ is defined in \eqref{eq:vgdvy3ve3e}.
\end{Theorem}
\begin{proof}
 Theorem \ref{thm:errorbasisloc} is a straightforward consequence of Lemma \ref{lem:lemidediffcontest} and Lemma \ref{lem:errorbasis} and the inequality
 $N\leq C H_{\min}^{-d}$.
 \end{proof}

\begin{Remark}
Observe that there is no loss of generality in requiring that $\min_{i\in \cN}\delta_i \geq  H_{\min}/10$ and that $\max_{i\in \cN}\delta_i\leq H \leq 1$.
 Note also that $\min_{i\in \cN}\delta_i \geq  H_{\min}/10$  corresponds to the condition that the minimum distance between $\partial \Omega_i \cap \Omega$ and the set of coarse nodes is at least $H_{\min}/10$. Although we need this assumption for the simplicity of the proof, numerical experiments suggest that it is not needed for convergence with an optimal rate.
\end{Remark}
\begin{Remark}
Numerical experiments show that $\big\|\diiv (a\nabla \phi_i^{\loc})\big\|_{L^2(\Omega_i^{H})}
+ \big\| \phi_i^{\loc}\big\|_{L^2(\Omega_i^H)}$ decays exponentially fast as a function of $n_i$ where $n_i$ is the minimum number of (coarse) edges (i.e., number of layers of triangles) separating
$x_i$ from $\partial \Omega_i \cap\Omega$. Hence if $n_i$ is of the order of $\ln (1/H)$ we can derive a posteriori error estimates from the
posterior observation that $E\leq H$. A priori estimates can be derived  by bounding the exponential decay of $\big\|\diiv (a\nabla \phi_i^{\loc})\big\|_{L^2(\Omega_i^{H})}
+ \big\| \phi_i^{\loc}\big\|_{L^2(\Omega_i^H)}$ as a function of the number of layers of triangles surrounding each node $x_i$ (work in progress). These a priori error estimates would guarantee $ \|u-u^{H,\loc}\|_{\H^1_0(\Omega)} \leq C \|g\|_{L^2(\Omega)} H$ provided that $\min_{i \in \cN}\dist(x_i,\partial \Omega_i \cap \Omega)\geq C^* H  \ln \frac{1}{H}$ for some constant $C^*$ depending on $\lambda_{\min}(a)$ and $\lambda_{\max}(a)$. We refer to this property as \emph{super-localization}.
Note that each (super-localized) basis function $\phi_i^{\loc}$ is the solution of a sparse/banded/nearly diagonal linear system.
This property is essential in keeping the computational cost minimal (see  Subsection \ref{subsec:compcost}).
Note also that since $\Omega_i^H$ is defined as a layer of thickness $H$ near the portion of the boundary of $\Omega_i$ that is contained in $\Omega$ (see Figure \ref{fig:subdomain}) $\Omega_i^H$ remains at distance $\mathcal{O}(n_i H)$ from $x_i$ even when $x_i$ is close to the boundary of $\Omega$ (and there are no ``boundary effects'' impacting the accuracy of the method).

Note also that because of the exponential decay of $\big\|\diiv (a\nabla \phi_i^{\loc})\big\|_{L^2(\Omega_i^{H})}
+ \big\| \phi_i^{\loc}\big\|_{L^2(\Omega_i^H)}$ we do not
 require $H$ of the same order as $H_{\min}$. In particular, one can increase interpolation points in local regions
where high accuracy is required without loss of accuracy provided that $H_{\min}$ remains bounded from below by a power of $H$.
\end{Remark}

\section{On time dependent problems}\label{sec:timedependent}
As shown in sections 4 and 5 of \cite{OwZh:2011} the accuracy of the basis elements $\phi_i$ and $\phi_i^{\loc}$ remains unchanged when those elements are used to approximate the solutions of the parabolic and hyperbolic equations associated with $-\diiv(a\nabla)$, i.e., for equations of the form
\begin{equation}\label{parabolicscalarproblem0}
\begin{cases}
   \partial_t u(x,t) -\diiv \Big(a(x)  \nabla u(x,t)\Big)=g(x,t) \quad  (x,t) \in \Omega_T;\, g \in L^2(\Omega_T), \\
    u=0 \quad \text{on}\quad \partial \Omega_T,
    \end{cases}
\end{equation}
and
\begin{equation}\label{huperboliccalarproblem0}
\begin{cases}
   \rho(x) \partial_{t}^2 u(x,t) -\diiv \Big(a(x)  \nabla u(x,t)\Big)=g(x,t) \quad  (x,t) \in \Omega_T;\, g \in L^2(\Omega_T), \\
    u=0 \quad \text{on}\quad  \partial \Omega_T,\\
    \partial_t u=0  \quad \text{on}\quad  \Omega \times \{t=0\}.
    \end{cases}
\end{equation}
More precisely the solutions of \eqref{parabolicscalarproblem0} and \eqref{huperboliccalarproblem0} are approximated, via finite element methods, using the linear space spanned by the interpolation basis $\phi_i^{\loc}$ as test and solution space. Note that the  approximate solutions are of the form
\begin{equation}
u^H(x,t)=\sum_i c_i(t) \phi_i^{\loc}(x).
\end{equation}
We refer to sections 4 and 5 of \cite{OwZh:2011} for  precise statements and proofs.

\begin{figure}[tp]
	\begin{center}
		\subfigure[Primal and dual mesh]{
			\includegraphics[width=0.2\textwidth]{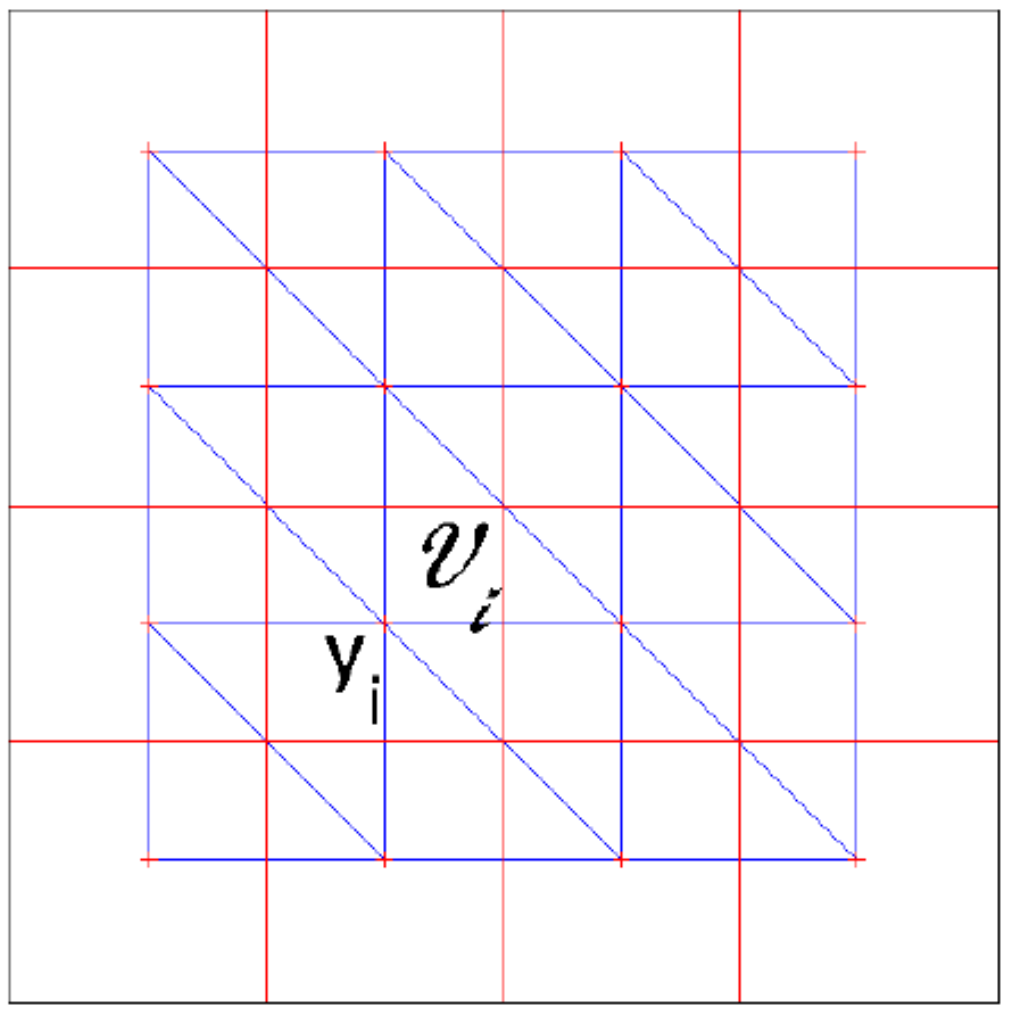}\label{fig:mesh}
		}
		\subfigure[$x_i$ and $\Omega_i$ with $l=1$]{
			\includegraphics[width=0.2\textwidth]{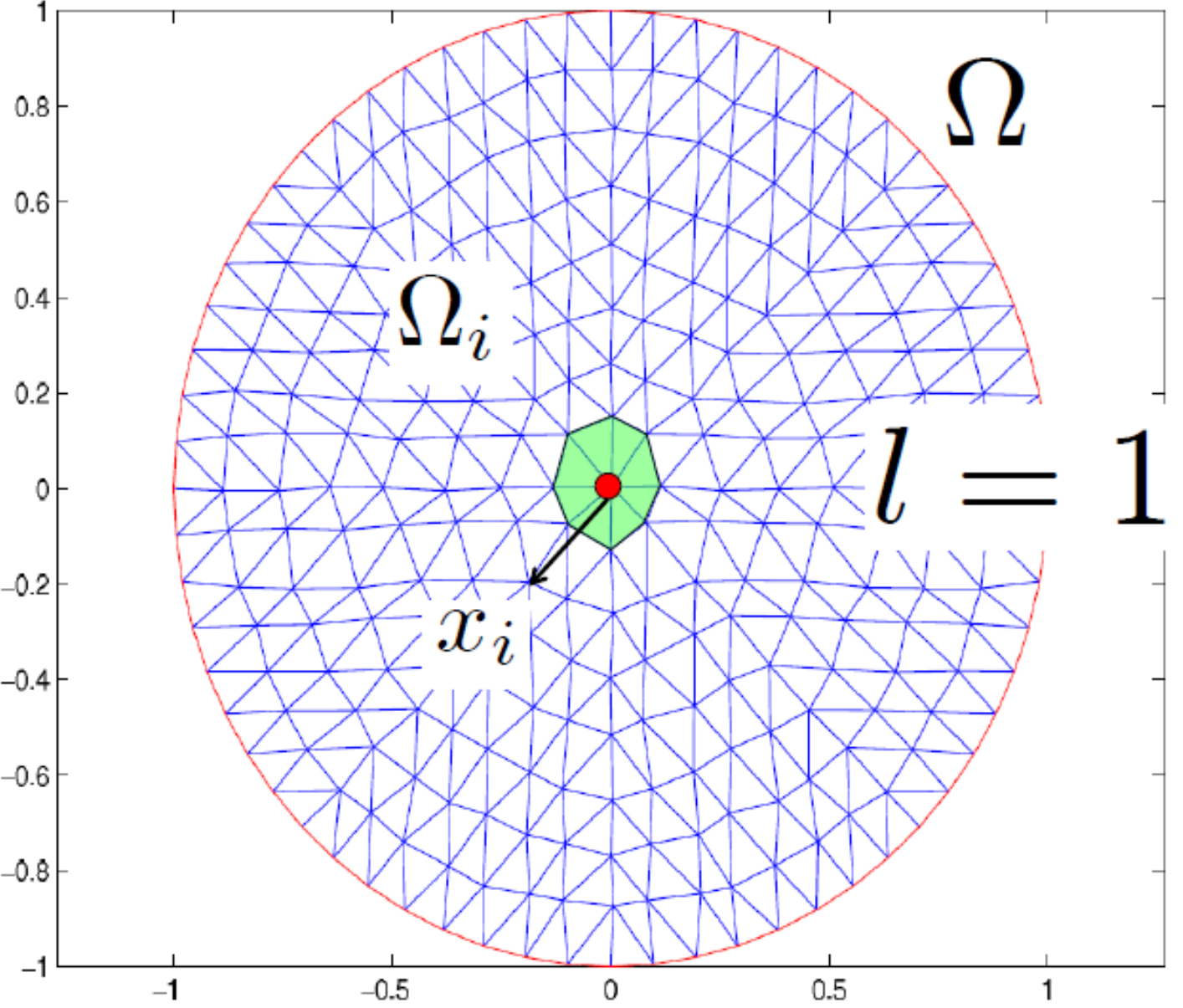}\label{fig:layer1}
		}
		\subfigure[$\Omega_i$ with $l=2$]{
			\includegraphics[width=0.2\textwidth]{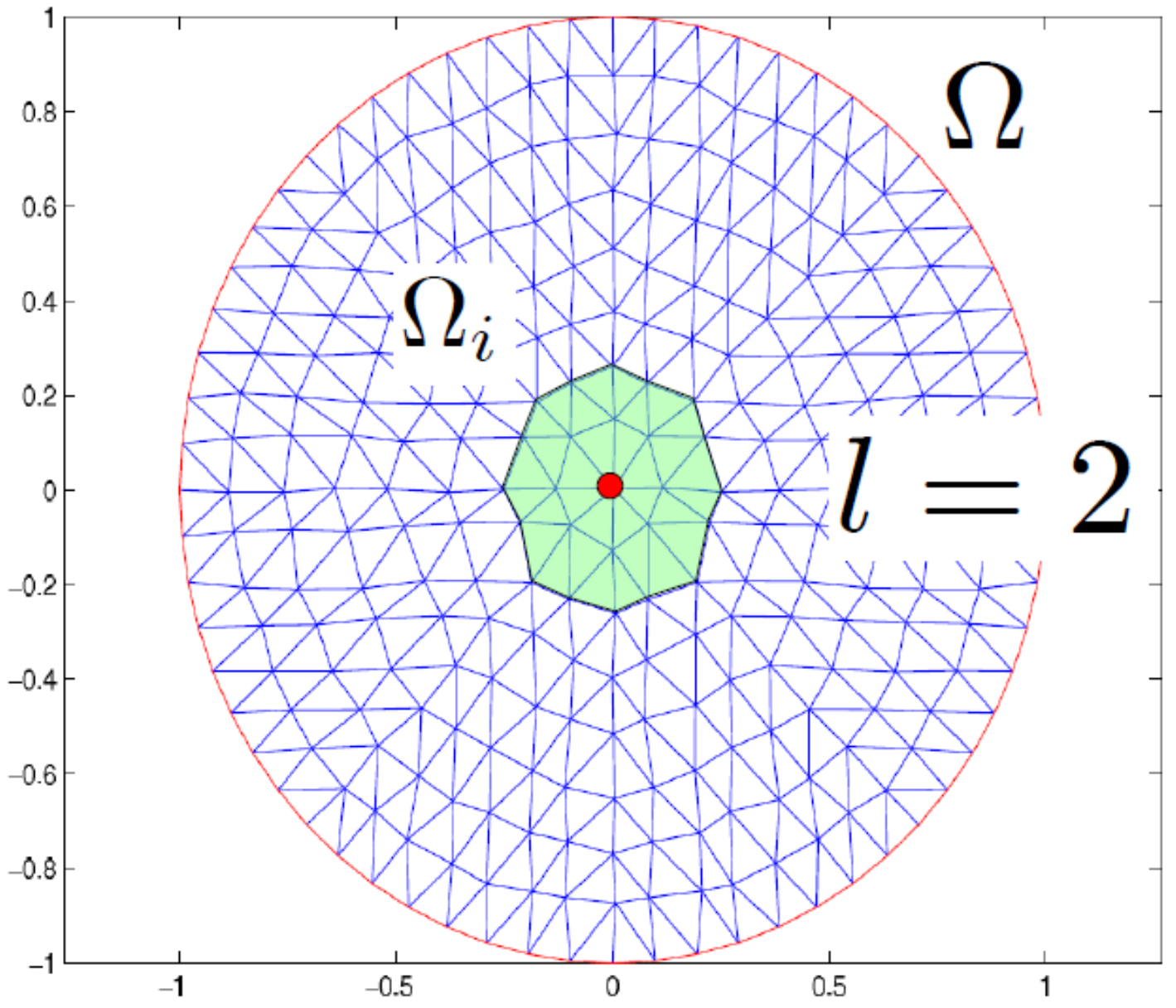} \label{fig:layer2}
		}
		\subfigure[$\Omega_i$ with $l=3$]{
			\includegraphics[width=0.2\textwidth]{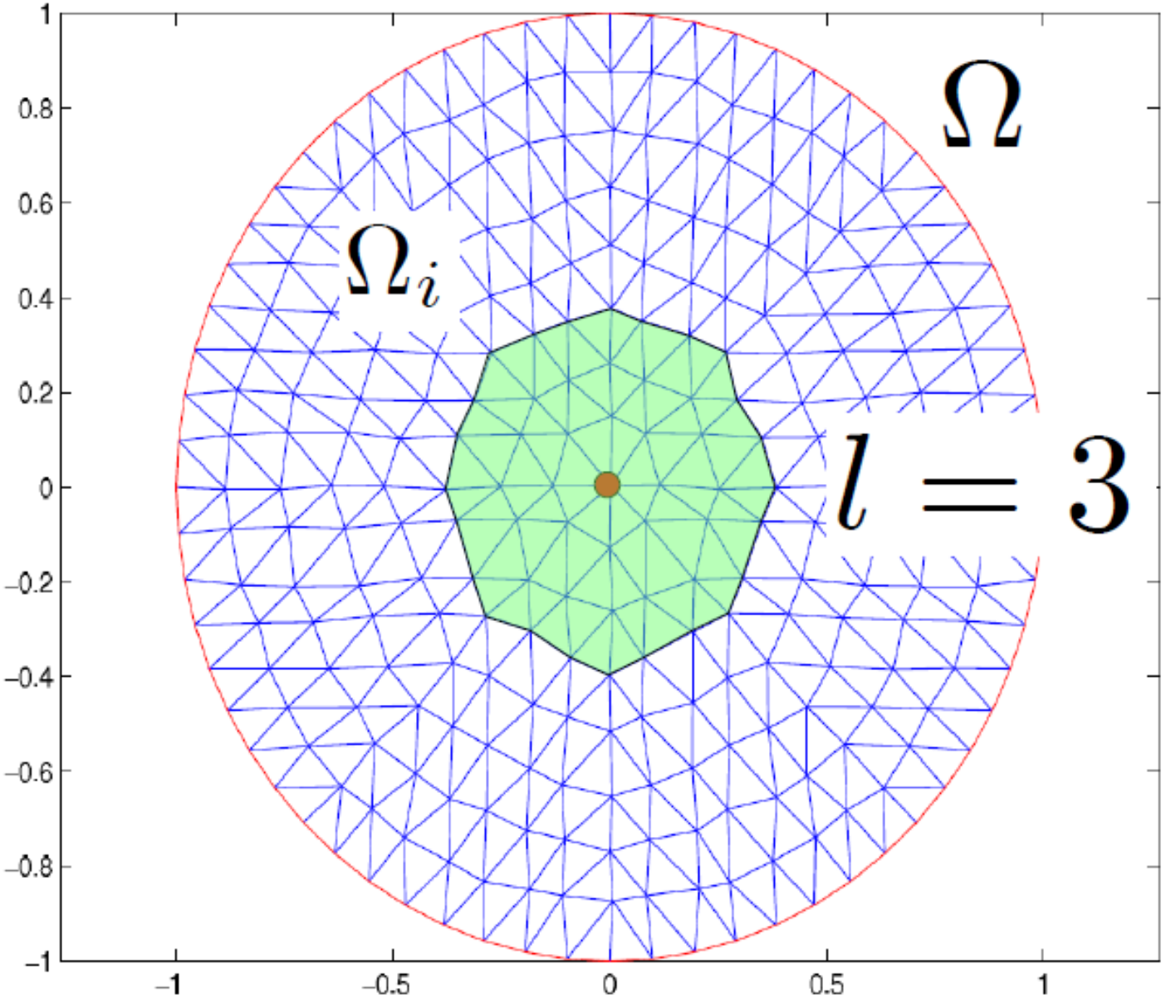} \label{fig:layer3}
		}
		\caption{Mesh details.}
		\label{fig:meshdetails}
	\end{center}
\end{figure}

\section{Numerical implementation and experiments}\label{sec:numexp}

\subsection{Implementation}
Our method is, in its formal description, meshless in the sense that only the locations of the points $(x_j)_{j\in \cN}$ are needed to construct the rough polyharmonic splines $\phi_i$ or $\phi_i^{\loc}$ as solutions of  the minimization problem \eqref{eqn:convexopt}
or \eqref{eqn:convexoptlocal}.

For numerical implementation, we need to formulate and solve the problem in the
full discrete setting. Although, for the sake of conciseness we have provided an error analysis only in the continuum setting, this analysis can be extended to the fully discrete setting without major difficulties.

We will now describe the implementation of our method in the discrete setting.
Let $\mathcal{T}_H$ be a tessellation of $\Omega$ of resolution $H$ with nodes $\{x_i\}$.
Assume
that $h\ll H$ and that $\mathcal{T}_h$ is obtained by repeated refinements of $\mathcal{T}_H$. Note that the accuracy of our method is independent from the regularity of  $\mathcal{T}_H$ or the aspect ratios of its elements (triangles for $d=2$).

Let $W_h(\Omega)$ be the set of piecewise linear functions on $\mathcal{T}_h$ with zero Dirichlet boundary condition on
$\partial \Omega$. For $\forall u\in W_h(\Omega)$, we can define $g_u$, which approximates
$-\diiv a \nabla u$ on $\Omega$. $g_u$ is piecewise constant on the dual mesh of $\mathcal{T}_h$, and can
be defined by the following finite volume formulation,
\begin{equation}
g_u(y_i)=\frac{1}{|\V_i|} \int_{\Omega} \nabla 1_{\V_i} a \nabla u
\end{equation}
where $\V_i$ is the dual Voronoi cell associated with the node $y_i$ of $\mathcal{T}_h$ and $1_{\V_i}$ is the characteristic function which equals to one on $\V_i$ and zero elsewhere.

Subfigure
\ref{fig:mesh} shows the primal (finite element) mesh and its dual mesh used for the numerical implementation in this paper: the red dots are the nodes $x_i$ of the triangulation $\T_H$. In Subfigure
\ref{fig:mesh},  the red square around each $y_i$ is the Voronoi cell $\V_i$ associated with the node $y_i$, and the dual mesh is composed of such cells.

With this definition, the local basis $\phi_{i,h}^\loc$ can be readily computed through the variational
formulation \eqref{eqn:fdlocalbasis}.

\begin{equation}\label{eqn:fdlocalbasis}
\begin{cases}
\text{Minimize }  \int_{\Omega}|g_\phi|^2\dx\\
\text{Subject to } \begin{cases} \phi\in W_h(\Omega_i), \phi=0 \text{ on }\pp\Omega_i\cup\Omega\backslash\Omega_i \\
\phi(x_j)=\delta_{ij}, \,j\in\cN.
\end{cases}
\end{cases}
\end{equation}

\begin{figure}[tp]
	\begin{center}
		\subfigure[$\phi_i$]{
			\includegraphics[width=0.4\textwidth]{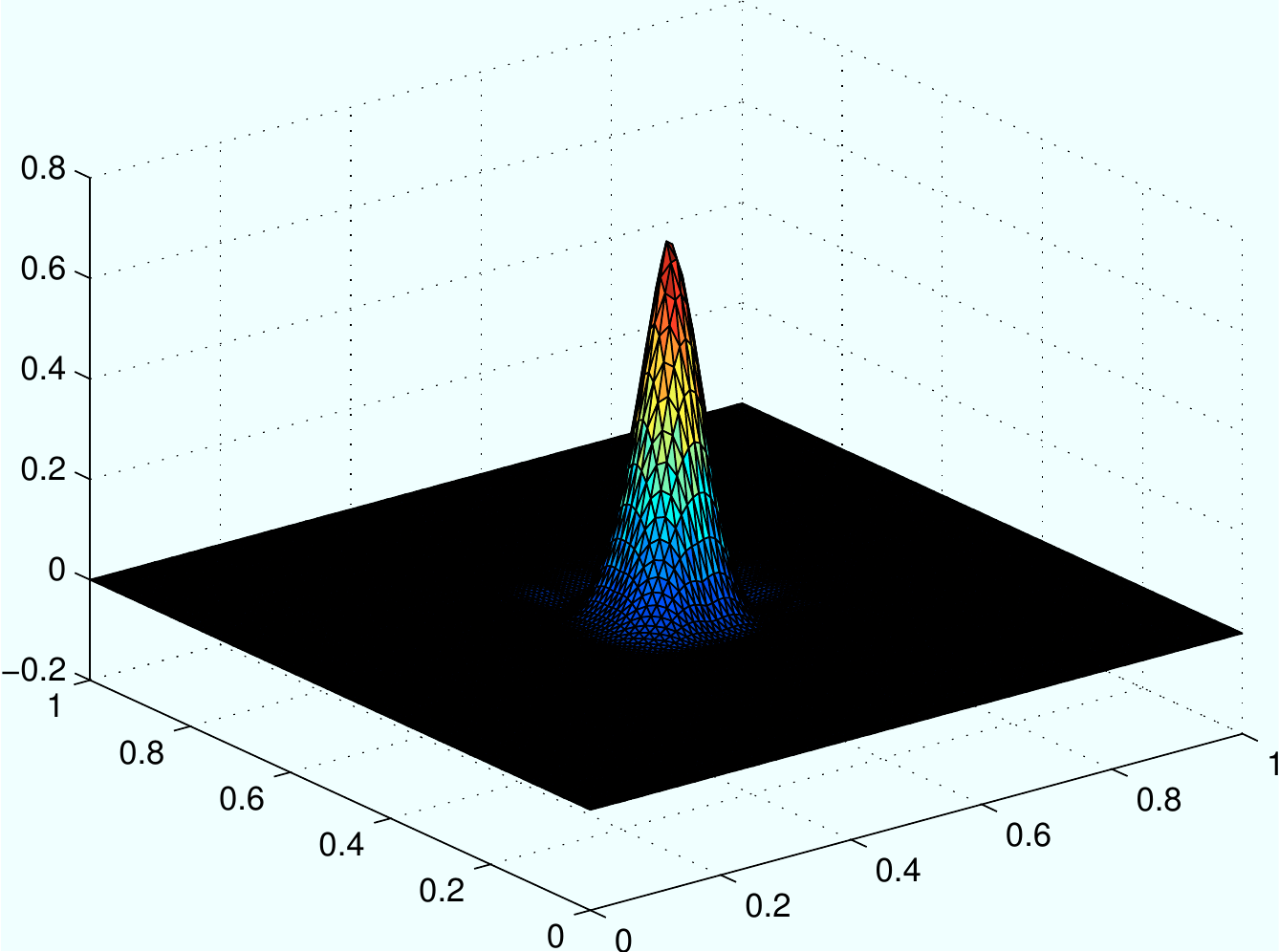}\label{fig:globalbasis2d}
		}
		\subfigure[Slice of $\phi_i$ along $x$ axis]{
			\includegraphics[width=0.4\textwidth]{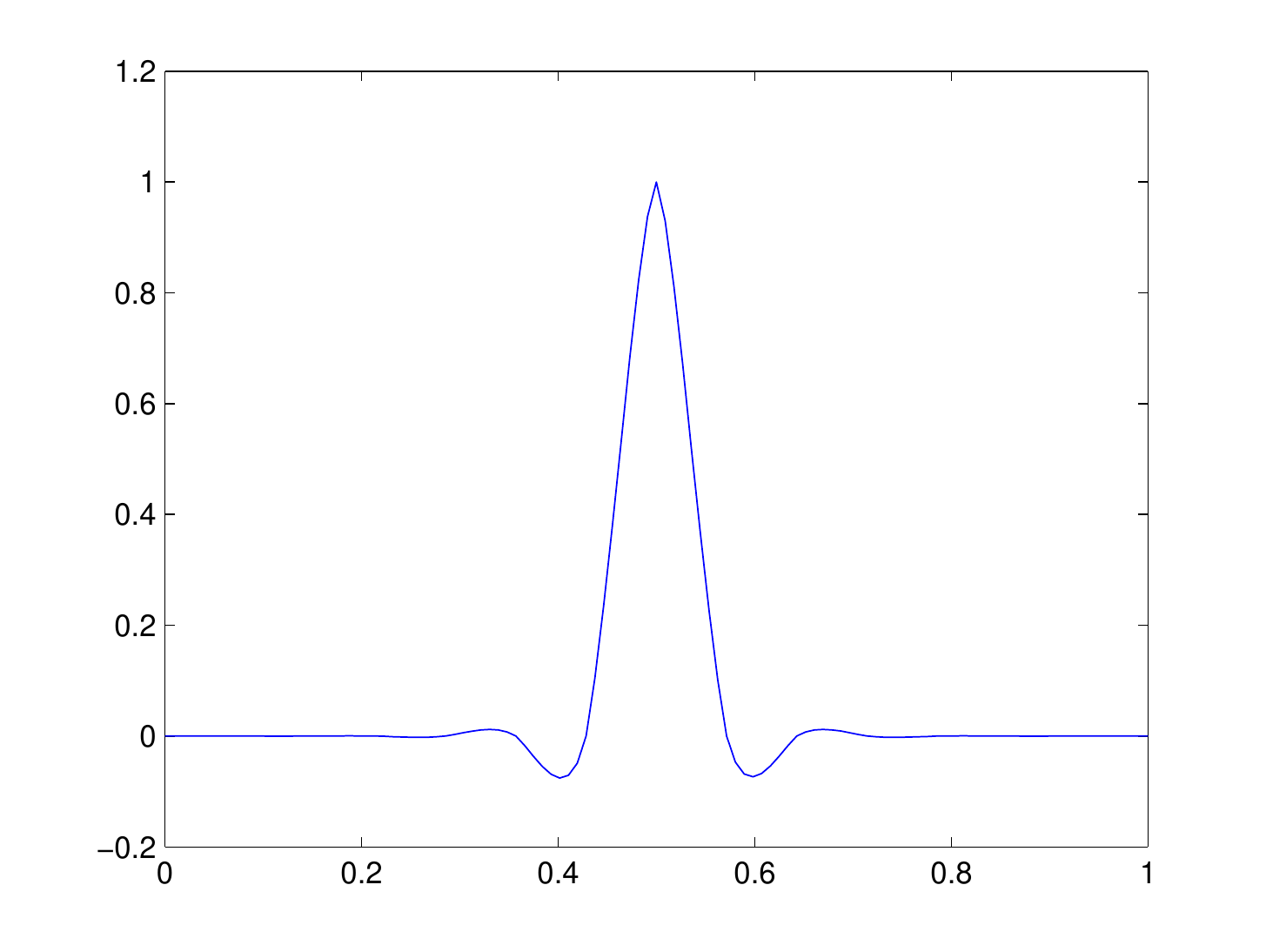}\label{fig:globalbasis2dslice}
		}
		\subfigure[$\|u-u^{H,\loc}\|$ in $\L^2$, $\H^1$, $\L^\infty$ norms with respect to the number of layers $l$]{
			\includegraphics[width=0.90\textwidth]{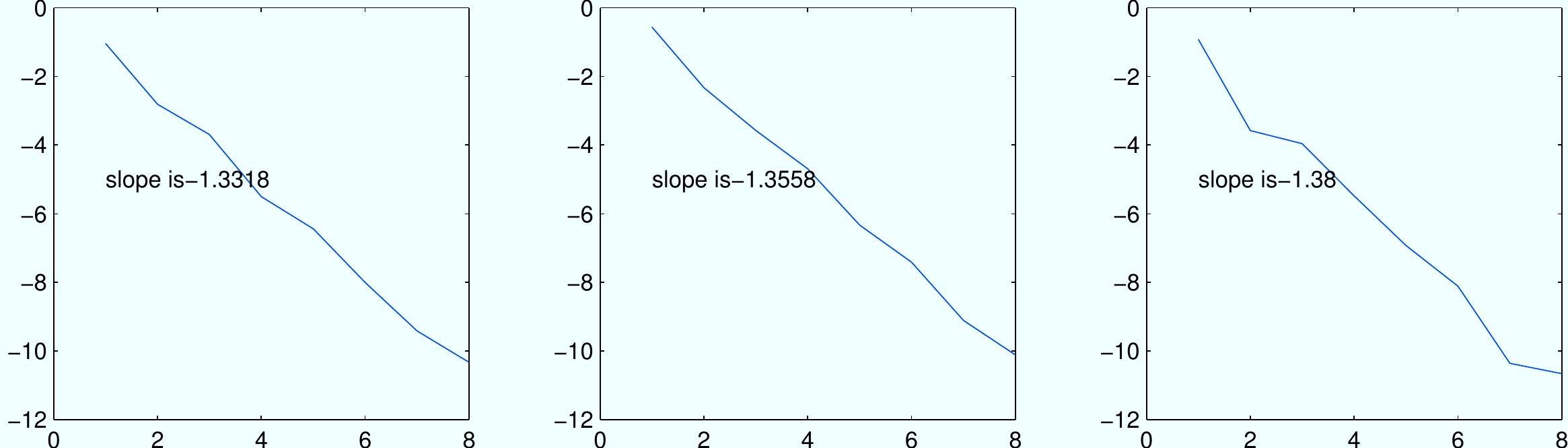}\label{fig:decayratelocalbasis2d}
		}
		\subfigure[$\|u-u^{H,\loc}\|_{L^2(\Omega)}$]{
			\includegraphics[width=0.40\textwidth]{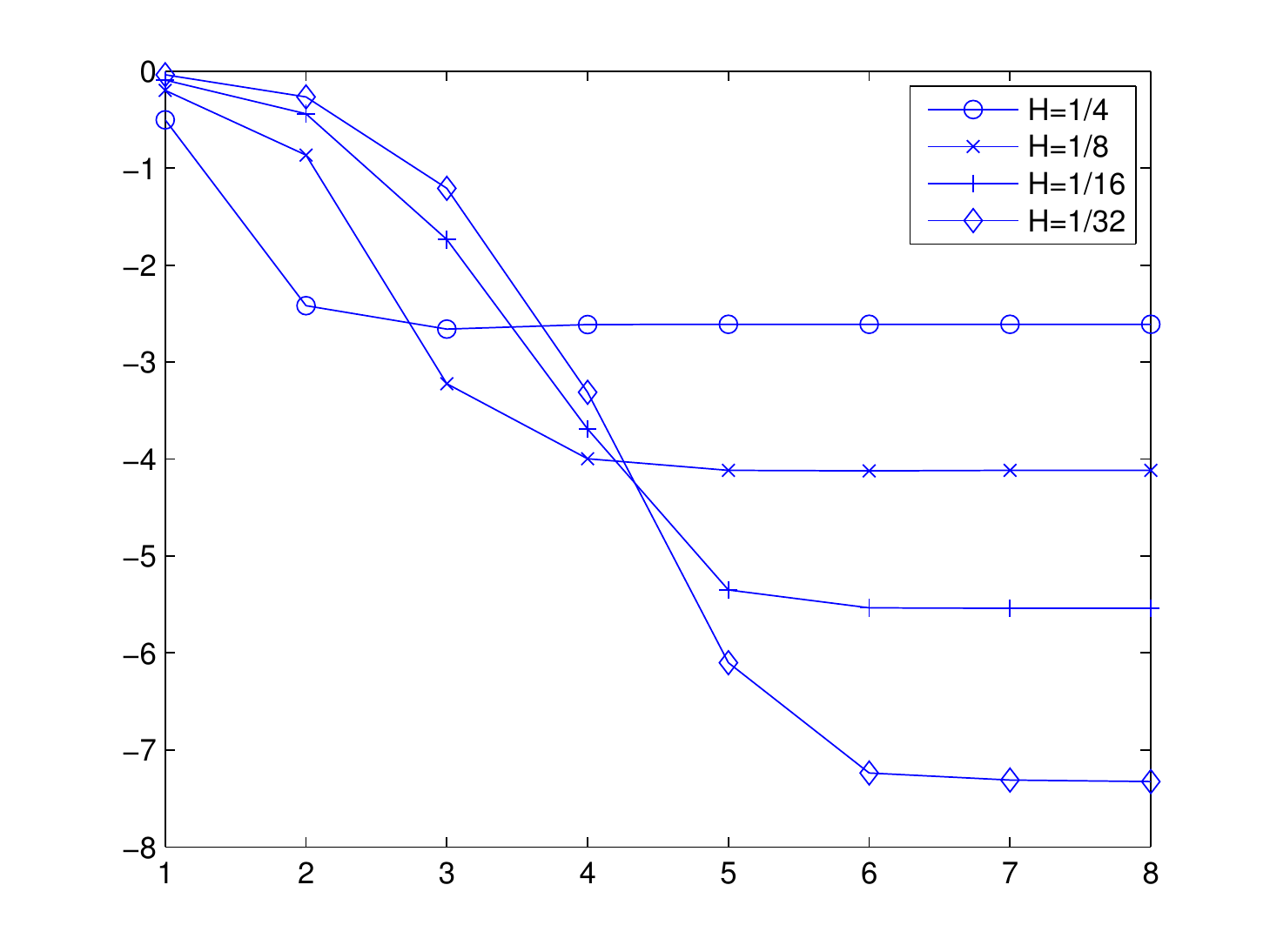} \label{fig:errordecay2da}
		}
		\subfigure[$\|u-u^{H,\loc}\|_{\H^1(\Omega)}$]{
			\includegraphics[width=0.40\textwidth]{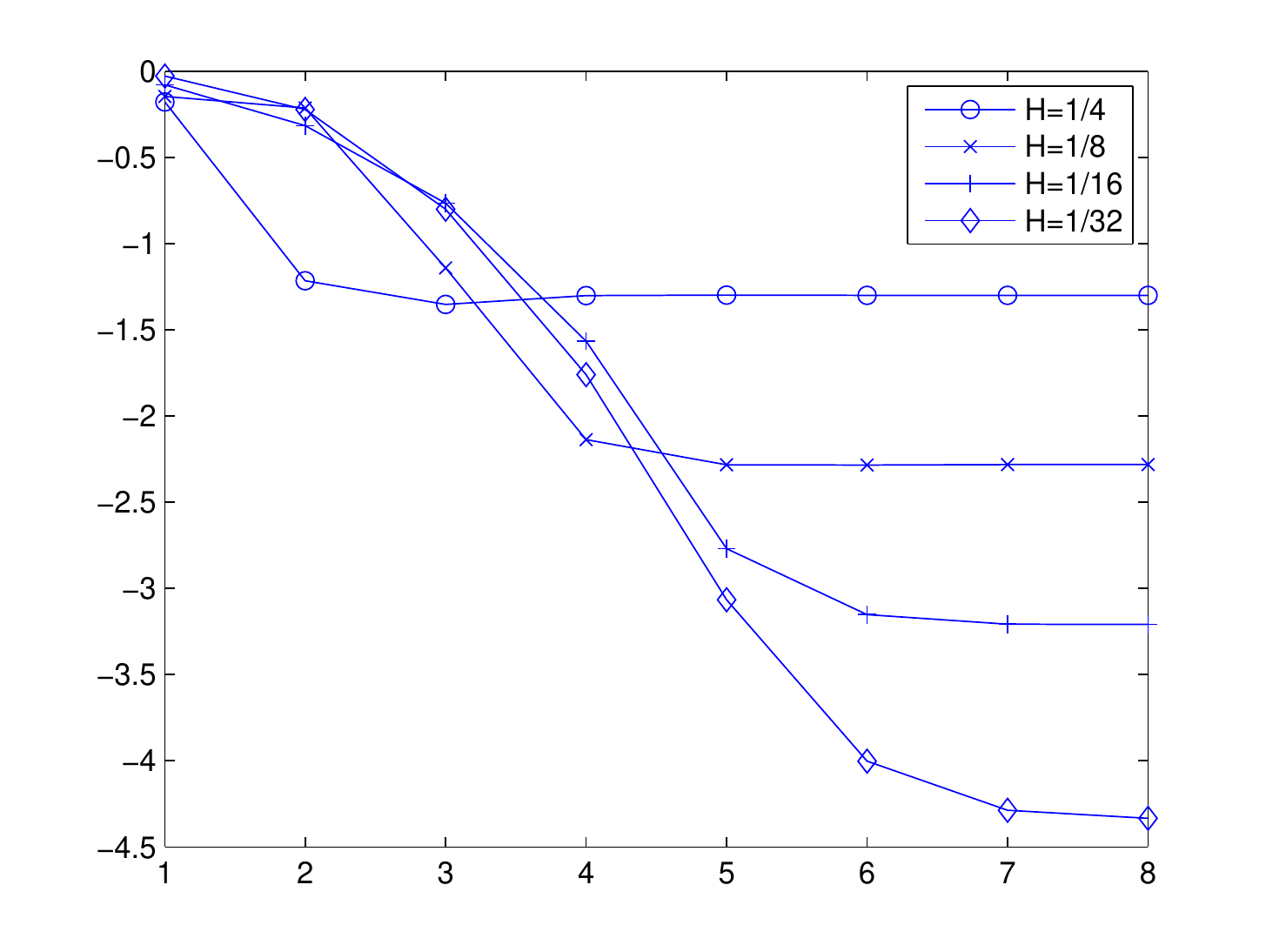} \label{fig:errordecay2db}
		}
		\caption{Two dimensional example \eqref{ex1}.}
		\label{fig:2d}
	\end{center}
\end{figure}

\subsection{A two dimensional example.}\label{subsec:2dex}
In this example  we consider $\Omega=(0,1)\times (0,1)$ and
\begin{equation}\label{ex1}
\begin{split}
a(x):=&\frac{1}{6}(\frac{1.1+\sin(2\pi x/\epsilon_{1})}{1.1+\sin(2\pi
y/\epsilon_{1})}+\frac{1.1+\sin(2\pi y/\epsilon_{2})}{1.1+\cos(2\pi
x/\epsilon_{2})}+\frac{1.1+\cos(2\pi x/\epsilon_{3})}{1.1+\sin(2\pi
y/\epsilon_{3})}+\\
&\frac{1.1+\sin(2\pi y/\epsilon_{4})}{1.1+\cos(2\pi
x/\epsilon_{4})}+\frac{1.1+\cos(2\pi x/\epsilon_{5})}{1.1+\sin(2\pi
y/\epsilon_{5})}+\sin(4x^{2}y^{2})+1),
\end{split}
\end{equation}
where $\epsilon_{1}=\frac{1}{5}$,$\epsilon_{2}=\frac{1}{13}$,$\epsilon_{3}=
\frac{1}{17}$,$\epsilon_{4}=\frac{1}{31}$,$\epsilon_{5}=\frac{1}{65}$.

Note that $a$, as defined by \eqref{ex1}, corresponds  to the one used in  \cite{OwZh:2007a} (Example 1 of Section 3) and \cite{MiYu06}.
The number of coarse nodes is $N=N_c\times N_c$ with $N_c=32$. The (regular) coarse mesh $\cT_H=\cT^0$ is obtained by first (uniformly) subdividing $\Omega$ into $N_c\times N_c$ rectangles, then partitioning each rectangle into two triangles along the direction $(1,1)$. We can further
refine the mesh to obtain $\cT^1$, $\cT^2$, ..., $\cT^k$. Global equations are solved on a fine triangulation with $66049$ nodes and $131072$ triangles.
We compute localized elements $\phi_i^{\loc}$ on localized sub-domains $\Omega_i^l$ defined by adding $l$ layers of coarse triangles around $x_i$. More precisely $\Omega_i^1$ is the union of triangles sharing $x_i$ as a node and $\Omega_i^{l+1}$ is the union of $\Omega_i^l$ with coarse triangles sharing a node with a triangle contained in $\Omega_i^l$. We refer to subfigures \ref{fig:layer1}, \ref{fig:layer2} and \ref{fig:layer3} for an illustration of $\Omega_i^l$ with $l=1,2,3$.

We refer to Figure \ref{fig:2d} for the results of our numerical experiment.
Subfigure \ref{fig:globalbasis2d} shows $\phi_i$ (one element of the global basis in dimension two).
Subfigure \ref{fig:globalbasis2dslice} shows a slice of $\phi_i$ along the $x$-axis.
Subfigure
\ref{fig:decayratelocalbasis2d} shows $\|\phi_i-\phi_i^{\loc}\|$ in $\L^2$, $\H^1$, and $\L^\infty$ norms in the log scale as a function of the number of layers $l=1,\dots,8$ used in the localization. Note that   $\|\phi_i-\phi_i^{\loc}\|$  decays exponentially fast with $l$
and that the decay rate is nearly independent from the choice of the norm.
Subfigurea \ref{fig:errordecay2da} and \ref{fig:errordecay2db} show $\|u-u^{H,\loc}\|_{L^2(\Omega)}$ and $\|u-u^{H,\loc}\|_{\H^1(\Omega)}$ in log-scale as a function of the number of layers $l=1,\dots,8$ (defining the size of localized sub-domains) for different values of (coarse) mesh resolution $H=1/4, 1/8, 1/16, 1/32$. Note that
as the number of layers increases, $\|u-u^{H,\loc}\|$  decreases first and then quickly saturates around $O(H)$.

\begin{figure}[tp]
	\begin{center}
		\subfigure[$u(x,T)$ for $T=1$]{
			\includegraphics[width=0.40\textwidth]{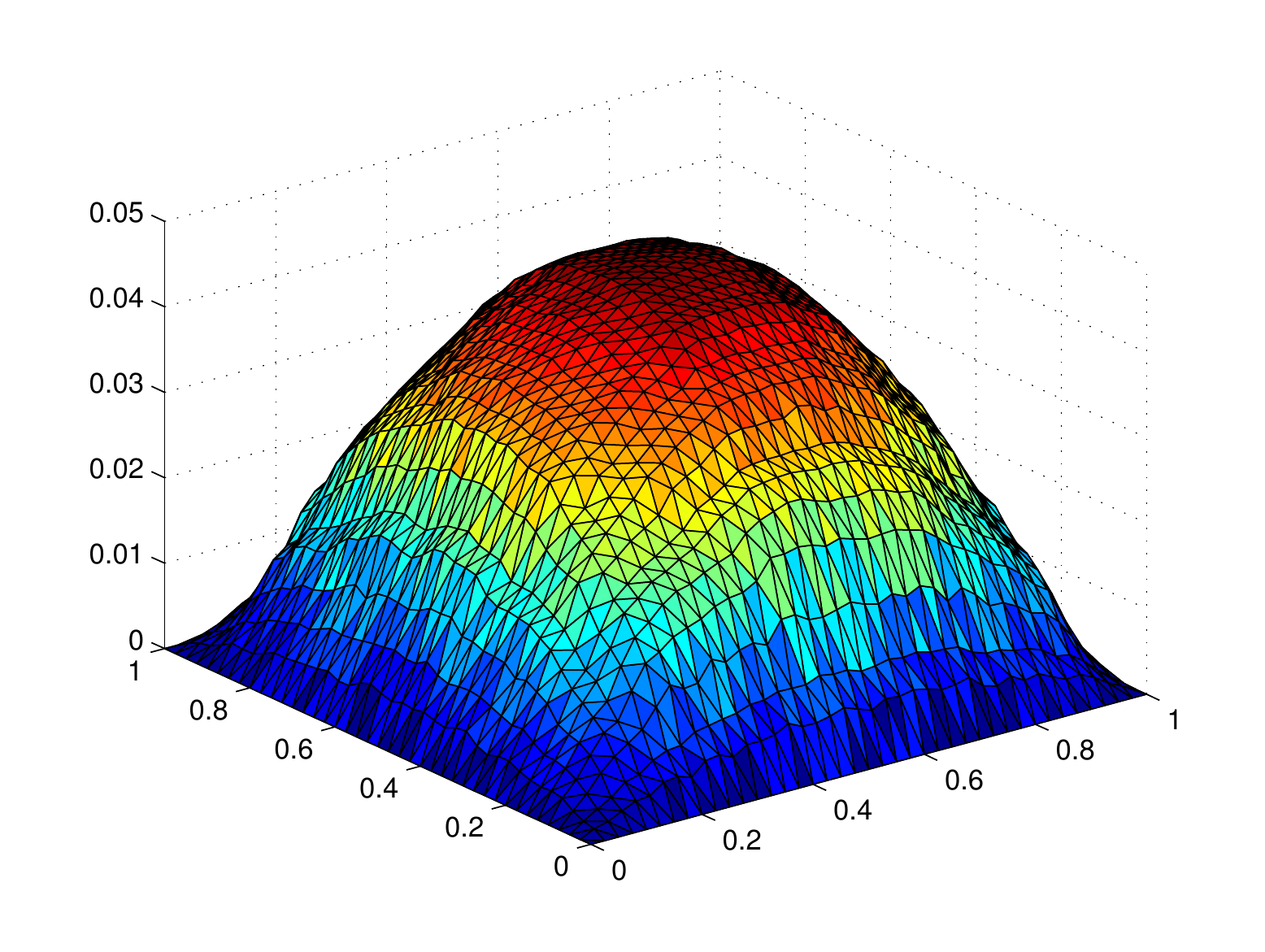}\label{fig:example2-2a}
		}
		\subfigure[$u^{H,\loc}(x,T)$ for $T=1$]{
			\includegraphics[width=0.40\textwidth]{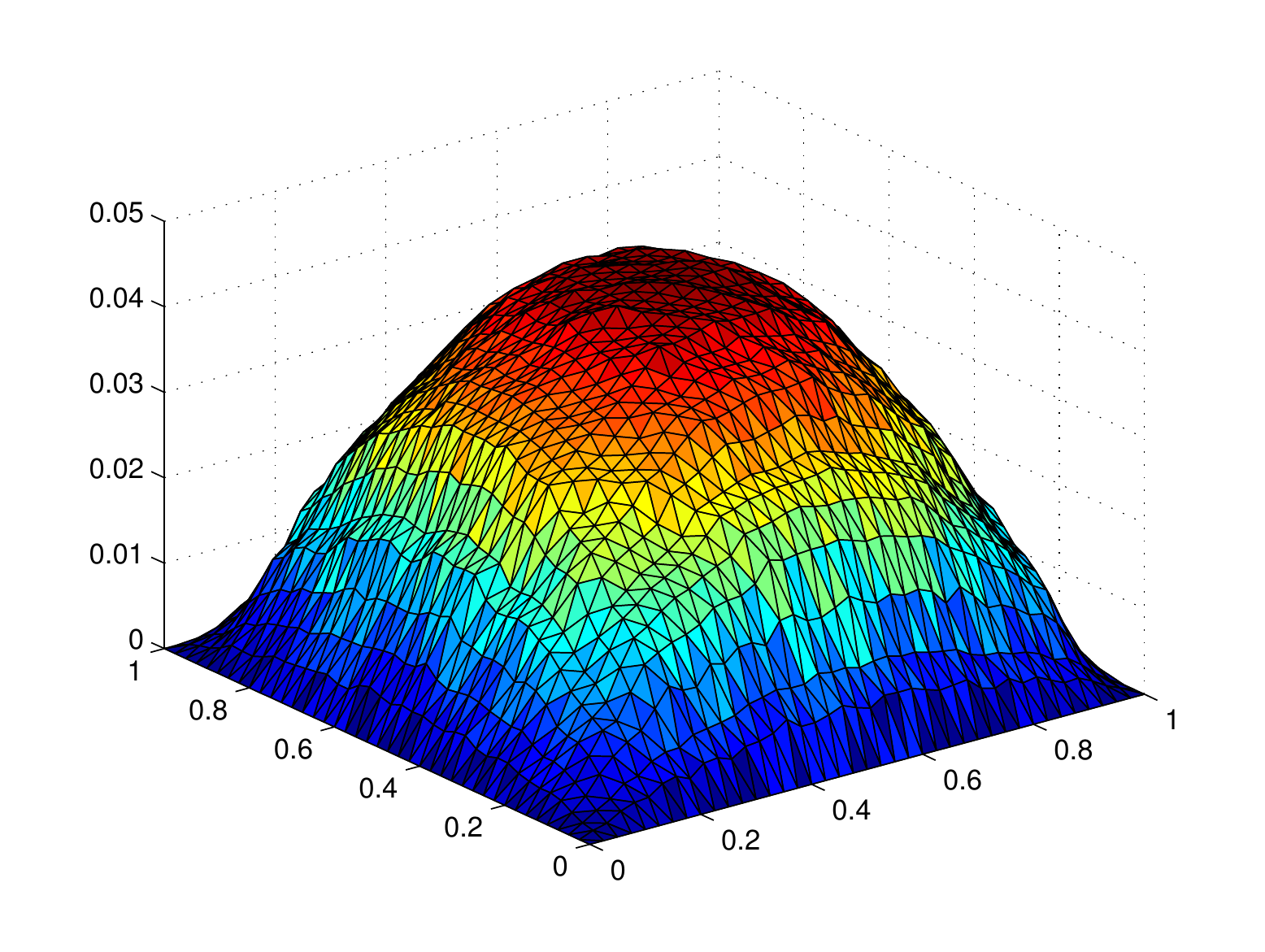}\label{fig:example2-2b}
		}
		\subfigure[$\|u-u^{H,\loc}(.,T)\|$ in $\L^2$, $\H^1$, and $\L^\infty$ in log scale as a function of $l$]{
			\includegraphics[width=0.80\textwidth]{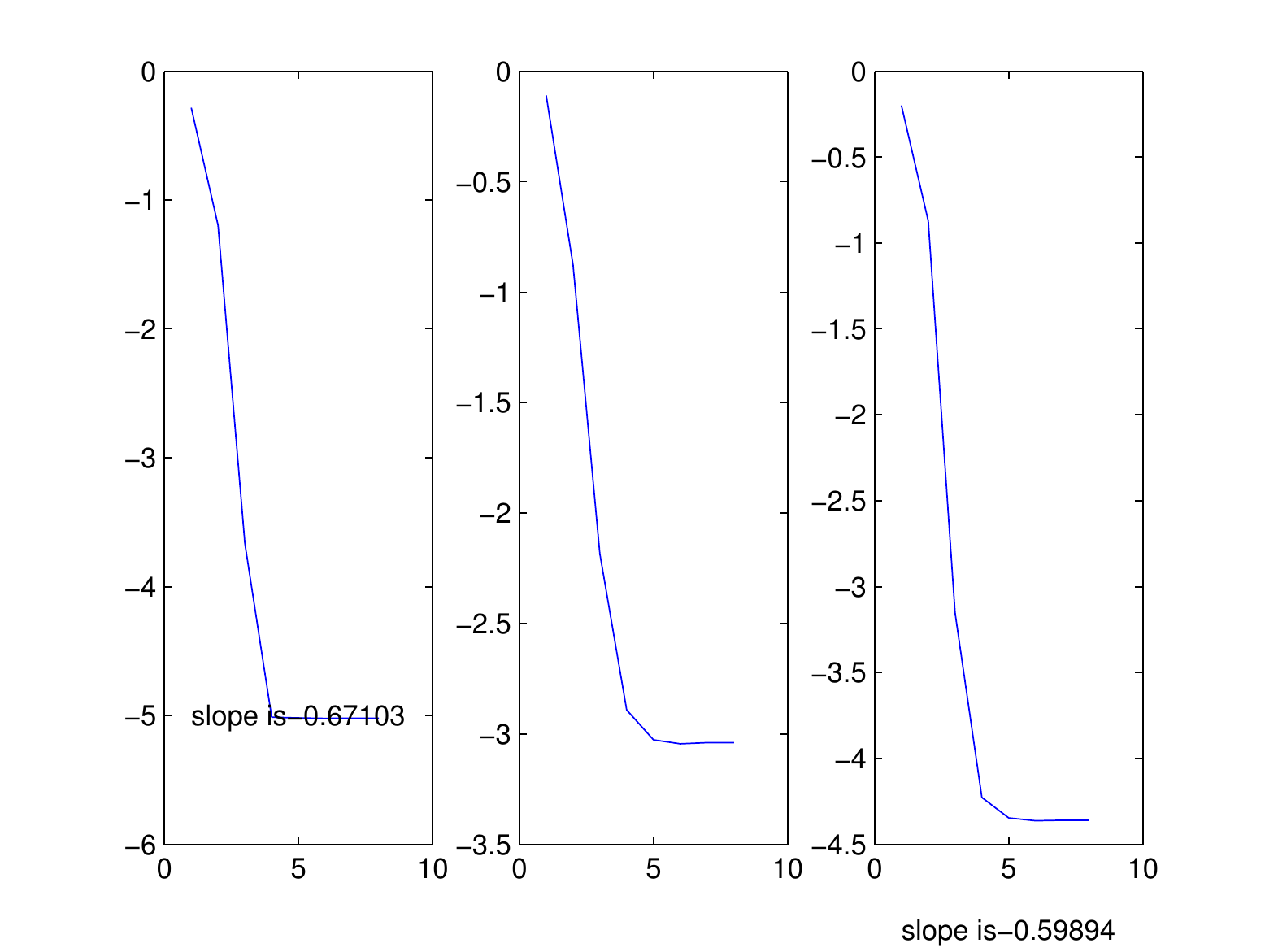}\label{fig:example2-1}
		}
		\caption{Wave equation of Subsection \ref{subsec:wave}}
		\label{fig:wave}
	\end{center}
\end{figure}

\subsection{Wave equation}\label{subsec:wave}
In this example we consider $a$ and $\Omega$ as defined in Subsection \ref{subsec:2dex} and we
 compute the solutions of wave equation \eqref{huperboliccalarproblem0} up to time $T=1$. The initial condition is
$u(x,0)=0$ and $u_t(x,0)=0$. The boundary condition is $u(x,t)=0$, for $x\in\partial\Omega$. The density $\rho$
is uniformly equal to one and we choose $g=\sin(\pi x)\sin(\pi y)$.
 Subfigures \ref{fig:example2-2a} and \ref{fig:example2-2b}  show  $u(x,T)$ and $u^{H,\loc}(x,T)$ with $T=1$. $u^{H,\loc}(x,T)$ is computed over the approximate solution space spanned by the localized basis elements $\phi_i^{\loc}$ obtained by choosing $\Omega_i$ as the union of 3 layers of triangles around $x_i$ (hence, we choose $l=3$). The resolution of the fine mesh is $h=1/80$. The resolution of the coarse mesh is $H=1/10$.
Subfigure \ref{fig:example2-1} shows approximation error $\|u-u^{H,\loc}(.,T)\|$ in $\L^2$, $\H^1$, and $\L^\infty$ norms in the log scale as a function of the number of layers $l=1,\dots,8$.

\begin{figure}[tp]
	\begin{center}
		\subfigure[$\phi_i^{\loc}$]{
			\includegraphics[width=0.40\textwidth]{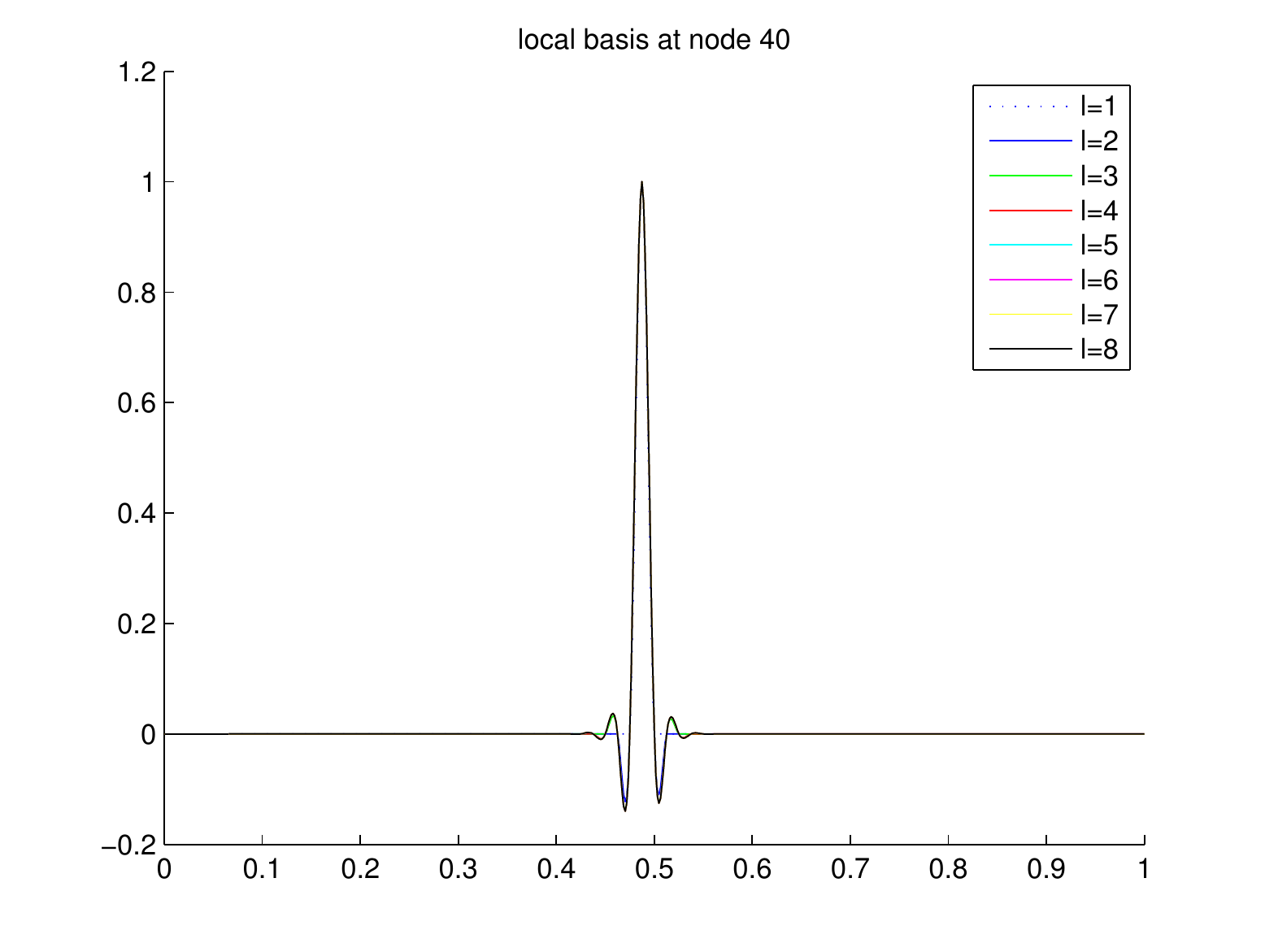}\label{fig:localbasis40}
		}
		\subfigure[$\phi_i$ in log scale]{
			\includegraphics[width=0.40\textwidth]{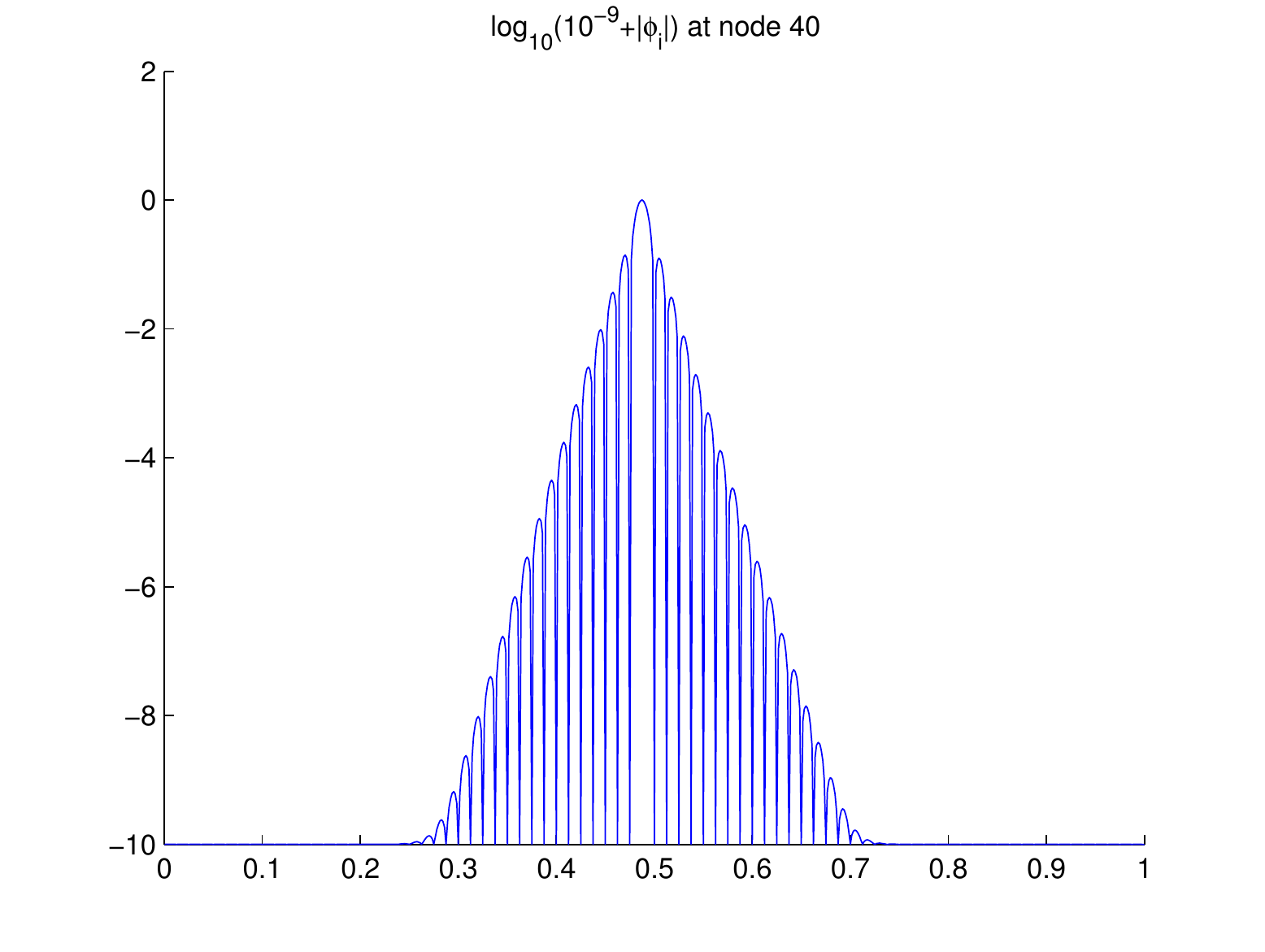}\label{fig:logglobalbasis40}
		}
		\subfigure[$\phi_i-\phi_i^{\loc}$]{
			\includegraphics[width=0.40\textwidth]{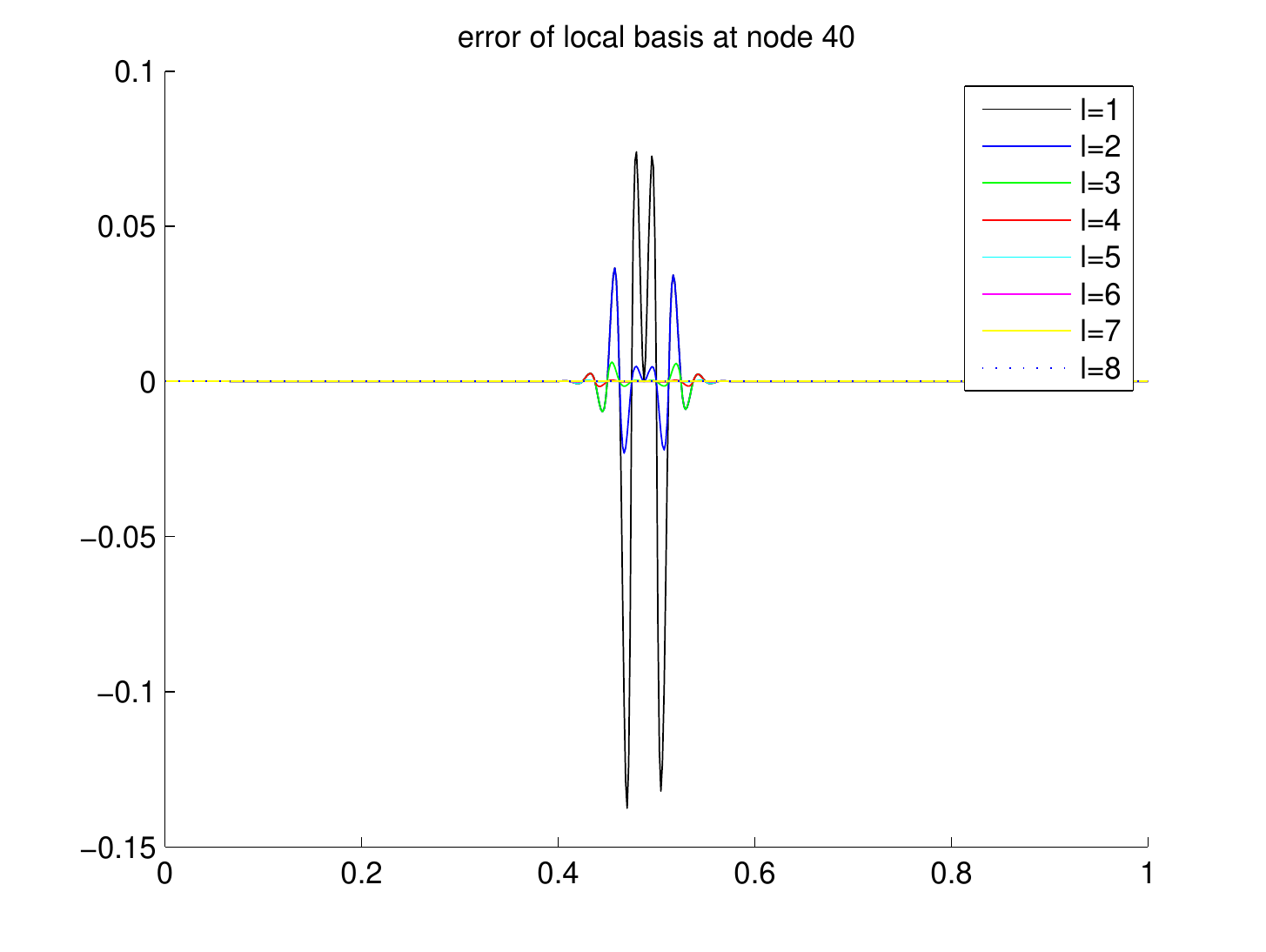}\label{fig:errlocalbasis40}
		}
		\subfigure[$-\diiv (a \nabla \phi_i)$]{
			\includegraphics[width=0.40\textwidth]{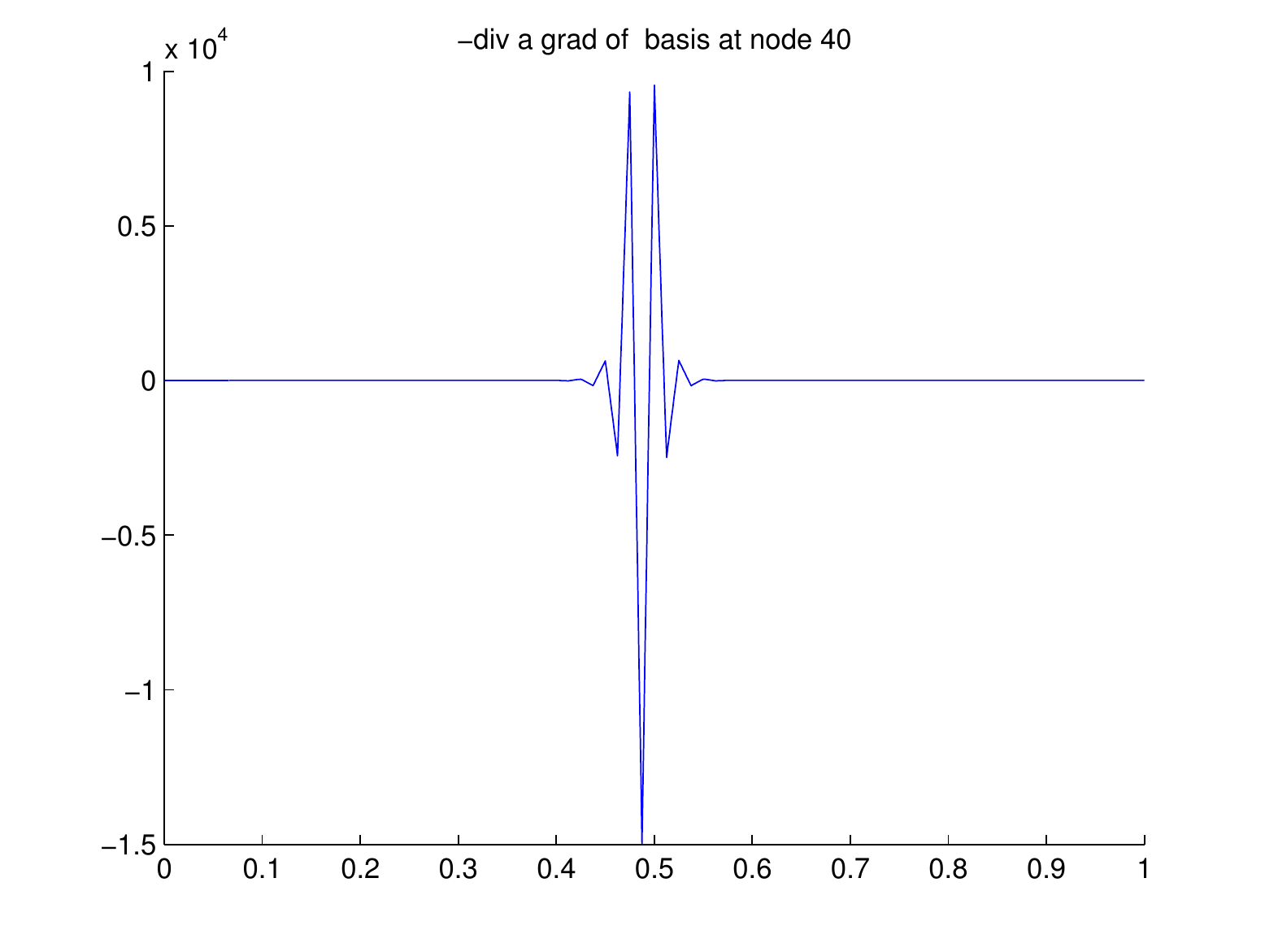}\label{fig:divagradglobalbasis40}
		}
		\subfigure[$-\diiv (a \nabla \phi_i)$ in log scale]{
			\includegraphics[width=0.40\textwidth]{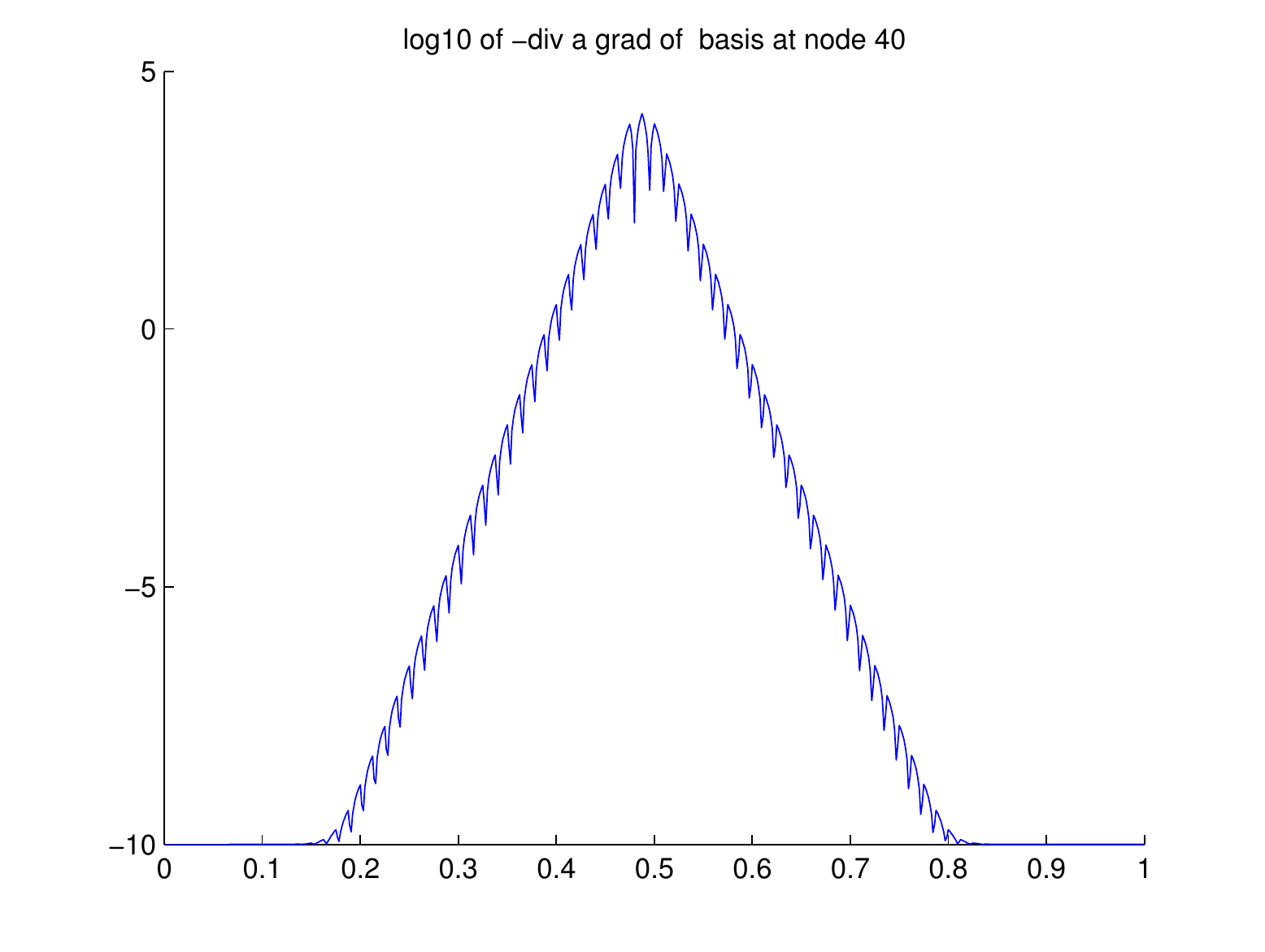}\label{fig:logdivagradglobalbasis40}
		}
		\subfigure[Matrix $\Theta$]{
			\includegraphics[width=0.40\textwidth]{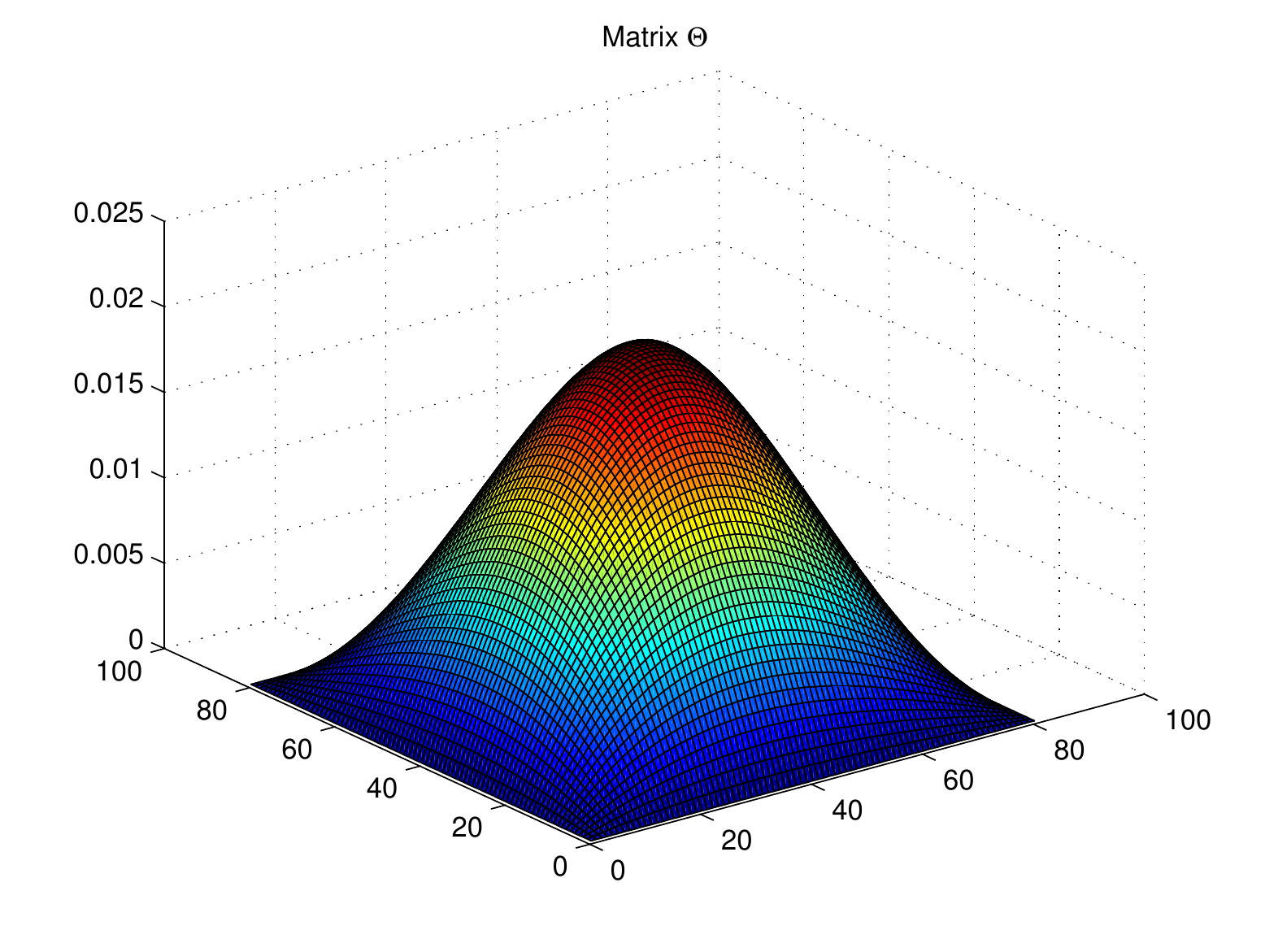}\label{fig:matrixTheta}
		}
		\subfigure[Matrix $P$]{
			\includegraphics[width=0.40\textwidth]{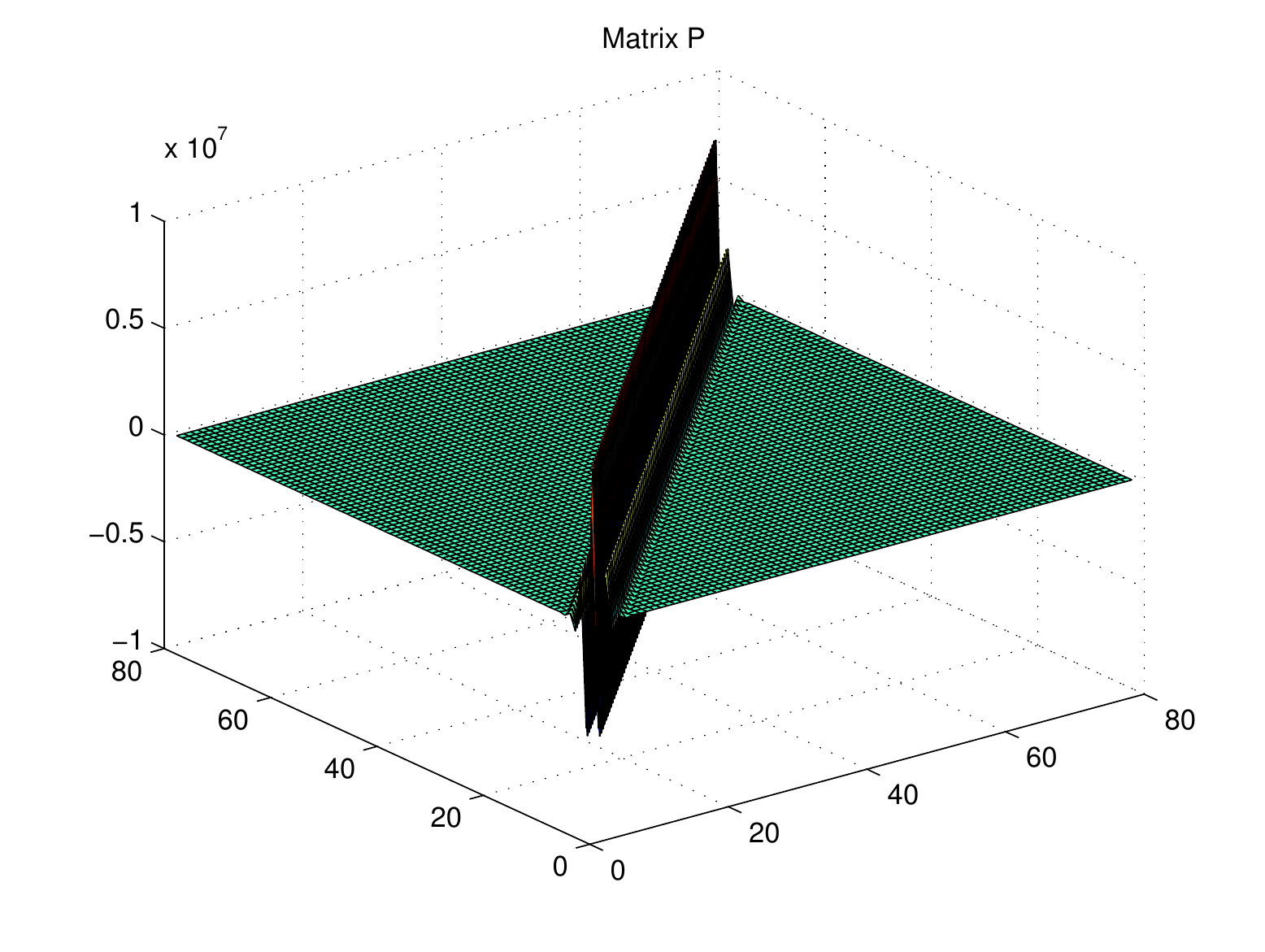}\label{fig:matrixP}
		}
		\subfigure[Matrix $P$ in log scale]{
			\includegraphics[width=0.40\textwidth]{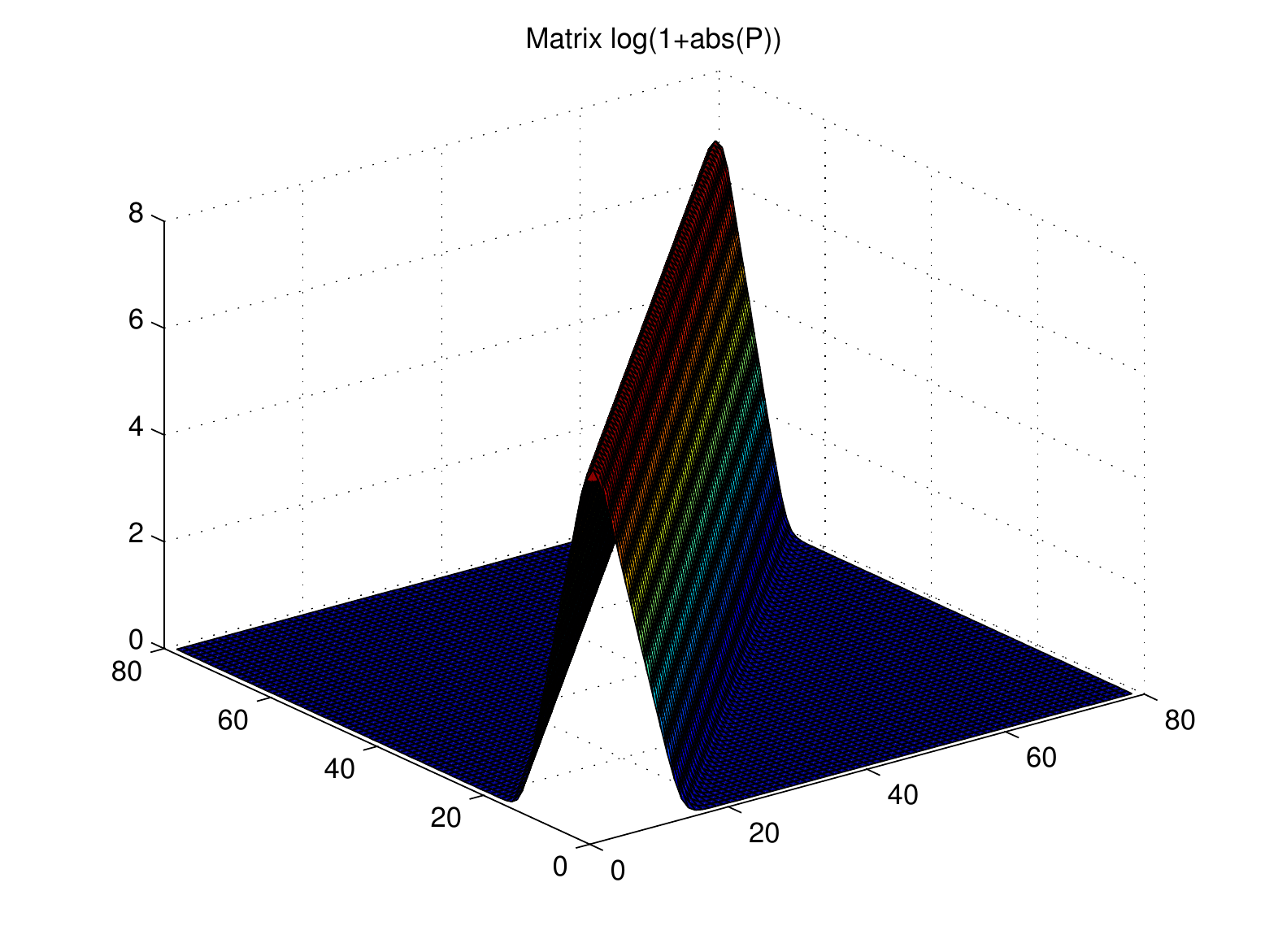}\label{fig:matrixlogP}
		}
		\caption{One dimensional example \eqref{eq:ain1d}.}
		\label{fig:1d}
	\end{center}
\end{figure}

\subsection{A one dimensional example}
In this example we consider $\Omega=(0,1)$ and
\begin{equation}\label{eq:ain1d}
 a(x) := 1 + \frac12\sin\big(\sum^K_{k=1}k^{-\alpha}(\zeta_{1k}\sin(kx)+\zeta_{2k}\cos(kx))\big)
\end{equation}
where $\{\zeta_{1k}\}$ and $\{\zeta_{2k}\}$ are two independent random vectors with independent entries uniformly distributed in
$[-\frac12,\frac12]$.
This example is taken  from \cite{HoWu:1997, MiYu06} and it has no scale separation. Note that
$a$  has an algebraically decaying (Fourier energy) spectrum, i.e.,
\begin{equation}
\<|\hat{a}(k)|^2\>\simeq |k|^{-\alpha},
\end{equation}
where $\<\cdot\>$ denotes ensemble averaging, and $\hat{a}(k)$ is the Fourier transform of the coefficients $a(x)$.
We choose $K=20$ in the numerical experiment.

The coarse nodes $(x_j)_{j\in \cN}$ correspond to $N=d_c = 80$  points uniformly distributed in $(0,1)$.
More-precisely, those points correspond to the nodes of a coarse mesh of $(0,1)$ and the distance between
 two successive points is $H=1/81$.  The fine mesh is then obtained by refining (bisecting) the coarse mesh
3 times. Local bases are computed on intervals with $l$ layers, where $l=1,\dots,8$. More precisely our localized sub-domains (over which the elements $\phi_i^{\loc}$ are computed) are $\Omega_i=(x_i-l/81,x_i+l/81)\cap\Omega$ for $l=1,\dots,8$.

We refer to Figure \ref{fig:1d} for the results of our numerical implementation.
Subfigure \ref{fig:localbasis40} shows the local basis $\phi_i^{\loc}$ for the node $i=40$ for various degrees of localization (i.e., for
$l=1,\dots,8$). Subfigure \ref{fig:logglobalbasis40} shows the global basis in the log scale  (i.e., $\log_{10}(10^{-9}+|\phi_{40}|)$) and illustrates the exponential decay of $\phi_i(x)$ away from $x_i$.
Subfigure \ref{fig:errlocalbasis40} shows $\phi_i-\phi_i^{\loc}$, the difference between the global and the local basis, at node $i=40$,
for various degrees of localization (i.e., for
$l=1,\dots,8$), and  illustrates the decay of this difference with respect to the number of layers $l$ defining $\Omega_i$.
Subfigure \ref{fig:divagradglobalbasis40} shows  $-\diiv (a \nabla\phi_i)$ for $i=40$. Subfigure \ref{fig:logdivagradglobalbasis40} shows  $-\diiv (a \nabla\phi_i)$ for $i=40$ in the log scale (i.e., $\log_{10}(10^{-9}+|\diiv a \nabla \phi_i|)$). Subfigures
 \ref{fig:divagradglobalbasis40} and \ref{fig:logdivagradglobalbasis40} illustrate the exponential decay of  $-\diiv (a \nabla \phi_i)(x)$  away from $x_i$.
Subfigure \ref{fig:matrixTheta} shows the matrix $\Theta_{i,j}$, as defined in \eqref{eq:theta:edf}, as a function of $(i,j)$ (two dimensional surface plot).
Similarly Subfigure  \ref{fig:matrixP} shows the matrix $P_{i,j}$, as defined in \eqref{eq:defP}, as a function of $(i,j)$.
Subfigure \ref{fig:matrixlogP} shows $P$ in the log scale and illustrates the exponential  decay of the entries of $P$  away from the diagonal.

\paragraph{Analogy with polyharmonic splines.}
Note that as $H\downarrow 0$ the matrix $\Theta$ shown in Subfigure \ref{fig:matrixTheta} converges towards $\tau(x,y)$ the fundamental solution of the operator $L=\big(\diiv(a\nabla \cdot)\big)^2$. Note also that since $P^{-1}=\Theta$, $P$ is,  in fact (in analogy with \cite{Rabut:1990, Rabut1:1992, Rabut2:1992}), a discrete approximation of the operator $L$ (i.e. $\sum_{j=1}^N P_{i,j} u(x_j)$ can be thought of as an approximation of $L u(x_i)$). Note the localization of
$P$ (near its diagonal) as shown in  Subfigure \ref{fig:matrixP}.
Note also since $\phi_i(x)=\sum_{j=1}^N P_{i,j} \tau(x,x_j)$, $\phi_i$ is obtained by applying the discrete version of $L$  against the fundamental solution of $L$ and in that sense the elements $\phi_i$ or $\phi_i^{\loc}$ (see Subfigure \ref{fig:localbasis40}) are approximations of masses of Diracs.

\section{On numerical homogenization}\label{sec:numhom}
\subsection{Short overview}
As mentioned in the introduction the field of numerical homogenization concerns the numerical approximation of the solution space of \eqref{eqn:scalar} with a finite-dimensional space. This problem is motivated by the fact that standard methods (such as finite-element method with piecewise linear elements  \cite{BaOs:2000}) can perform arbitrarily badly for PDEs with rough coefficients such as \eqref{eqn:scalar}.
Although some numerical homogenization methods (described below) are directly inspired from classical homogenization concepts such as periodic homogenization and scale separation \cite{BeLiPa78}, oscillating test functions, $G$ or $H$-convergence and compensated compactness \cite{Mur78, Spagnolo:1968,Gio75} and localized cell problems \cite{PapanicolaouVaradhan:1981}, one of the main objectives of numerical homogenization is to achieve a numerical approximation of the solution space of \eqref{eqn:scalar} with arbitrary rough coefficients  (i.e., in particular, without the assumptions found in classical homogenization, such as scale separation, ergodicity at fine scales and $\epsilon$-sequences of operators). Numerical homogenization methods differ in (computational)  cost and assumptions required for accuracy and by now, the field of numerical homogenization has become large enough that it is not possible to give a complete review in this short paper. For the sake of completeness, we will quote below
the non exhaustive list of numerical homogenization methods discussed in \cite{OwZh:2011}:

-  The multi-scale finite element method \cite{HoWu:1997, Wu:2002, HouWu:1999, EfendievHou:2009} can be seen as a numerical generalization of this idea of oscillating test functions found in $H$-convergence. A convergence analysis for periodic media revealed a resonance error introduced by the microscopic boundary condition \cite{HoWu:1997, HouWu:1999}. An over-sampling technique was proposed to  reduce the resonance error \cite{HoWu:1997}.

- Harmonic coordinates play an important role in various homogenization approaches, both theoretical and numerical. These coordinates  were  introduced in \cite{Kozlov:1979} in the context of random homogenization. Next, harmonic coordinates have been used in one-dimensional and quasi-one-dimensional divergence form elliptic problems  \cite{BaOs:1983, BaCaOs:1994}, allowing for efficient finite
dimensional approximations. These coordinated have been used in \cite{BenarousOwhadi:2003, Owhadi:2003, Owhadi:2004} to establish the anomalous diffusion of stochastic processes evolving in media characterized by an infinite number of scales with no scale separation and unbounded contrast and provide a proof of Davies' conjecture   \cite{Owhadi:2003} on the behavior of the heat kernel in periodic media.
The connection of these coordinates with classical homogenization is made explicit in \cite{AllBri05} in the context of multi-scale finite element methods.
The idea of using particular solutions  in numerical homogenization to approximate the solution space of \eqref{eqn:scalar} appears to  have been first proposed in reservoir modeling in the 1980s \cite{BraWu09}, \cite{WhHo87} (in which a global scale-up method was introduced based on generic flow solutions (i.e., flows calculated from generic boundary conditions)). Its rigorous mathematical analysis has been done
only recently \cite{OwZh:2007a} and is based on the fact that solutions are $\H^2$-regular with respect to harmonic coordinates  (recall that they are $\H^1$-regular with respect to Euclidean coordinates). The main message here is that if the right hand side of \eqref{eqn:scalar} is in $L^2$, then solutions can be approximated at small scales (in $\H^1$-norm) by linear combinations of $d$ (linearly independent) particular solutions ($d$ being the dimension of the space). In that sense, harmonic coordinates are (simply) good candidates for being $d$ linearly independent particular solutions.
The idea of a global change of coordinates analogous to harmonic coordinates has been implemented numerically in order to up-scale porous media flows \cite{EfHo:2007, EfGiHouEw:2006, BraWu09}. We refer, in particular, to a recent review article \cite{BraWu09}   for an overview of some main challenges in reservoir modeling and a description of global scale-up strategies based on generic flows.

- In  \cite{EEngquist:2003, EnSou08}, the structure of the
medium is numerically decomposed into a micro-scale and a
macro-scale (meso-scale) and solutions of cell problems are computed
on the micro-scale, providing local homogenized matrices that are
transferred (up-scaled) to the macro-scale grid. See also  \cite{Abdulle:2004, Engquist:2011, Abdulle:2011}. We also refer to
 \cite{AnGlo06} for convergence results on the Heterogeneous
Multiscale Method in the framework of $G$ and $\Gamma$-convergence.

- More recent work includes an adaptive projection based method
\cite{NoPaPi:2008}, which is consistent with homogenization when there is scale
separation, leading to adaptive algorithms for solving problems with
no clear scale separation; finite difference approximations of fully nonlinear,
uniformly elliptic PDEs with Lipschitz continuous viscosity
solutions \cite{Caffarelli:2008} and operator splitting methods
\cite{Arbogast:2007, Arbogast:2006}.

- We  refer to \cite{Blanc:2007, Blanc:2006} (and references therein) for recent results on homogenization of  scalar divergence-form elliptic operators with stochastic coefficients. Here, the stochastic coefficients $a(x /\ve, \omega)$ are obtained from stochastic deformations (using random diffeomorphisms) of the periodic and stationary ergodic setting.

- We refer to \cite{BalJing10} for a  detailed study of the fluctuation error (error in
approximating the limiting probability distribution of the random part of the difference
between heterogeneous and homogenized solutions) of MsFEM and HMM in the context of stochastic homogenization and a rigorous assessment of the behavior of those multiscale algorithms in non-periodic situations. In particular, \cite{BalJing10} shows that savings in computational cost based on scale separation
may come at the price of amplifying the variances of these fluctuations.

\subsection{On Localization.}
The numerical complexity of numerical homogenization is related to the size of the support of the basis elements used to approximate the solution space of
\eqref{eqn:scalar}. The possibility to compute such bases on localized sub-domains of the global domain $\Omega$ without loss of accuracy is therefore a problem of practical importance.
We refer to  \cite{ChuHou09}, \cite{EfGaWu:2011}, \cite{BaLip10}, \cite{OwZh:2011} and \cite{MaPe:2012} for recent localization results for divergence-form elliptic PDEs.
The strategy of \cite{ChuHou09} is to construct triangulations and finite element bases that are adapted to the shape of high conductivity inclusions via coefficient dependent boundary conditions for the subgrid problems (assuming $a$ to be piecewise constant and the number of inclusions to be bounded).
The strategy of \cite{EfGaWu:2011} is to solve local eigenvalue problems, observing that only a few eigenvectors are sufficient to obtain a good pre-conditioner. Both \cite{ChuHou09} and \cite{EfGaWu:2011}  require specific assumptions on the morphology and number of inclusions. The idea of the strategy is to observe that if $a$ is piecewise constant and the number of inclusions is bounded, then $u$ is locally $\H^2$ away from the interfaces of the inclusions. The inclusions can then be taken care of by adapting the mesh and the boundary values of localized problems or by observing that those inclusions will affect only a finite number of eigenvectors.

The strategy of \cite{BaLip10} is to construct Generalized Finite Elements
 by partitioning the computational domain  into to a collection of preselected subsets and compute optimal local bases (using the concept of $n$-widths \cite{Pi:1985}) for the approximation of harmonic functions. Local bases are constructed by solving local eigenvalue problems (corresponding to computing eigenvectors of $P^* P$, where $P$ is the restriction of $a$-harmonic functions from $\omega^*$ onto $\omega\subset \omega^*$, $P^*$ is the adjoint of $P$, and $\omega$ is a sub-domain of $\Omega$ surrounded by a larger sub-domain $\omega^*$).
 The method proposed in \cite{BaLip10} achieves a near exponential convergence rate (in the number of pre-computed bases functions) for harmonic functions. Non-zero right hand sides (denoted by $g$) are then taken care of by finding (for each different $g$) particular solutions on preselected subsets with a constant  Neumann boundary condition (determined according to the consistency condition).
 The near exponential rate of convergence observed in \cite{BaLip10} is explained by the fact that the source space considered in \cite{BaLip10} is more regular than $L^2$ (see \cite{BaLip10}) and requires the computation of particular (local) solutions for each right hand side $g$ and each non-zero boundary condition, the basis obtained in \cite{BaLip10} is in fact adapted to $a$-harmonic functions away from the boundary).

The strategy of \cite{OwZh:2011} is to use the transfer property of the flux-norm introduced in \cite{BeOw:2010} (and the strong compactness of the solution space \cite{OwZh:2011}) to identify the global basis and then (inspired by an idea stemming from classical homogenization \cite{Gloria:2011, GloriaOtto:2012, PapanicolaouVaradhan:1981, Yurinski:1986}) localize the computation of the basis to sub-domains of size $\mathcal{O}(\sqrt{H}\ln (1/H))$ (without loss of accuracy) by replacing  the operator $-\diiv(a\nabla \cdot)$ with the operator $\frac{1}{T}\cdot-\diiv(a\nabla \cdot)$ in the computation of this basis with a well chosen $T$ balancing the created exponential decay in the Green's function with the deterioration of the transfer property.
We also refer to \cite{Sym12} for a generalization of the transfer property to Hilbert triple settings and
a different approach to localization for hyperbolic initial-boundary value problems based on the observation that weak solutions of symmetric hyperbolic systems propagate at finite speed.

In \cite{GrGrSa2012}, it is shown that for scalar elliptic problems with $L^\infty$ coefficients, there exists a local (generalized) finite element basis (AL basis) with $O\big((\log\frac{1}{H})^{d+1}\big)$ basis functions per nodal point, achieving the $O(H)$ convergence rate.
Although the construction provided in \cite{GrGrSa2012} involves solving global problems with specific right hand sides, its theoretical result supports
 the possibility of constructing localized bases.

The strategy of \cite{MaPe:2012} is to introduce a modified Cl\'{e}ment interpolation $\mathcal{I}_H$ (in which the interpolation of $u$ is not based on its nodal values but on  volume averages around nodes), and writing $V^f$ the kernel of $\mathcal{I}_H$ (i.e., the set of functions $u$ such that $\mathcal{I}_H u=0$), identify $V_H^{ms}$ as the orthogonal complement of $V^f$ with respect to the scalar product defined by $a(u,v)=\int_{\Omega} \nabla u \,a \nabla v$. The finite element basis $\phi_i$ is then identified by projecting nodal piecewise linear elements onto $\mathcal{I}_H$.  The work \cite{MaPe:2012}  shows that the computation of  $\phi_i$ can be localized to sub-domains of size $\mathcal{O}(H \ln \frac{1}{H})$ without loss of accuracy. The main differences between the strategy of \cite{MaPe:2012} and ours are as follows: (1) The identification of the space $V^f$  in \cite{MaPe:2012} and the related projections lead to saddle point problems
 (even after localization) whereas the method proposed here leads to localized elliptic linear systems (which ensures the fact that the computational cost of our method remains minimal and plays an important role in its stability, which we observe numerically, in the high contrast regime) (2) The interpolation of solutions in \cite{MaPe:2012} is not based on nodal values but on  volume averages around nodes whereas our interpolation of solutions is based on nodal values (which allows for an accurate estimation of solutions from point measurements and the introduction of homogenized linear systems defined on nodal values) (3) The accuracy of our method does not depend on the aspect ratios of the triangles of the mesh formed by the interpolation points.

We refer to subsection \ref{subsec:compcost} for a description of the computational cost of our method.

\subsection{Computational Cost}\label{subsec:compcost}

Write $d_{lf}$  the number nodes  of the fine mesh contained in the local domain $\Omega_i$ ($d_{lf}$ is the number of nodes/degrees of freedom of $\T_h\cap\Omega_i$). Write $d_f$  the number of interior nodes of the fine mesh $\T_h$ and $d_c$  the number of interior nodes of the coarse mesh  $\T_H$. Then the cost for solving local basis $\phi_i^{\loc}$, introduced in this paper, is
$O(d_{lf}d_c)$ (because each $\phi_i^{\loc}$ is the solution of a banded/nearly diagonal linear system), and the cost for constructing the global stiffness matrix is $O(d_{lf}(d_{lf}d_c/d_f)d_c)=O(d_{lf}^2d_c^2/d_f)$.
Notice that $d_{lf}d_c/d_f$ is the number of local domains which overlap with $\Omega_i$. Therefore, the overall cost is
$O(d_{lf}d_c+d_{lf}^2d_c^2/d_f)=O(d_{lf}d_c(1+d_{lf}d_c/d_f))$, and the dominant part is $O(d_{lf}^2d_c^2/d_f)$.


For quasi-uniform triangulations $\T_h$ and $\T_H$, we have $d_f\sim h^{-d}$ and $d_c\sim H^{-d}$. For the method introduced in this paper, to achieve
$O(H)$ accuracy in $\H^1$ norm, we need $d_{lf}\sim (H\log(1/H))^{-d}$. Therefore the overall cost of (pre-)computing of all the localized basis elements $\phi_i^{\loc}$ is $C_1 \sim(\frac{\log(1/H)}{h})^d$.




\subsection{Connections with classical homogenization theory}\label{subsec:connectionsclasshom}

We will now describe connections between  classical homogenization theory \cite{Mur78, BeLiPa78, JiKoOl91, Milton:2002}, our previous work \cite{BeOw:2010,OwZh:2011} and the present paper. Recall that in classical homogenization theory \cite{Mur78} one considers situations where the conductivity $a$ depends on a (small) parameter $\epsilon$ (i.e., $a:=a(x,\epsilon)$) and one is interested in approximating the solution $u_\epsilon$ of \eqref{eqn:scalar} in the limit where $\epsilon$ converges towards zero by a function $\hat{u}_\varepsilon$ that is easier to compute or simpler to represent.

In periodic homogenization \cite{BeLiPa78, Nguetseng:1989, JiKoOl91},  where the conductivity $a(x,\epsilon)$ is assumed to be of the form $a^{\textrm{per}}(\frac{x}{\epsilon})$ (and $a^{\textrm{per}}$ is assumed to be periodic in its argument), this approximation can be achieved by observing that $u_\epsilon$ converges, as $\epsilon \downarrow 0$, towards the solution $u_0$ of $-\diiv(\hat{a}\nabla u_0)=g$ (with zero Dirichlet boundary condition) and where $\hat{a}$ is a constant (effective, homogenized) conductivity determined by solving the so called (periodic) cell problems $\diiv\big(a^{\textrm{per}} (e_i+\nabla \chi_i)\big)=0$ over one period of $a^{\textrm{per}}$ ($i\in \{1,\ldots,d\}$, the vectors $e_i$ form an orthonormal basis of $\R^d$ and $a^{\textrm{per}}$ is invariant under translation by $e_i$). $u_0$ is much simpler to compute than $u_\epsilon$ because its variations are at the $\mathcal{O}(1)$-scale (referred to as the coarse scale) in contrast
to the oscillations of $u_\epsilon$ that are at the $\epsilon$-scale (referred to as the fine scale).

However the convergence of $u_\epsilon$ towards $u_0$ is limited to the $L^2$-norm, which does not control the derivatives.
  Since in most applications (e.g., computing stresses in elasticity and currents in electromagnetism) the approximation of derivatives is required, an important  {\it corrector problem} arose: find an approximation of the form
$\hat{u}_\varepsilon(x)=u_0(x) + \epsilon u_1(x)$ that approximates  the solution $u_\epsilon(x)$  of the original problem in $\H^1$ norm.  Solving this problem was an essential step in the development of periodic homogenization. It turned out that the corrector term  $\epsilon u_1(x)$    admits an explicit representation via the solution of  the cell problems $\chi_i(\frac{x}{\epsilon})$ and  the zero order approximation $u_0(x)$ of the form $\epsilon \sum_{i=1}^d \chi_i(\frac{x}{\epsilon}) \partial_i u_0(x)$. Thus the corrector term  contains both fine and coarse scales. The concept of the cell problem played the key role in periodic homogenization since it  is pivotal for both identifying the effective conductivity and an $\H^1$-norm approximation of the solution of the original problem. Namely  it allows  to construct  both the homogenized PDE and  the corrector term. The  advantage of computing  $\hat{u}_\epsilon$ rather than directly computing $u_\epsilon$ comes from the fact that  computing of   $\chi_i(y), i=1,...,d$  and  $u_0(x)$  is less expensive  than computing $u_\epsilon$.

The question of how far these ideas and concepts can be extended  beyond periodic homogenization lead to development stochastic homogenization, i.e. the extension of the theory to coefficients of the form $a=\{a_{ij}(\frac{x}{\epsilon}, \omega)\}$, where $\omega$ is an atom of a probability space and these coefficients are assumed to be stationary and ergodic.  While there is no periodic cell in this problem, the generalized  cell problem was successfully introduced in \cite{Kozlov:1979, PapanicolaouVaradhan:1981} and the homogenized problem was derived in terms of these problems analogously to the periodic case, that is the objective of identifying the homogenized PDE was achieved. However, the construction of the corrector terms  turned out to be a much harder problem and  by now it is only partially resolved  \cite{BourgeatPiat:1999, ArmSou2011, LioSou2003, KosRezVar2006,  BalJing10}.
It is also worth mentioning that stochastic homogenization, as it currently stands, is non-local in the sense that a deterministic and compactly supported perturbation of $a=\{a_{ij}(\frac{x}{\epsilon}, \omega)\}$ does not affect (in the limit where $\epsilon \downarrow 0$) the value of the effective conductivity nor the cell problems since these are defined taking averages  (perturbations that leave the statistical properties of the
random coefficient unchanged do not affect the homogenization result).

Many practical problems such as study of  oil reservoirs, biological tissues, etc. contain multiple scales that are not necessarily well-separated.   Again, one can ask  to which extent the ideas and concepts of classical   homogenization  can be used in such problems. These questions, which one can address   in the context of the model problem \eqref{eqn:scalar}, are challenging because if coefficients $a(x)$ are just $L^\infty$ functions,   then they have ``no structure'' (no translational invariance as in the periodic and  ergodic cases).

However the present work shows that, although there is no direct analog of the parameter $\epsilon$, one can still develop a generalization of most of the concepts and objectives of classical homogenization. First observe that the linear space spanned by the basis elements $\phi_i$ (or $\phi_i^{\loc}$ with support of size $\mathcal{O}(H \ln \frac{1}{H})$) constitutes a finite dimensional approximation of the solution space of \eqref{eqn:scalar} and this finite-dimensional approximation is made possible  by the (strong) compactness of the solution space (in $\H^1$-norm) when $g\in L^2(\Omega)$. Note that the requirement that $g$ belongs to $L^2(\Omega)$ is also present in classical homogenization and the notion of compactness is also present in the (weaker) form of compactness by compensation \cite{Mur78}.
Write $S$ the stiffness matrix associated with the elements $\phi_i^{\loc}$ defined by $S_{i,j}:=\int_{\Omega}\nabla \phi_i^{\loc} a \nabla \phi_j^{\loc}$ and observe that if $\varphi_i$ are piecewise linear nodal basis elements on a (coarse) mesh with vertices $(x_i)_{i\in \cN}$ and if $c^*$ the $N$-dimensional vector solution of the linear system $\sum_{j=1}^N S_{i,j} c_j=\int_{\Omega}\phi_i^{\loc} g$ then the distance between $u$ (the solution of \eqref{eqn:scalar}) and $\bar{u}:=\sum_{i=1}^N c_i^* \varphi_i$ is at most $C H \|g\|_{L^2(\Omega)}$ in $L^2$ norm (this can be obtained in a straightforward manner by writing $u^H$ the finite element solution
of \eqref{eqn:scalar} over $(\phi_i^{\loc})_{i\in \cN}$, i.e. $u^H:=\sum_{i=1}^N c_i^* \phi_i^{\loc}$, applying the triangle inequality to $u-\bar{u}=u-u^H+u^H-\bar{u}$, and Theorem 1.1 of \cite{Narcowich:2005} to $u^H-\bar{u}$). In that sense, $S$ can be viewed as a homogenized linear system, generalizing the concept of an homogenized PDE, and its entries can be seen as localized effective (edge) conductivities \cite{DesDonOw:2012} generalizing the concept of effective conductivity.
Furthermore noting that, although $\bar{u}$ does not constitute a good approximation of the derivatives of $u$,
the distance between $u$ and $u^H$ is at most $C H \|g\|_{L^2(\Omega)}$ in $\H^1$ norm, we deduce that the elements $\phi_i^{\loc}-\varphi_i$ can be interpreted as generalized correctors. That is we define the generalize corrector as an  $\H^1$ approximation of the exact solution which is constructed out of standard solutions that solve  {\it generalized cell problems} and therefore does not depend on the RHS of the original equation.  Several ways of introducing such problems were  proposed in  \cite{BeOw:2010,OwZh:2011} and the present paper.
Of course in the present work the notion of corrector is superfluous since Rough Polyharmonic Splines directly lead to an $\H^1$-accurate approximation of the solution space. It should also be noted that the elements $(\phi_i^{\loc})_{i\in \cN}$ provide an explicit representation  allowing for the identification of the (local) effect of a local perturbation of $a$ on the solution and on effective conductivities.

\paragraph{Acknowledgements.}
 The work of H. Owhadi is partially supported by the National Science Foundation under Award Number  CMMI-092600, the Department of Energy National Nuclear Security Administration under Award Number DE-FC52-08NA28613, the Air Force Office of Scientific Research under
 Award Number  FA9550-12-1-0389 and a contract from the DOE Exascale Co-Design Center for Materials in Extreme Environments.
 The work of L. Zhang is supported by the Young
Thousand Talents Program of China.  The work of L. Berlyand is supported  by DOE under Award Number DE-FG02-08ER25862.  H. Owhadi thanks M. Desbrun  and F. de Goes for stimulating discussions. H. Owhadi also thanks Fran\c{c}ois Murat for helpful discussions on Lemma \ref{lem:murat}.
L. Berlyand  thanks M. Potomkin for useful comments and suggestions. We also thank an anonymous referee for carefully reading the manuscript and detailed comments and suggestions.

\bibliographystyle{plain}
\bibliography{RPS}

\def\cprime{$'$}
\begin{thebibliography}{10}

\bibitem{Abdulle:2011}
Assyr Abdulle and Marcus~J. Grote.
\newblock Finite element heterogeneous multiscale method for the wave equation.
\newblock {\em Multiscale Model. Simul.}, 9(2):766--792, 2011.

\bibitem{Abdulle:2004}
Assyr Abdulle and Christoph Schwab.
\newblock Heterogeneous multiscale {FEM} for diffusion problems on rough
  surfaces.
\newblock {\em Multiscale Model. Simul.}, 3(1):195--220 (electronic), 2004/05.

\bibitem{AllBri05}
G.~Allaire and R.~Brizzi.
\newblock A multiscale finite element method for numerical homogenization.
\newblock {\em Multiscale Model. Simul.}, 4(3):790--812 (electronic), 2005.

\bibitem{Arbogast:2006}
T.~Arbogast and K.~J. Boyd.
\newblock Subgrid upscaling and mixed multiscale finite elements.
\newblock {\em SIAM J. Numer. Anal.}, 44(3):1150--1171 (electronic), 2006.

\bibitem{Arbogast:2007}
T.~Arbogast, C.-S. Huang, and S.-M. Yang.
\newblock Improved accuracy for alternating-direction methods for parabolic
  equations based on regular and mixed finite elements.
\newblock {\em Math. Models Methods Appl. Sci.}, 17(8):1279--1305, 2007.

\bibitem{ArmSou2011}
Scott~N. Armstrong and Panagiotis~E. Souganidis.
\newblock Stochastic homogenization of {H}amilton-{J}acobi and degenerate
  {B}ellman equations in unbounded environments.
\newblock {\em J. Math. Pures Appl. (9)}, 97(5):460--504, 2012.

\bibitem{Atteia:1970}
Marc Atteia.
\newblock Fonctions ``spline'' et noyaux reproduisants d'{A}ronszajn-{B}ergman.
\newblock {\em Rev. Fran\c caise Informat. Recherche Op\'erationnelle},
  4(R-3):31--43, 1970.

\bibitem{BaCaOs:1994}
I.~Babu{\v{s}}ka, G.~Caloz, and J.~E. Osborn.
\newblock Special finite element methods for a class of second order elliptic
  problems with rough coefficients.
\newblock {\em SIAM J. Numer. Anal.}, 31(4):945--981, 1994.

\bibitem{BaLip10}
I.~Babu{\v{s}}ka and R.~Lipton.
\newblock Optimal local approximation spaces for generalized finite element
  methods with application to multiscale problems.
\newblock {\em Multiscale Model. Simul.}, 9:373--406, 2011.

\bibitem{BaOs:1983}
I.~Babu{\v{s}}ka and J.~E. Osborn.
\newblock Generalized finite element methods: their performance and their
  relation to mixed methods.
\newblock {\em SIAM J. Numer. Anal.}, 20(3):510--536, 1983.

\bibitem{BaOs:2000}
I.~Babu{\v{s}}ka and J.~E. Osborn.
\newblock Can a finite element method perform arbitrarily badly?
\newblock {\em Math. Comp.}, 69(230):443--462, 2000.

\bibitem{BalJing10}
Guillaume Bal and Wenjia Jing.
\newblock Corrector theory for {MSFEM} and {HMM} in random media.
\newblock {\em Multiscale Model. Simul.}, 9(4):1549--1587, 2011.

\bibitem{BenarousOwhadi:2003}
G{\'e}rard Ben~Arous and Houman Owhadi.
\newblock Multiscale homogenization with bounded ratios and anomalous slow
  diffusion.
\newblock {\em Comm. Pure Appl. Math.}, 56(1):80--113, 2003.

\bibitem{BeLiPa78}
A.~Bensoussan, J.~L. Lions, and G.~Papanicolaou.
\newblock {\em Asymptotic analysis for periodic structure}.
\newblock North Holland, Amsterdam, 1978.

\bibitem{BeOw:2010}
L.~Berlyand and H.~Owhadi.
\newblock Flux norm approach to finite dimensional homogenization
  approximations with non-separated scales and high contrast.
\newblock {\em Archives for Rational Mechanics and Analysis}, 198(2):677--721,
  2010.

\bibitem{Blanc:2006}
X.~Blanc, C.~Le Bris, and P.-L. Lions.
\newblock Une variante de la th\'eorie de l'homog\'en\'eisation stochastique
  des op\'erateurs elliptiques.
\newblock {\em C. R. Math. Acad. Sci. Paris}, 343(11-12):717--724, 2006.

\bibitem{Blanc:2007}
X.~Blanc, C.~Le Bris, and P.-L. Lions.
\newblock Stochastic homogenization and random lattices.
\newblock {\em J. Math. Pures Appl. (9)}, 88(1):34--63, 2007.

\bibitem{BourgeatPiat:1999}
A.~Bourgeat and A.~Piatnitski.
\newblock Estimates in probability of the residual between the random and the
  homogenized solutions of one-dimensional second-order operator.
\newblock {\em Asymptot. Anal.}, 21(3-4):303--315, 1999.

\bibitem{BraWu09}
L.~V. Branets, S.~S. Ghai, L.~L., and X.-H. Wu.
\newblock Challenges and technologies in reservoir modeling.
\newblock {\em Commun. Comput. Phys.}, 6(1):1--23, 2009.

\bibitem{Brownlee:2004}
R.~A. Brownlee.
\newblock {\em Error estimates for interpolation of rough and smooth functions
  using radial basis functions}.
\newblock PhD thesis, University of Leicester, 2004.

\bibitem{Caffarelli:2008}
L.~A. Caffarelli and P.~E. Souganidis.
\newblock A rate of convergence for monotone finite difference approximations
  to fully nonlinear, uniformly elliptic {PDE}s.
\newblock {\em Comm. Pure Appl. Math.}, 61(1):1--17, 2008.

\bibitem{ChuHou09}
C.-C. Chu, I.~G. Graham, and T.~Y. Hou.
\newblock A new multiscale finite element method for high-contrast elliptic
  interface problems.
\newblock {\em Math. Comp.}, 79:1915--1955, 2010.

\bibitem{DesDonOw:2012}
M.~Desbrun, R.~Donaldson, and H.~Owhadi.
\newblock Modeling across scales: Discrete geometric structures in
  homogenization and inverse homogenization.
\newblock {\em Reviews of Nonlinear Dynamics and Complexity. Special issue on
  Multiscale Analysis and Nonlinear Dynamics.}, 2012.

\bibitem{Duchon:1976}
Jean Duchon.
\newblock Interpolation des fonctions de deux variables suivant le principe de
  la flexion des plaques minces.
\newblock {\em Rev. Francaise Automat. Informat. Recherche Operationnelle Ser.
  RAIRO Analyse Numerique}, 10(R-3):5--12, 1976.

\bibitem{Duchon:1977}
Jean Duchon.
\newblock Splines minimizing rotation-invariant semi-norms in {S}obolev spaces.
\newblock In {\em Constructive theory of functions of several variables
  ({P}roc. {C}onf., {M}ath. {R}es. {I}nst., {O}berwolfach, 1976)}, pages
  85--100. Lecture Notes in Math., Vol. 571. Springer, Berlin, 1977.

\bibitem{Duchon:1978}
Jean Duchon.
\newblock Sur l'erreur d'interpolation des fonctions de plusieurs variables par
  les {$D^{m}$}-splines.
\newblock {\em RAIRO Anal. Num\'er.}, 12(4):325--334, vi, 1978.

\bibitem{EEngquist:2003}
Weinan E and Bjorn Engquist.
\newblock The heterogeneous multiscale methods.
\newblock {\em Commun. Math. Sci.}, 1(1):87--132, 2003.

\bibitem{EfGaWu:2011}
Y.~Efendiev, J.~Galvis, and X.~Wu.
\newblock Multiscale finite element and domain decomposition methods for
  high-contrast problems using local spectral basis functions.
\newblock {\em Journal of Computational Physics}, 230(4):937--955, 2011.

\bibitem{EfGiHouEw:2006}
Y.~Efendiev, V.~Ginting, T.~Hou, and R.~Ewing.
\newblock Accurate multiscale finite element methods for two-phase flow
  simulations.
\newblock {\em J. Comput. Phys.}, 220(1):155--174, 2006.

\bibitem{EfHo:2007}
Y.~Efendiev and T.~Hou.
\newblock Multiscale finite element methods for porous media flows and their
  applications.
\newblock {\em Appl. Numer. Math.}, 57(5-7):577--596, 2007.

\bibitem{EfendievHou:2009}
Yalchin Efendiev and Thomas~Y. Hou.
\newblock {\em Multiscale finite element methods}, volume~4 of {\em Surveys and
  Tutorials in the Applied Mathematical Sciences}.
\newblock Springer, New York, 2009.
\newblock Theory and applications.

\bibitem{EkTe:1987}
I.~Ekeland and R.~Temam.
\newblock {\em Convex Analysis and Variational Problems}, volume~28 of {\em
  Classics in Applied Mathematics}.
\newblock Society for Industrial and Applied Mathematics, 1987.

\bibitem{EnSou08}
B.~Engquist and P.~E. Souganidis.
\newblock Asymptotic and numerical homogenization.
\newblock {\em Acta Numerica}, 17:147--190, 2008.

\bibitem{Engquist:2011}
Bj{\"o}rn Engquist, Henrik Holst, and Olof Runborg.
\newblock Multi-scale methods for wave propagation in heterogeneous media.
\newblock {\em Commun. Math. Sci.}, 9(1):33--56, 2011.

\bibitem{Giaquinta:1983}
Mariano Giaquinta.
\newblock {\em Multiple integrals in the calculus of variations and nonlinear
  elliptic systems}, volume 105 of {\em Annals of Mathematics Studies}.
\newblock Princeton University Press, Princeton, NJ, 1983.

\bibitem{GiTr:1983}
D.~Gilbarg and N.S. Trudinger.
\newblock {\em Elliptic Partial Differential Equations of Second Order}.
\newblock Springer-Verlag, 1983.
\newblock 2nd ed.

\bibitem{Gio75}
E.~De Giorgi.
\newblock Sulla convergenza di alcune successioni di integrali del tipo
  dell'aera.
\newblock {\em Rendi Conti di Mat.}, 8:277--294, 1975.

\bibitem{AnGlo06}
A.~Gloria.
\newblock Analytical framework for the numerical homogenization of elliptic
  monotone operators and quasiconvex energies.
\newblock {\em SIAM MMS}, 5(3):996--1043, 2006.

\bibitem{Gloria:2011}
Antoine Gloria.
\newblock Reduction of the resonance error---{P}art 1: {A}pproximation of
  homogenized coefficients.
\newblock {\em Math. Models Methods Appl. Sci.}, 21(8):1601--1630, 2011.

\bibitem{GloriaOtto:2012}
Antoine Gloria and Felix Otto.
\newblock An optimal error estimate in stochastic homogenization of discrete
  elliptic equations.
\newblock {\em Ann. Appl. Probab.}, 22(1):1--28, 2012.

\bibitem{GrGrSa2012}
L.~Grasedyck, I.~Greff, and S.~Sauter.
\newblock The al basis for the solution of elliptic problems in heterogeneous
  media.
\newblock {\em Multiscale Modeling \& Simulation}, 10(1):245--258, 2012.

\bibitem{GrWi:1982}
M.~Gr\"{u}ter and K.~Widman.
\newblock The green function for uniformly elliptic equations.
\newblock {\em Manuscripta Math.}, 37:303--342, 1982.

\bibitem{Harder:1972}
R.~L. Harder and R.~N. Desmarais.
\newblock Interpolation using surface splines.
\newblock {\em J. Aircraft}, 9:189--191, 1972.

\bibitem{HouWu:1999}
T.~Y. Hou, X.-H. Wu, and Z.~Cai.
\newblock Convergence of a multiscale finite element method for elliptic
  problems with rapidly oscillating coefficients.
\newblock {\em Math. Comp.}, 68(227):913--943, 1999.

\bibitem{HoWu:1997}
T.Y. Hou and X.H. Wu.
\newblock A multiscale finite element method for elliptic problems in composite
  materials and porous media.
\newblock {\em J. Comput. Phys.}, 134(1):169--189, 1997.

\bibitem{JiKoOl91}
V.~V. Jikov, S.~M. Kozlov, and O.~A. Oleinik.
\newblock {\em Homogenization of Differential Operators and Integral
  Functionals}.
\newblock Springer-Verlag, 1991.

\bibitem{KosRezVar2006}
Elena Kosygina, Fraydoun Rezakhanlou, and S.~R.~S. Varadhan.
\newblock Stochastic homogenization of {H}amilton-{J}acobi-{B}ellman equations.
\newblock {\em Comm. Pure Appl. Math.}, 59(10):1489--1521, 2006.

\bibitem{Kounchev:2005}
O.~Kounchev and H.~Render.
\newblock Polyharmonic splines on grids {$\Bbb Z\times a\Bbb Z^n$} and their
  limits.
\newblock {\em Math. Comp.}, 74(252):1831--1841 (electronic), 2005.

\bibitem{Kozlov:1979}
S.~M. Kozlov.
\newblock The averaging of random operators.
\newblock {\em Mat. Sb. (N.S.)}, 109(151)(2):188--202, 327, 1979.

\bibitem{LioSou2003}
Pierre-Louis Lions and Panagiotis~E. Souganidis.
\newblock Correctors for the homogenization of {H}amilton-{J}acobi equations in
  the stationary ergodic setting.
\newblock {\em Comm. Pure Appl. Math.}, 56(10):1501--1524, 2003.

\bibitem{Madych:1988}
W.~R. Madych and S.~A. Nelson.
\newblock Multivariate interpolation and conditionally positive definite
  functions.
\newblock {\em Approx. Theory Appl.}, 4(4):77--89, 1988.

\bibitem{Madych1:1990}
W.~R. Madych and S.~A. Nelson.
\newblock Multivariate interpolation and conditionally positive definite
  functions. {II}.
\newblock {\em Math. Comp.}, 54(189):211--230, 1990.

\bibitem{Madych2:1990}
W.~R. Madych and S.~A. Nelson.
\newblock Polyharmonic cardinal splines.
\newblock {\em J. Approx. Theory}, 60(2):141--156, 1990.

\bibitem{Madych3:1990}
W.~R. Madych and S.~A. Nelson.
\newblock Polyharmonic cardinal splines: a minimization property.
\newblock {\em J. Approx. Theory}, 63(3):303--320, 1990.

\bibitem{MaPe:2012}
A.~Malqvist and D.~Peterseim.
\newblock Localization of elliptic multiscale problems.
\newblock Technical report, arXiv:1110.0692, 2012.

\bibitem{Matveev:1992a}
O.~V. Matveev.
\newblock Some methods for the reconstruction of functions of {$n$} variables
  defined on chaotic grids.
\newblock {\em Dokl. Akad. Nauk}, 326(4):605--609, 1992.

\bibitem{Matveev:1992b}
O.~V. Matveev.
\newblock Spline interpolation of functions of several variables and bases in
  {S}obolev spaces.
\newblock {\em Trudy Mat. Inst. Steklov.}, 198:125--152, 1992.

\bibitem{Matveev:1994}
O.~V. Matveev.
\newblock Interpolation of functions on chaotic grids.
\newblock {\em Dokl. Akad. Nauk}, 339(5):594--597, 1994.

\bibitem{Matveev:1996}
O.~V. Matveev.
\newblock Methods for the approximate recovery of functions defined on chaotic
  grids.
\newblock {\em Izv. Ross. Akad. Nauk Ser. Mat.}, 60(5):111--156, 1996.

\bibitem{Matveev:1997}
O.~V. Matveev.
\newblock On a method for the interpolation of functions on chaotic grids.
\newblock {\em Mat. Zametki}, 62(3):404--417, 1997.

\bibitem{Melenk:2000}
J.~M. Melenk.
\newblock On {$n$}-widths for elliptic problems.
\newblock {\em J. Math. Anal. Appl.}, 247(1):272--289, 2000.

\bibitem{Milton:2002}
Graeme~W. Milton.
\newblock {\em The theory of composites}, volume~6 of {\em Cambridge Monographs
  on Applied and Computational Mathematics}.
\newblock Cambridge University Press, Cambridge, 2002.

\bibitem{MiYu06}
P.~Ming and X.~Yue.
\newblock Numerical methods for multiscale elliptic problems.
\newblock {\em J. of Comput. Phys.}, 214:421--445, 2006.

\bibitem{Moser:2012}
R.~Moser.
\newblock Theory of partial differential equations.
\newblock {\em MA6000A: Lecture Notes}, 2012.
\newblock \url{http://people.bath.ac.uk/rm257/MA6000A/notes.pdf}.

\bibitem{Mur78}
F.~Murat and L.~Tartar.
\newblock H-convergence.
\newblock {\em S\'{e}minaire d'Analyse Fonctionnelle et Num\'{e}rique de
  l'Universit\'{e} d'Alger}, 1978.

\bibitem{Narcowich:2005}
Francis~J. Narcowich, Joseph~D. Ward, and Holger Wendland.
\newblock Sobolev bounds on functions with scattered zeros, with applications
  to radial basis function surface fitting.
\newblock {\em Math. Comp.}, 74(250):743--763 (electronic), 2005.

\bibitem{Nguetseng:1989}
Gabriel Nguetseng.
\newblock A general convergence result for a functional related to the theory
  of homogenization.
\newblock {\em SIAM J. Math. Anal.}, 20(3):608--623, 1989.

\bibitem{NoPaPi:2008}
J.~Nolen, G.~Papanicolaou, and O.~Pironneau.
\newblock A framework for adaptive multiscale methods for elliptic problems.
\newblock {\em Multiscale Model. Simul.}, 7(1):171--196, 2008.

\bibitem{OwZh:2007a}
H.~Owhadi and L.~Zhang.
\newblock Metric-based upscaling.
\newblock {\em Comm. Pure Appl. Math.}, 60(5):675--723, 2007.

\bibitem{OwZh:2011}
H.~Owhadi and L.~Zhang.
\newblock Localized bases for finite dimensional homogenization approximations
  with non-separated scales and high-contrast.
\newblock {\em SIAM Multiscale Modeling \& Simulation}, 9:1373--1398, 2011.
\newblock arXiv:1011.0986.

\bibitem{Owhadi:2003}
Houman Owhadi.
\newblock Anomalous slow diffusion from perpetual homogenization.
\newblock {\em Ann. Probab.}, 31(4):1935--1969, 2003.

\bibitem{Owhadi:2004}
Houman Owhadi.
\newblock Averaging versus chaos in turbulent transport?
\newblock {\em Comm. Math. Phys.}, 247(3):553--599, 2004.

\bibitem{PapanicolaouVaradhan:1981}
G.~C. Papanicolaou and S.~R.~S. Varadhan.
\newblock Boundary value problems with rapidly oscillating random coefficients.
\newblock In {\em Random fields, {V}ol. {I}, {II} ({E}sztergom, 1979)},
  volume~27 of {\em Colloq. Math. Soc. J\'anos Bolyai}, pages 835--873.
  North-Holland, Amsterdam, 1981.

\bibitem{Pi:1985}
A.~Pinkus.
\newblock {\em N-Widths in Approximation Theory}.
\newblock Springer-Verlag, 1985.

\bibitem{Rabut:1990}
C.~Rabut.
\newblock {\em B-splines Polyarmoniques Cardinales: Interpolation,
  Quasi-interpolation, filtrage}.
\newblock These d’Etat, Universite de Toulouse, 1990.

\bibitem{Rabut1:1992}
Christophe Rabut.
\newblock Elementary {$m$}-harmonic cardinal {$B$}-splines.
\newblock {\em Numer. Algorithms}, 2(1):39--61, 1992.

\bibitem{Rabut2:1992}
Christophe Rabut.
\newblock High level {$m$}-harmonic cardinal {$B$}-splines.
\newblock {\em Numer. Algorithms}, 2(1):63--84, 1992.

\bibitem{Rossini:2009}
Milvia Rossini.
\newblock Detecting discontinuities in two-dimensional signals sampled on a
  grid.
\newblock {\em JNAIAM J. Numer. Anal. Ind. Appl. Math.}, 4(3-4):203--215, 2009.

\bibitem{Schoenberg:1973}
I.~J. Schoenberg.
\newblock {\em Cardinal spline interpolation}.
\newblock Society for Industrial and Applied Mathematics, Philadelphia, Pa.,
  1973.
\newblock Conference Board of the Mathematical Sciences Regional Conference
  Series in Applied Mathematics, No. 12.

\bibitem{Spagnolo:1968}
S.~Spagnolo.
\newblock Sulla convergenza di soluzioni di equazioni paraboliche ed
  ellittiche.
\newblock {\em Ann. Scuola Norm. Sup. Pisa (3) 22 (1968), 571-597; errata,
  ibid. (3)}, 22:673, 1968.

\bibitem{St:1965}
G.~Stampacchia.
\newblock Le probl\`{e}me de dirichlet pour les \'{e}quations elliptiques du
  second ordre \`{a} coefficients discontinus.
\newblock {\em Ann. Inst. Fourier (Grenoble)}, 15(1):189--258, 1965.

\bibitem{Stampaccia:1964}
Guido Stampacchia.
\newblock {\em \`{E}quations elliptiques du second ordre \`a coefficients
  discontinus}.
\newblock S\'eminaire Jean Leray no 3 (1963-1964). Numdam, 1964.

\bibitem{Sym12}
William Symes.
\newblock Transfer of approximation and numerical homogenization of hyperbolic
  boundary value problems with a continuum of scales.
\newblock {\em TR12-20 Rice Tech Report}, 2012.

\bibitem{Taylor:2012}
J.L. Taylor, S.~Kim, and R.M. Brown.
\newblock The green function for elliptic systems in two dimensions.
\newblock {\em arXiv:1205.1089}, 2012.

\bibitem{Vy:2008}
J.~Vybiral.
\newblock Widths of embeddings in function spaces.
\newblock {\em J. Complexity}, 24(4):545--570, 2008.

\bibitem{WhHo87}
C.~D. White and R.~N. Horne.
\newblock Computing absolute transmissibility in the presence of finescale
  heterogeneity.
\newblock {\em SPE Symposium on Reservoir Simulation}, page 16011, 1987.

\bibitem{Wu:2002}
X.~H. Wu, Y.~Efendiev, and T.~Y. Hou.
\newblock Analysis of upscaling absolute permeability.
\newblock {\em Discrete Contin. Dyn. Syst. Ser. B}, 2(2):185--204, 2002.

\bibitem{Yurinski:1986}
V.~V. Yurinski{\u\i}.
\newblock Averaging of symmetric diffusion in a random medium.
\newblock {\em Sibirsk. Mat. Zh.}, 27(4):167--180, 215, 1986.

\end{thebibliography}

\end{document}